\documentclass[12pt]{article}
 \usepackage[margin=0.97in]{geometry}

\usepackage{kpfonts,comment} 
\usepackage{enumitem} 
\usepackage{kz_style} 
\usepackage{xcolor}

\newcommand{\edit}[1]{{\color{black}#1}}
\newcommand{\issue}[1]{{\color{red}#1}}

\title{\Large Global Convergence  of Policy Gradient Methods\\ to (Almost) Locally Optimal Policies}

\begin{document}
\author{Kaiqing Zhang \and Alec Koppel \and Hao Zhu \and Tamer Ba\c{s}ar}
\date{January 3, 2019}
\maketitle

\begin{abstract}
Policy gradient (PG) methods are a widely used reinforcement learning methodology in many applications such as videogames, autonomous driving, and robotics. In spite of its empirical success, a rigorous understanding of the global convergence of PG methods is lacking in the literature. In this work, we close the gap by viewing PG methods from a nonconvex optimization perspective. In particular,  we propose a new variant of PG methods for infinite-horizon problems that uses a random rollout horizon for the Monte-Carlo estimation of the policy gradient. This method then yields an unbiased estimate of the policy gradient with bounded variance, which enables the  tools from nonconvex optimization to be applied to establish global convergence. Employing this perspective, we first recover the convergence results with rates to the stationary-point policies in the literature. More interestingly, motivated by advances  in nonconvex optimization, we modify the proposed PG method by introducing periodically enlarged stepsizes. The modified algorithm is shown to escape saddle points under mild assumptions on the reward and the policy parameterization. Under a further strict saddle points assumption, this result establishes convergence to essentially locally-optimal policies of the underlying  problem, and thus bridges the gap in existing literature on the convergence of PG methods.  Results from experiments on the inverted pendulum are then provided to corroborate our theory, namely, by slightly reshaping the reward function to satisfy our assumption, unfavorable saddle points can be avoided and better limit points can be attained. Intriguingly, this empirical finding justifies the benefit of reward-reshaping from a nonconvex optimization perspective.

\end{abstract}

%!TEX root = Policy_Grad.tex
%%%%%%%%%%%%%%%%%%%%%%%%%%%%%%%%%%%%%%%%%%%%%%%%%%%%%%%%%%%%%%%%%%%%%%%%%%%%%
%%%%%%%%%%%%%%%%%%%%%%%%%%%%%%%%%%%%%%%%%%%%%%%%%%%%%%%%%%%%%%%%%%%%%%%%%%%%%%%%%%%%%%%%%%%%%%%% S  E  C T  I  O  N%%%%%%%%%%%%%%%%%%%%%%%%%%%%%%%%%%%%%%%%%%%%%%%%%%%%%%%%%%%%%%%%%%%%%%%%%%%%%%%%%%%%%%%%%%%%%%%%%%%%%%%%%%%%%%%%%%%%%%%%%%%%%%%%%%%%%%%%%%%%%%%%%%%%%%%%%%%%%%%%%%%%%%%%%%%%%%%%%%%%%%%%%%%%%%%%%%%%%%%%%%%
\section{Introduction}  

In reinforcement learning (RL) \cite{sutton2017reinforcement,bertsekas2005dynamic}, an autonomous agent moves through a state space  and seeks to learn a policy  which maps states to a probability distribution over actions to maximize a long-term accumulation of rewards. When the agent selects a given action at a particular state, a reward is revealed and a random transition to a new state occurs according to a probability density that only depends on the current state and action, i.e., state transitions are Markovian. {This evolution process is usually modeled as a Markov decision process (MDP). } Under this setting, the agent must evaluate the merit of different actions by interacting with the environment. 
Two dominant approaches to reinforcement learning have emerged: those based on optimizing {the accumulated reward  directly from the policy space}, referred to as  ``direct policy search'', and those based on {finding the value function by} 
solving the Bellman fixed point equations \cite{bellman57a}. The goal of this work is to rigorously understand the former approach of direct policy search, specifically policy gradient (PG) methods \cite{sutton2000policy}. Policy search has gained traction recently, thanks to  its ability to scale gracefully to large and even continuous spaces  \cite{silver2014deterministic,schulman2015trust} and to incorporate deep networks {as function approximators  }\cite{lillicrap2015continuous,mnih2016asynchronous}.  

Despite the increasing prevalence of policy gradient methods, their global convergence in the infinite-horizon discounted setting,  which is conventional in dynamic programming \cite{bertsekas2005dynamic},  is not yet well understood. This gap  stems firstly  from the fact that obtaining unbiased estimates of the policy gradient through sampling is often elusive. 
Specifically, following the Policy Gradient Theorem \cite{sutton2000policy},
%yields the fact that the policy gradient contains two distinct factors: (1) the score function, i.e., the gradient of the logarithm of the policy, and (2) the action-value (or $Q$) function induced by a particular policy, both evaluated at a state-action pair drawn from an ergodic distribution of the Markov chain under the policy. 
obtaining an unbiased estimate of the policy gradient requires two significant conditions to hold: (i) the state-action pair is drawn from the discounted state-action occupancy measure  of the Markov chain under the policy; (ii) the estimate of the action-value (or $Q$) function induced by the  policy is unbiased.   
This gap also results from the fact that  the value function to be maximized in RL is in general \emph{nonconvex} with respect to the policy parameter \cite{fazel2018global,zhang2019policy,bhandari2019global,liu2019neural,agarwal2019optimality}. 
In the same vein as our work, there is a surging interest in studying the global convergence of PG methods, see the recent work  \cite{fazel2018global,zhang2019policy,bu2019lqr}, and concurrent work \cite{bhandari2019global,liu2019neural,agarwal2019optimality}. In particular, orthogonal to our work, these work considered convergence to the \emph{global optimum} in several \emph{special} 	RL settings: \cite{fazel2018global,zhang2019policy,bu2019lqr} considered the linear quadratic setting, \cite{bhandari2019global,agarwal2019optimality} considered the tabular setting, \cite{liu2019neural,wang2019neural} focused on the setting with overparameterized neural networks for  
function approximation, and \cite{agarwal2019optimality} also considered the setting when the optimality  gap of using certain policy class can be quantified.  
%However, it  is worth mentioning that orthogonal to our work, these work considered convergence to the \emph{global optimum} in several \emph{special} 	RL settings, e.g., linear quadratic setting \cite{fazel2018global,bu2019lqr}, tabular setting \cite{bhandari2019global,agarwal2019optimality}, the setting with overparameterized neural networks for  
%function approximation \cite{liu2019neural}, or the setting where the optimality  gap of using certain policy class can be quantified \cite{agarwal2019optimality}.  
%although the recent work \cite{fazel2018global}, and the concurrent work \cite{bhandari2019global,liu2019neural,agarwal2019optimality} have identified several special RL problems with  certain policy parameterization can achieve global optima, 
In contrast, our focus is on the case  where the  nonconvexity might be general, so that solving the problem can be  NP-hard.    

% (``the double sampling problem"). 
%While the former factor may be estimated using sampling, the later depends on the policy parameters, which change across iterations. 

%\kznote{Probably we may want to not use the term "double sampling". In RL, double sampling is reserved for the case when expected value of the product does not equal the product of two expected values in policy evaluation. In policy gradient method, what we have is just a "double-loop" expectation. }

{
When one restricts the  focus to \emph{episodic} reinforcement learning, 
Monte-Carlo rollout may be used to obtain \emph{unbiased} estimates of the Q-function. In particular, the  rollout   simulates the MDP under certain policy up to a finite time horizon, and then collects  the rewards and state-action histories along the trajectory. }
 However, this finite-horizon rollout, {though generally used in practice,} is known to introduce bias in estimating an \emph{infinite-horizon}  discounted value function. 
Such  a bias in estimating the policy gradient for infinite-horizon problems has been identified in the earlier  work \cite{baxter2001infinite,bartlett2011experiments}, both analytically and empirically.  
%This bias issue has led most analyses to focus on \emph{asymptotic} behavior by invoking connections to dynamical systems \cite{borkar2000ode,borkar2008stochastic}, while some recent efforts towards \emph{finite sample analysis} of the infinite-horizon case have appeared recently for linear quadratic problems only, a special case of general MDPs.  \cite{fazel2018global}. 
%{In fact, such a bias in estimating the policy gradient for infinite-horizon problems has been identified in the earlier  work \cite{baxter2001infinite,bartlett2011experiments}, both analytically and empirically.  }
To address this bias issue, we employ in this work  random geometric time rollout horizons, a technique first proposed in \cite{santi2018stochastic}. This rollout procedure  allows us to obtain unbiased estimates of the $Q$ function, using only rollouts of finite horizons. Moreover, the random rollout horizon also creates an unbiased sampling of the state-action pair from the discounted  occupancy measure \cite{sutton2000policy}.  With these two challenges addressed,  the policy gradient can be estimated unbiasedly. Consequently,    the policy gradient methods  can be more naturally connected to  the classical stochastic programming algorithms  \cite{shapiro2009lectures}, {where the unbiasedness of the stochastic gradient   is a critical assumption}. We refer to our algorithm as \emph{random-horizon}  policy gradient (RPG), to emphasize that the finite horizon of the Monte-Carlo rollout is random. 
%We refer to the resulting variant of policy gradient method as \emph{R3INFORCE}, since it is established upon the Monte-Carlo rollout of trajectories  as REINFORCE  \cite{sutton2000policy}, but with \emph{random} rollout horizons.

%\kznote{I think we should not say "nothing more than", since our algorithm, although is based on Monte-Carlo rollout, is still different from the actual REINFORCE. In particular, we estimate $Q$-value and score function for a specific $(s,a)$, while REINFORCE evaluate the rollout return and the score function "along the whole trajectory".}

Leveraging this connection, we are able to address a noticeably open issue in policy {gradient methods}: a technical understanding of the effect of the policy parameterization on both the limiting and finite-iteration algorithm behaviors. In particular, it is well known  in nonconvex optimization that with only first-order information and no additional hypothesis, convergence to a stationary point {with zero gradient-norm }is the best one may hope to achieve \cite{wright1999numerical}.  
Indeed, this is the type of points that most current PG methods are guaranteed to converge to, as pointed out by  \cite{agarwal2019optimality}.  
{However, in some asymptotic analyses for policy gradient methods with function approximation \cite{pirotta2015policy}, or their  variant, actor-critic algorithms \cite{bhatnagar2008incremental,bhatnagar2009natural,bhatnagar2010actor,chow2017risk}, it was claimed that the limit points of the algorithms starting from any initialization  constitute the \emph{locally-optimal policies}, i.e., the algorithms enjoy global  convergence  to the local-optima.  
However, by the theory of stochastic approximation \cite{borkar2008stochastic}, such a claim can only be made locally, i.e., 
the local-optimality can only  be obtained if the algorithm starts around a local minima,   under the assumption that a \emph{strict} Lyapunov function exists. Therefore, \emph{global}  convergence of PG methods to the \emph{actual locally-optimal }policies, though claimed in words in some literature, is still an open question.  
Another line of  theoretical studies  of policy gradient methods only focuses on showing the one-step policy improvement \cite{pirotta2013adaptive,pirotta2015policy,papini2017adaptive}, by choosing appropriate stepsizes and/or batch data sizes. Such one-step result still does not imply any global convergence result.    
In summary, the misuse of the term \emph{locally-optimal policy} and the lack of studying global convergence property of PG methods  motivate us to further investigate this problem from a nonconvex optimization perspective. Thanks to  the analytical tools from optimization, we are able to first recover the asymptotic convergence, and then provide the convergence rate, to \emph{stationary-point} policies. 
}

Encouraged by this connection between nonconvex optimization and policy search, we then  tackle a related question: what implications do recent algorithms   that can escape saddle points for nonconvex problems (\cite{jin2017escape,daneshmand2018escaping}) have on policy gradient methods in RL? To answer this question,   we identify several  structural properties of  RL problems   that can be exploited to mitigate the underlying  nonconvexity, {which rely on  some  key assumptions on the } policy parameterization and reward. Specifically, {the reward needs to be bounded and {either strictly positive or negative}, and the policy parameterization need to be  \emph{regular}, i.e., its Fisher  information {matrix} is positive definite (a conventional assumption in RL \cite{kakade2002natural}).} {Under these mild  conditions,} we can establish  that policy gradient methods can escape saddle points and converge to approximate  \emph{second-order} stationary points  {with high probability},  when a \emph{periodically enlarged stepsize} strategy is employed. We {refer to the resulting method as Modified RPG (MRPG)}. 
{Nevertheless, the strict positivity/{negativity} of reward function may amplify the variance of the  gradient estimate, compared to the setting  that has reward values with both signs but of smaller magnitude. This increased variance  can be alleviated by introducing a \emph{baseline} in the gradient estimate, as advocated  by  existing  work  \cite{greensmith2004variance,peters2006policy,bhatnagar2008incremental}. Therefore, we propose two further modified updates that include the baselines, both  shown to converge to approximate second-order stationary points  as well.

\vspace{6pt}
\noindent \textbf{Main Contribution:} The main contribution of the present work is three-fold: i) we propose a series of random-horizon PG methods that unbiasedly estimate the true policy gradient for \emph{infinite-horizon discounted}  MDPs, which facilitates the use of analytical tools from nonconvex optimization to establish their convergence to \emph{stationary-point} policies; ii) by virtue of such a  connection of PG methods and nonconvex optimization, we propose   modified RPG methods with periodically enlarged stepsizes, with guaranteed convergence  to actual \emph{locally-optimal policies} under mild conditions on the reward functions and parametrization  of the policies; iii) we connect  the condition on the reward function to the  reward-reshaping technique advocated in empirical RL studies, justifying its benefit, both analytically and empirically, from a nonconvex optimization perspective. Additionally, we believe such a perspective opens the door to exploiting more advancements in nonconvex optimization to improve the convergence properties of policy gradient methods in RL.  
}
%\issue{Check this out, may be further simplified.}
%\issue{The strict positivity/\edit{negativity} requirement on the reward for this setting further motivates use of a modified update where we substitute the $Q$ function by a $Q$ function minus a baseline, for the purpose of variance reduction. The modified update that uses baselines is also shown to converge to second-order stationary points. }

%\kznote{What did you mean by "the same convergence guarantees"? This is convergence to 2nd-order stationary point, right? Also, should we have a name for this algorithm, say, Modified R3INFORCE? Currently in the simulation figures, we can the algorithms \emph{RPG and MRPG}.}

%

%
The rest of the paper is organized as follows. In  \S\ref{sec:prob}, we clarify the problem setting of reinforcement learning and the technicalities of Markov Decision Processes. In  \S\ref{sec:alg} we develop the policy gradient method using random geometric Monte-Carlo rollout horizons,  i.e., the RPG method. Further, we establish both its limiting (Theorem \ref{thm:alg1_as_conv}) and finite-sample  (Theorem \ref{thm:alg1_conv_rate} and Corollary \ref{coro:alg1_conv_rate_const_step}) behaviors under standard conditions. We note that Corollary \ref{coro:alg1_conv_rate_const_step} provides  one of the first constant learning rate results in reinforcement learning. In { \S\ref{sec:conv_second_stationary_point}}, we  focus on problems with positive bounded rewards and policies whose parameterizations are \emph{regular}, and propose a variant of policy gradient method that employs a periodically enlarged stepsize scheme.
 The salient feature of this modified algorithm is that it is able to escape saddle points, an undesirable subset of stationary points, and converge to approximate second-order stationary points (Theorem \ref{thm:conv_MRPG}). Numerical experiments in  \S\ref{sec:simulations} corroborate our main findings: for Algorithm \ref{alg:RPG}, the use of random rollout horizons avoids stochastic gradient bias and hence exhibits reliable convergence that matches the theoretically established rates; moreover, for the {modified RPG algorithm}, use of periodically enlarged stepsizes makes it possible to escape   from undesirable saddle points and yields   better limiting solutions. All proofs, which constitute an integral part of the paper, are relegated to nine appendices at the end of the paper, so as not to disrupt the flow of the presentation of the main results. 
 
% \kznote{We may need to discuss and agree on some of the terms used in the Introduction, as well as throughout the paper.}

\vspace{6pt}
\noindent \textbf{Notations:} We denote the probability distribution over the space $\cS$ by $\cP(\cS)$, and the  set of integers $\{1,\cdots,N\}$ by $[N]$. We  use $\RR$ to denote the set of real numbers, and $\EE$ to denote the expectation operator. We let $\|\cdot\|$  denote the $2$-norm of a vector in $\RR^d$, or the  spectral norm of a matrix in $\RR^{d\times d}$. We  use $|\cA|$ to denote the cardinality of a finite set $\cA$, or the area of a  region $\cA$, i.e., $|\cA|=\int_{\cA}da$. For any matrix $A\in\RR^{d\times d}$, we use $A\succ 0$ and $A\succeq 0$ to denote that $A$ is positive definite and positive semi-definite, respectively.  We  use $\lambda_{\min}(A)$ and $\lambda_{\max}(A)$ to denote, respectively, the smallest and largest eigenvalues of some square symmetric matrix  $A$, respectively. We use $\EE_X$ or $\EE_{X\sim f(x)}$ to denote the expectation with respect to random variable $X$. Otherwise specified, we   use $\EE$ to denote the full expectation with respect to all random variables.

%!TEX root = Policy_Grad.tex
%%%%%%%%%%%%%%%%%%%%%%%%%%%%%%%%%%%%%%%%%%%%%%%%%%%%%%%%%%%%%%%%%%%%%%%%%%%%%
%%%%%%%%%%%%%%%%%%%%%%%%%%%%%%%%%%%%%%%%%%%%%%%%%%%%%%%%%%%%%%%%%%%%%%%%%%%%%%%%%%%%%%%%%%%%%%%% S  E  C T  I  O  N%%%%%%%%%%%%%%%%%%%%%%%%%%%%%%%%%%%%%%%%%%%%%%%%%%%%%%%%%%%%%%%%%%%%%%%%%%%%%%%%%%%%%%%%%%%%%%%%%%%%%%%%%%%%%%%%%%%%%%%%%%%%%%%%%%%%%%%%%%%%%%%%%%%%%%%%%%%%%%%%%%%%%%%%%%%%%%%%%%%%%%%%%%%%%%%%%%%%%%%%%%%
\section{Problem Formulation}\label{sec:prob}
%We first introduce the formulation of the reinforcement learning (RL) problem and the policy gradient algorithm.
%
 In reinforcement learning, an autonomous agent moves through a state space $\cS$ and takes actions that belong to some action space  $\cA$. Here the spaces $\cS$ and $\cA$ are allowed to be either finite sets, or compact real vector spaces, i.e.,  $\cS\subseteq   \mathbb{R}^q$ and $\cA\subseteq  \mathbb{R}^p$.  
An action at the state causes a transition to the next state, where the transition mapping  that depends on the current state and action; every such transition generates  a reward  revealed by the environment. The goal is for the agent  to accumulate as much reward as possible in the long term. 
This situation can be formalized as a Markov decision process (MDP) characterized by a tuple $(\cS,\cA,\PP,R,\gamma)$ with Markov kernel $\PP(s'\given s,a):\cS\times\cA\to \cP(\cS)$ that determines the transition probability from $(s,a)$ to state ${s}'$. 
$\gamma\in(0,1)$ is the discount factor. $R(\cdot,\cdot)$ is the reward that is a function\footnote{$R(s_t,a_t)$ may be a random variable given $(s_t,a_t)$. Here without loss of generality, we assume that it is deterministic for simplicity.} of $s$ and $a$. 
% that parameterizes the value of a given sequence of actions, to be  defined  shortly. 

%\textcolor{blue}{we should use bold for states and actions since they are vectors. This will also help is distinguish vectors from scalars. Also, $\theta$ should be $\mathbf{\theta}$. }
%\textcolor{red}{Let's discuss on this. It's not necessarily a vector. Due to function approximation, actually both of them could be continuous.}

At each time $t$, the agent executes an action $a_t\in\cA$ given the current state $s_t\in\cS$, following a possibly stochastic policy $\pi:\cS\to \cP(\cA)$, i.e., $a_t\sim \pi(\cdot\given s_t)$. 
Then, given the state-action pair $(s_t,a_t)$, the agent observes a reward $r_t=R(s_t,a_t)$. 
Thus, under any policy $\pi$ that maps states to actions, one can define the value function $V_{\pi}:\cS\to\RR$ as 
\$
V_{\pi}(s)=\EE_{a_t\sim \pi(\cdot\given s_t),s_{t+1}\sim \PP(\cdot\given s_t,a_t)}\bigg(\sum_{t=0}^\infty \gamma^t r_t\bigggiven s_0=s\bigg),
\$
which quantifies the long term expected accumulation of rewards discounted by $\gamma$. We can further define the value $V_{\pi}:\cS\times\cA\to\RR$ conditioned on a given initial action as the action-value, or Q-function as $Q_{\pi}(s,a)=\EE\big(\sum_{t=0}^\infty \gamma^t r_t\biggiven s_0=s,a_0=a\big)$.
We also define $A_{\pi}(s,a)=Q_{\pi}(s,a)-V_{\pi}(s)$   for any $s,a$ to be the \emph{advantage function}. 
Given any initial state $s_0$, 
the  goal  is to find the optimal policy $\pi$ that maximizes  the long-term return $V_{\pi}(s_0)$, i.e., to solve the following optimization problem 
\#\label{equ:max_goal}
\max_{\pi\in\Pi}~~V_{\pi}(s_0),
%J(\pi)}.
\#
when the model, i.e.,  the transition probability $\PP$ and the reward function $R$, is unknown to the agent. 
In this work, we investigate policy search methods to solve \eqref{equ:max_goal}. In general, we must search over an arbitrarily complicated function class $\Pi$ which may include those which are unbounded and discontinuous. To mitigate this issue, we propose to \emph{parameterize}  policies $\pi$ in $\Pi$ by a vector $\theta\in\RR^d$, i.e., $\pi=\pi_{\theta}$, which gives rise to RL algorithms called \emph{policy gradient methods} \cite{konda2000actor,bhatnagar2009natural,castro2010convergent}. 
With this parameterization, we may reduce a search over arbitrarily complicated function class $\Pi$ in \eqref{equ:max_goal} to one over the Euclidean space $\mathbb{R}^d$. Nonparametric parameterizations are also possible \cite{koppel2017pkgtd,koppel2018kqlearning}, but here we fix the parameterization in order to simplify exposition.
For notational convenience, we define $J(\theta):=V_{\pi_{\theta}}(s_0)$, then the vector-valued optimization problem can be written as  
\#\label{equ:max_goal_para}
\max_{\theta\in\mathbb{R}^d}~~J(\theta). 
\#

Generally, the value function is nonconvex with respect to the parameter $\theta$, meaning that obtaining a globally optimal solution to \eqref{equ:max_goal_para} is 
NP-hard, 
 unless in several special RL settings that have been identified very recently \cite{fazel2018global,bhandari2019global}. 
% the problem has  additional structured properties,  as in  phase retrieval  \cite{sun2016geometric},  matrix factorization \cite{li2016symmetry}, and tensor decomposition \cite{ge2015escaping}. 
 %\textcolor{green}{[some reference about matrix completion/tensor decomposition non-convexity not being a problem]}. 
  In fact, the limit point of most gradient-based  methods to nonconvex optimization is a stationary solution, which could either be a saddle point or a local optimum. Usually the local optima achieve reasonably good performance, in some cases comparable to the global optima, whereas the saddle points are undesirable and can stall training procedures. Therefore, it is beneficial to design methods that may escape saddle points -- see recent efforts on escaping saddle points  with first-order methods, e.g., perturbed gradient descent \cite{ge2015escaping,jin2017escape,daneshmand2018escaping}, and second-order methods \cite{dauphin2014identifying,xu2017newton}.
% \textcolor{green}{Here we need a reference to some old paper about nonconvex analysis}. 
 
 Our goal in this work is to develop stochastic gradient methods to maximize $J(\theta)$ and rigorously understand the interplay between its limiting properties and the necessity of augmenting the algorithmic update, reward function, and policy parameterization, all toward escaping undesirable limit points. This issue was first observed and addressed  in \cite{konda1999actor}  by adding random perturbations in the reinforcement learning update (which may amplify variance), based on the asymptotic convergence results in \cite{pemantle1990nonconvergence}. Here we provide a modern perspective and incorporate the latest developments in nonconvex optimization.

%%%%%%%%%%%%%%%%%%%%%%%%%%%%%%%%%%%%%%%%%%%%%%%%%%%%%%%%%%%%%%%%%%%%%%%%%%%%%
%%%%%%%%%%%%%%%%%%%%%%%%%%%%%%%%%%%%%%%%%%%%%%%%%%%%%%%%%%%%%%%%%%%%%%%%%%%%%%%%%%%%%%%%%%%%%%%% S  E  C T  I  O  N%%%%%%%%%%%%%%%%%%%%%%%%%%%%%%%%%%%%%%%%%%%%%%%%%%%%%%%%%%%%%%%%%%%%%%%%%%%%%%%%%%%%%%%%%%%%%%%%%%%%%%%%%%%%%%%%%%%%%%%%%%%%%%%%%%%%%%%%%%%%%%%%%%%%%%%%%%%%%%%%%%%%%%%%%%%%%%%%%%%%%%%%%%%%%%%%%%%%%%%%%%%
\section{Policy Gradient Methods}\label{sec:alg}
%Note that the optimization problem \eqref{equ:max_goal} over the functional space is difficult to solve in general. Thus, in policy-optimization-based reinforcement learning algorithms, the policy is usually parameterized by a vector $\theta\in\Theta\subseteq\RR^d$, i.e., $\pi=\pi_{\theta}$. 
%With general parameterization of $\pi_{\theta}$, the optimization of $J(\theta)$ over $\theta$ can be highly nonconvex. Therefore, locating the local minimizer is the best one can hope for in this case. 

In this section, we connect stochastic gradient ascent, as it is called in stochastic  optimization, with the policy gradient method, a flavor of direct policy search, in reinforcement learning.
We start with the   following standard assumption on the regularity of the MDP problem and the smoothness of the parameterized policy $\pi_\theta$. 

%%%%%%%%%%%%%%%%%%%%%%%%%%%%%%%%%%%%%%%%%%%%%%%%%%%%%%%%%%%%%%%%%%%%%%%%%%%%%
%%%%%%%%%%%%%%%%%%%%%%%%%%%%%%%%%%%%%%%%%%%%%%%%%%%%%%%%%%%%%%%%%%%%%%%%%%%%%%%%%%%%%%%%%%%%%%%% A  S  S  U  M  P  T  I  O  N  %%%%%%%%%%%%%%%%%%%%%%%%%%%%%%%%%%%%%%%%%%%%%%%%%%%%%%%%%%%%%%%%%%%%%%%%%%%%%%%%%%%%%%%%%%%%%%%%%%%%%%%%%%%%%%%%%%%%%%%%%%%%%%%%%%%%%%%%%%%%%%%%%%%%%%%%%%%%%%%%%%%%%%%%%%%%%%%%%%%%%%%%%%%%%%%%%%%%%%%%%%%
\begin{assumption}\label{assum:regularity}
	Suppose the reward function $R$ and the parameterized policy $\pi_\theta$ satisfy the following conditions:
	\vspace{-6pt}
\begin{enumerate}[label=(\roman*)]
		\item The absolute value of the reward $R$ is  uniformly bounded, say   by $U_{R}$, i.e., $|R(s,a)|\in[0,U_{R}]$ for any $(s,a)\in\cS\times\cA$. \label{as:bounded_reward}
%		\item The set of policy parameters $\Theta\subseteq\RR^d$ is compact.
		\item The policy $\pi_{\theta}$ is  differentiable with respect to $\theta$, and  $\nabla\log\pi_{\theta}(a\given s)$, known as  		the \emph{score function}  		corresponding to the distribution $\pi_{\theta}(\cdot\given s)$,  exists. Moreover, it is $L_\Theta$-Lipschitz and has bounded norm for any  $(s,a)\in\cS\times\cA$,  
%		the gradient $\nabla \pi_{\theta}$ is $L$-Lipschitz, i.e., 
%\small
		\#
		&\|\nabla \log\pi_{\theta^1}(a\given s)-\nabla \log\pi_{\theta^2}(a\given s)\|\leq L_\Theta\cdot \|\theta^1-\theta^2\|,\text{~~for any~~} \theta^1,\theta^2,\label{equ:assum_L_Lip}\\
		&\|\nabla\log\pi_{\theta}(a\given s)\|\leq B_{\Theta},\text{~~for some constant $B_\Theta$~~for any~~} \theta.\label{equ:assum_score_bnded}
		\#
		\label{as:bounded_policy}
		\normalsize
		for some constant $B_\Theta>0$. 
%		\item The norm of the score function $\nabla\log\pi_{\theta}(a\given s)$ is uniformly bounded by $B_{\Theta}<\infty$, i.e., for any $(s,a)\in\cS\times\cA, \theta\in\RR^d$,  
%		\#\label{equ:assum_score_bnded}
%		
%		\#
	\end{enumerate}
\end{assumption}

%
%\textcolor{green}{Can you modify the statement of Assumption \ref{assum:regularity} so that it is broken into a list where each technical condition is given it's own display equation? This will make it easier to refer to specific statements during the proof, and also to clarify the technical setting we consider.}
%
Note that the  boundedness of the reward function in Assumption\ref{assum:regularity}\ref{as:bounded_reward}  is standard in the literature of policy gradient/actor-critic algorithms \cite{bhatnagar2008incremental,bhatnagar2009natural,castro2010convergent,zhang2018fully,zhang18cdc}. 
The uniform boundedness of $R$ also implies that the absolute value of  the Q-function is upper bounded by $U_{R}/(1-\gamma)$, since by definition 
\$
|Q_{\pi_\theta}(s,a)|\leq 
%\EE\bigg(\sum_{t=0}^\infty \gamma^t \cdot|r_t|\biggiven s_0=s,a_0=a\bigg)\leq 
\sum_{t=0}^\infty \gamma^t \cdot U_{R}=U_{R}/(1-\gamma), ~~\text{for any}~~ (s,a)\in\cS\times\cA.
\$
The same bound also applies  to  $V_{\pi_\theta}(s)$ for any $\pi_\theta$ and $s\in\cS$, and thus to the objective $J(\theta)$ which is defined as $V_{\pi_\theta}(s_0)$, i.e.,
\$
|V_{\pi_\theta}(s)|\leq U_{R}/(1-\gamma),~~\text{for any $s\in\cS$},~~\quad |J(\theta)|\leq U_{R}/(1-\gamma).
\$

In addition, the conditions   \eqref{equ:assum_L_Lip} and \eqref{equ:assum_score_bnded} 
%(or stricter versions of them) 
have also been adopted in several recent work on   the convergence analysis of policy gradient  algorithms \cite{castro2010convergent,pirotta2015policy,papini2018stochastic,papini2019smoothing}. 
%We note that the conditions   \eqref{equ:assum_L_Lip} and \eqref{equ:assum_score_bnded} are even weaker than the counterparts assumed in \cite{castro2010convergent,pirotta2015policy,papini2018stochastic}, where the smoothness of $\nabla\log\pi_{\theta}(a\given s)$ and boundedness of $\|\nabla\log\pi_{\theta}(a\given s)\|$ need to hold for any $a\in\cA$ (not in expectation).  
 Both of the conditions    can be readily satisfied   by many common parametrized policies such as the  Boltzmann policy \cite{konda1999actor} and the Gaussian policy \cite{doya2000reinforcement}. For example, for Gaussian policy\footnote{Note that in practice, the action space $\cA$ is bounded, thus a truncated Gaussian policy over $\cA$ is often used; see \cite{papini2018stochastic}. } in continuous spaces, $\pi_\theta(\cdot\given s)=\cN(\phi(s)^\top\theta,\sigma^2)$, where $\cN(\mu,\sigma^2)$ denotes the Gaussian distribution with mean $\mu$ and variance $\sigma^2$, and $\phi(s)$ is the feature vector that incorporates some domain knowledge to approximate the mean action at state $s$. Then the score function has the form $[a-\phi(s)^\top\theta]\phi(s)/\sigma^2$, which satisfies \eqref{equ:assum_L_Lip} and \eqref{equ:assum_score_bnded} if the following three conditions hold: the norm of the feature  $\|\phi(s)\|$ is bounded; the parameter $\theta$ lies in some bounded set; and the actions $a\in\cA$ is bounded.

 Under Assumption \ref{assum:regularity},   the gradient of $J(\theta)$ with respect to the policy parameter $\theta$, given by the Policy Gradient Theorem \cite{sutton2000policy},  has the following  form\footnote{Note that here we use $\int$ to represent both summation over finite sets and integral over continuous spaces.}:

%To this end, one tempting solution is to use first-order optimization methods, namely, the stochastic gradient method, to find the local minimum of the problem.  In fact, it has been shown in \cite{sutton2000policy} that the policy gradient with respect to $\theta$ can be characterized in the following form: 
\#
\nabla J(\theta)&=\int_{s\in\cS,a\in\cA}\sum_{t=0}^\infty\gamma^t \cdot p(s_t=s\given s_0,\pi_\theta)\cdot\nabla \pi_{\theta}(a\given s)\cdot Q_{\pi_\theta}(s,a)dsda\label{equ:policy_grad_1}\\
&=\frac{1}{1-\gamma}\int_{s\in\cS,a\in\cA}(1-\gamma)\sum_{t=0}^\infty\gamma^t \cdot p(s_t=s\given s_0,\pi_\theta)\cdot\nabla \pi_{\theta}(a\given s)\cdot Q_{\pi_\theta}(s,a)dsda\notag\\
&=\frac{1}{1-\gamma}\int_{s\in\cS,a\in\cA}\rho_{\pi_\theta}(s)\cdot\pi_{\theta}(a\given s)\cdot\nabla \log[\pi_{\theta}(a\given s)]\cdot Q_{\pi_\theta}(s,a)dsda\notag\\
&=\frac{1}{1-\gamma}\cdot\EE_{(s,a)\sim \rho_{\theta}(\cdot,\cdot)}\big[\nabla\log\pi_{\theta}(a\given s)\cdot Q_{\pi_\theta}(s,a)\big]. \label{equ:policy_grad_3}
\#
Here, we denote by  $p(s_t=s\given s_0,\pi_\theta)$  the probability that state $s_t$ equals $s$ given initial state $s_0$ and policy parameter $\theta$, and the distribution $\rho_{\pi_\theta}(s)=(1-\gamma)\sum_{t=0}^\infty\gamma^t p(s_t=s\given s_0,\pi_\theta)$ which has been shown to be  a valid probability measure over the state $\cS$  in \cite{sutton2000policy}. We refer to  $\rho_{\pi_\theta}(s)$ as the  \emph{discounted state-occupancy
measure} hereafter. For notational convenience, we let $\rho_{\theta}(s,a)=\rho_{\pi_\theta}(s)\cdot \pi_{\theta}(a\given s)$, which denotes the \emph{discounted state-action  occupancy
measure}.
%
%The derivative of the logarithm of the policy $\nabla\log[\pi_{\theta}(\cdot\given s)]$ is usually referred to as the \emph{score function} corresponding to the probability distribution $\pi_{\theta}(\cdot\given s)$ for any $s\in\cS$.
%, which comes from the fact that the score function in statistics defines the sensitivity of a likelihood to its parameters \textcolor{blue}{[cite some statistics textbook]}. 

In addition, based on the fact that for any function $b:\cS\to\RR$ independent of action $a$, 
\$
\int_{a\in\cA}\pi_\theta(a\given s)\nabla\log\pi_{\theta}(a\given s)\cdot b(s)da=\nabla \int_{a\in\cA}\pi_{\theta}(a\given s)da\cdot b(s)=\nabla 1\cdot b(s)=0,~~\text{for~any}~~s\in\cS,
\$
the policy gradient in \eqref{equ:policy_grad_3} can be written as 
\$
\nabla J(\theta)=\frac{1}{1-\gamma}\cdot\EE_{(s,a)\sim \rho_{\theta}(\cdot,\cdot)}\big\{\nabla\log\pi_{\theta}(a\given s)\cdot [Q_{\pi_\theta}(s,a)-b(s)]\big\},
\$
where $b(s)$ is usually referred to as a \emph{baseline} function. One common choice of the baseline is the state-value function $V_{\pi_\theta}(s)$, which gives the following advantage-based policy gradient
\#
\nabla J(\theta)=\frac{1}{1-\gamma}\cdot\EE_{(s,a)\sim \rho_{\theta}(\cdot,\cdot)}\big\{\nabla\log\pi_{\theta}(a\given s)\cdot A_{\pi_\theta}(s,a)\big\}.\label{equ:policy_grad_5}
\#
In this work, we devise methods that can use iterative updates based on the classical policy gradient \eqref{equ:policy_grad_3} or its variant that makes use of  the advantage function \eqref{equ:policy_grad_5} through the aforementioned identity regarding baselines. 
First note that under Assumption   \ref{assum:regularity}, we can establish the Lipschitz continuity of the policy gradient $\nabla J(\theta)$ as in the following lemma, whose proof is deferred to \S\ref{apx_lemma:lip_policy_grad}.

%%%%%%%%%%%%%%%%%%%%%%%%%%%%%%%%%%%%%%%%%%%%%%%%%%%%%%%%%%%%%%%%%%%%%%%%%%%%%
%%%%%%%%%%%%%%%%%%%%%%%%%%%%%%%%%%%%%%%%%%%%%%%%%%%%%%%%%%%%%%%%%%%%%%%%%%%%%%%%%%%%%%%%%%%%%%%%L	E	M	M	A %%%%%%%%%%%%%%%%%%%%%%%%%%%%%%%%%%%%%%%%%%%%%%%%%%%%%%%%%%%%%%%%%%%%%%%%%%%%%%%%%%%%%%%%%%%%%%%%%%%%%%%%%%%%%%%%%%%%%%%%%%%%%%%%%%%%%%%%%%%%%%%%%%%%%%%%%%%%%%%%%%%%%%%%%%%%%%%%%%%%%%%%%%%%%%%%%%%%%%%%%%%
\begin{lemma}[Lipschitz-Continuity of Policy Gradient]\label{lemma:lip_policy_grad}
Under Assumption \ref{assum:regularity}, 
the policy gradient $\nabla J(\theta)$ is Lipschitz continuous with some constant $L>0$, i.e., 
	for any $\theta^1,\theta^2\in\RR^d$
	\$
	\|\nabla J(\theta^1)-\nabla J(\theta^2)\|\leq L\cdot \|\theta^1-\theta^2\|, 
	\$
	where the value of the Lipschitz constant $L$ is defined as 
	 \#\label{equ:L_Theta_def}
	L:=\frac{U_{R}\cdot L_\Theta}{(1-\gamma)^2} +\frac{(1+\gamma)\cdot U_R\cdot B^2_{\Theta}}{(1-\gamma)^3}. 
	\#
%	in \eqref{equ:L_Theta_def} in \S\ref{apx_lemma:lip_policy_grad}. 
\end{lemma}

%\textcolor{blue}{I think due to flow, we should defer \emph{all technical lemmas} to appendix. They seem to disrupt the flow of the paper a lot since there are so many of them. We can explain their meaning in words earlier on, and their role in the proof, but defer their exact statement to appendices. I just want to make sure that we can present both algorithms within the first 5-10 pages.}
%\edit{Agreed}

Next, we discuss how \eqref{equ:policy_grad_3} and \eqref{equ:policy_grad_5} can be used to develop first-order stochastic approximation methods to address \eqref{equ:max_goal_para}. 
Unbiased samples of the gradient $\nabla J(\theta)$  are required to perform the stochastic gradient ascent, which hopefully converges to a stationary solution of the nonconvex optimization problem. Moreover, through the addition of carefully designed perturbations, we aim to attain a local optimum, namely, asymptotically stable stationary point, as in \cite{konda1999actor,ge2015escaping,jin2017escape}.

%%%%%%%%%%%%%%%%%%%%%%%%%%%%%%%%%%%%%%%%%%%%%%%%%%%%%%%%%%%%%%
%%%%%%%%%%%%%%%%%%%%%%%%%%%%%%%%%%%%%%%%%%%%%%%%%%%%%%%%%%%%%%%%%%%%%%%%%%%%%%%%%%%%%%%%%%%%%%%% A  L  G  O  R  I  T   H  M%%%%%%%%%%%%%%%%%%%%%%%%%%%%%%%%%%%%%%%%%%%%%%%%%%%%%%%%%%%%%%%%%%%%%%%%%%%%%%%%%%%%%%%%%%%%%%%%%%%%%%%%%%%%%%%%%%%%%%%%%%%%%%%%%%%%%%%%%%%%%%%%%%%%%%%%%%%%%%%%%%%%%%%%%%%%%%%%%%%%%%%%%%%%%%%%%%%%%%%%%%%

\begin{algorithm}[t]
	\caption{~\textbf{EstQ:}~Unbiasedly Estimating Q-function}\label{alg:est_Q}
	\centering
	\begin{algorithmic}
		\STATE \textbf{Input:} $s,a$, and $\theta$. Initialize  $\hat{Q}\leftarrow 0$, $s_0\leftarrow s$, and $a_0\leftarrow a$.
		\STATE Draw $T$ from the geometric distribution $\text{Geom}(1-\gamma^{1/2})$, i.e., $P(T=t)=(1-\gamma^{1/2})\gamma^{t/2}$.
		\FORALL{$t=0,\cdots,T-1$}
			\STATE Collect and add the instantaneous reward  $R(s_t,a_t)$ to $\hat{Q}$, $\hat{Q}\leftarrow\hat{Q}+\gamma^{t/2}\cdot R(s_t,a_t)$.
			\STATE Simulate the next state $s_{t+1}\sim \PP(\cdot\given s_t,a_t)$ and action $a_{t+1}\sim \pi(\cdot\given s_{t+1})$.
		\ENDFOR		
		\STATE Collect $R(s_T,a_T)$ by $\hat{Q}\leftarrow \hat{Q}+\gamma^{T/2}\cdot R(s_T,a_T)$.
%		, and scale $\hat{Q}=(1-\gamma)\hat{Q}$.
		\RETURN $\hat{Q}$.
		\end{algorithmic}
\end{algorithm}

\vspace{6pt}
\noindent{\bf Sampling the Policy Gradient:} 
In order to obtain an unbiased sample of $\nabla J(\theta)$, it is necessary to: i) draw state-action pair $(s,a)$ from the distribution  $\rho_\theta(\cdot,\cdot)$; and ii) obtain an unbiased estimate of the Q-function $Q_{\pi_{\theta}}(s,a)$, or the advantage function $A_{\pi_{\theta}}(s,a)$ evaluated at $(s,a)$.  
%{

Both of the requirements can be satisfied by using a random horizon $T$ that follows certain geometric distribution  in the sampling process.
In particular, to ensure the condition i) is satisfied, we  use the last sample $(s_T,a_T)$ of a finite sample  trajectory $(s_0,a_0,s_1,\cdots,s_T,a_T)$ to be the sample at which $Q_{\pi_{\theta}}(\cdot,\cdot)$ and $\nabla\log\pi_{\theta}(\cdot\given \cdot)$ are evaluated, where the horizon $T\sim \text{Geom}(1-\gamma)$. It can be shown that $(s_T,a_T)\sim \rho_\theta(\cdot,\cdot)$. Moreover, given $(s_T,a_T)$, we perform Monte-Carlo rollouts  for another horizon $T'\sim \text{Geom}(1-\gamma^{1/2})$ independent of $T$, and estimate the Q-function value $Q_{\pi_{\theta}}(s,a)$ as follows by collecting the $\gamma^{1/2}$-discounted  rewards along the trajectory:
%
%
%
% While the later issue can be mitigated by stochastic approximation, i.e., the realization $(s_t,a_t)$ needs to be drawn following  $(s,a)\sim \rho_{\theta}(s,a)$, the former issue can be addressed in a few ways. The simplest is Monte-Carlo rollouts \textcolor{blue}{[need a citation here]}, where the agent moves $T$ steps ahead in time, and records the \emph{actual return} of rewards:
%
\#\label{equ:Est_Q_ori}
\hat{Q}_{\pi_\theta}(s,a)=\sum_{t=0}^{T'}\gamma^{t/2}\cdot R(s_t,a_t)\biggiven s_0=s,a_0=a.
\#
Then, it can be shown that $\hat{Q}_{\pi_\theta}(s,a)$ unbiasedly estimates ${Q}_{\pi_\theta}(s,a)$ for any $(s,a)$ (see Theorem \ref{thm:unbiased_and_bnd_grad_est},  whose proof is given in Appendix \ref{apx_thm:unbiased_and_bnd_grad_est}). The subroutine of estimating the Q-function is summarized as  \textbf{EstQ} in Algorithm \ref{alg:est_Q}. 
%
%The question is how to ensure that the preceding expression is an unbiased estimate for the actual $Q$-function. }This may be achieved by following \cite{santi2018stochastic} in using a stochastic horizon $T$ for drawing samples $(s_t,a_t)$ to estimate $Q$, where the horizon is a geometric random variable parameterized by the discount factor: $T\sim 
%\text{Geom}(\gamma)$.

%%%%%%%%%%%%%%%%%%%%%%%%%%%%%%%%%%%%%%%%%%%%%%%%%%%%%%%%%%%%%%%%%%%%%%%%%%%%%
%%%%%%%%%%%%%%%%%%%%%%%%%%%%%%%%%%%%%%%%%%%%%%%%%%%%%%%%%%%%%%%%%%%%%%%%%%%%%%%%%%%%%%%%%%%%%%%% R  E  M  A  R  K%%%%%%%%%%%%%%%%%%%%%%%%%%%%%%%%%%%%%%%%%%%%%%%%%%%%%%%%%%%%%%%%%%%%%%%%%%%%%%%%%%%%%%%%%%%%%%%%%%%%%%%%%%%%%%%%%%%%%%%%%%%%%%%%%%%%%%%%%%%%%%%%%%%%%%%%%%%%%%%%%%%%%%%%%%%%%%%%%%%%%%%%%%%%%%%%%%%%%%%%%%%
\begin{remark} \label{remark:episodic}
Thanks to  randomness of the  horizon, we note that the aforementioned sampling process creates the first unbiased estimate of the Q-function in the \emph{discounted infinite-horizon} setting, using the Monte-Carlo rollouts of \emph{finite horizons}. 
%In other words, this technique generates an unbiased Q-function estimate in the infinite-horizon setting, as the classical REINFORCE algorithm does for the finite-horizon one. 
While in practice,  usually  finite-horizon rollouts are used to approximate the infinite-horizon Q-function, e.g., in the REINFORCE algorithm, which causes bias in the Q-function  estimate, and hence the policy gradient estimate. Our sampling technique addresses this challenge, and ends up with an unbiased estimate of the policy gradient as to be introduced next. We note that the proposed sampling technique for estimating the Q-function  improves the one in \cite{santi2018stochastic} that uses $ \text{Geom}(1-\gamma)$ (instead of $ \text{Geom}(1-\gamma^{1/2})$) to generate the rollout horizon $T'$. In particular, the proposed Q-function estimate is almost surely bounded thanks to the $\gamma^{1/2}$-discount   factor in \eqref{equ:Est_Q_ori}, which later leads to  almost sure boundedness of the stochastic policy gradient, a necessary assumption required in the convergence analysis to approximate second-order stationary points in \S\ref{sec:conv_second_stationary_point}.
\end{remark}

 %%%%%%%%%%%%%%%%%%%%%%%%%%%%%%%%%%%%%%%%%%%%%%%%%%%%%%%%%%%%%%
%%%%%%%%%%%%%%%%%%%%%%%%%%%%%%%%%%%%%%%%%%%%%%%%%%%%%%%%%%%%%%%%%%%%%%%%%%%%%%%%%%%%%%%%%%%%%%%% A  L  G  O  R  I  T   H  M%%%%%%%%%%%%%%%%%%%%%%%%%%%%%%%%%%%%%%%%%%%%%%%%%%%%%%%%%%%%%%%%%%%%%%%%%%%%%%%%%%%%%%%%%%%%%%%%%%%%%%%%%%%%%%%%%%%%%%%%%%%%%%%%%%%%%%%%%%%%%%%%%%%%%%%%%%%%%%%%%%%%%%%%%%%%%%%%%%%%%%%%%%%%%%%%%%%%%%%%%%
\begin{algorithm}[t]
	\caption{~\textbf{EstV:}~Unbiasedly Estimating State-Value function}\label{alg:est_V}
	\centering
	\begin{algorithmic}
		\STATE \textbf{Input:} $s$ and $\theta$. Initialize  $\hat{V}\leftarrow 0$, $s_0\leftarrow s$, and draw $a_0\sim \pi_\theta(\cdot\given s_0)$.
		\STATE Draw $T$ from the geometric distribution $\text{Geom}(1-\gamma^{1/2})$.
		\FORALL{$t=0,\cdots,T-1$}
			\STATE Collect the instantaneous reward $R(s_t,a_t)$ and add to value $\hat{V}$;  $\hat{V}\leftarrow\hat{V}+\gamma^{t/2}\cdot R(s_t,a_t)$.
			\STATE Simulate the next state $s_{t+1}\sim \PP(\cdot\given s_t,a_t)$ and action $a_{t+1}\sim \pi(\cdot\given s_{t+1})$.
		\ENDFOR		
		\STATE Collect $R(s_T,a_T)$ by $\hat{V}\leftarrow\hat{V}+\gamma^{T/2}\cdot R(s_T,a_T)$.
%		, and scale $\hat{Q}=(1-\gamma)\hat{Q}$.
		\RETURN $\hat{V}$.
		\end{algorithmic}
\end{algorithm} 

Motivated  by the form of policy gradient in \eqref{equ:policy_grad_3}, we propose the following stochastic estimate $\hat{\nabla}J(\theta)$  
\#\label{equ:SGD_eva}
\hat{\nabla}J(\theta)=\frac{1}{1-\gamma}\cdot \hat{Q}_{\pi_\theta}(s_T,a_T)\cdot \nabla\log[\pi_{\theta}(a_T\given s_T)]. 
\#
In addition, we can also estimate the policy gradient using advantage functions as in \eqref{equ:policy_grad_5}, where the  advantage function is estimated by either the difference between the value function and the action-value function, or the temporal difference (TD) error. In particular, we propose the following two  stochastic  policy gradients
\#
\check{\nabla}J(\theta)&=\frac{1}{1-\gamma}\cdot [\hat{Q}_{\pi_\theta}(s_T,a_T)-\hat{V}_{\pi_\theta}(s_T)]\cdot \nabla\log[\pi_{\theta}(a_T\given s_T)],\label{equ:SGD_eva_2}\\
\tilde{\nabla}J(\theta)&=\frac{1}{1-\gamma}\cdot [R(s_T,a_T)+\gamma \hat{V}_{\pi_\theta}(s'_T)-\hat{V}_{\pi_\theta}(s_T)]\cdot \nabla\log[\pi_{\theta}(a_T\given s_T)],\label{equ:SGD_eva_3}
\#
where $\hat{V}_{\pi_\theta}(s)$ is an unbiased estimate of the value function ${V}_{\pi_\theta}(s)$, and $s'_T$ is the next state given state $s_T$ and $a_T$.  
The process of  estimating $\hat{V}_{\pi_\theta}(s)$ employs the same idea as  the \textbf{EstQ} algorithm, where $\hat{V}_{\pi_\theta}(s)$ 
is obtained by collecting the $\gamma^{1/2}$-discounted rewards along the trajectory starting from $s_0$ (instead of a state-action pair $(s,a)$),  following  $a_t\sim \pi_\theta(\cdot\given s_t)$, and of length $T'\sim \text{Geom}(1-\gamma^{1/2})$, i.e., 
$
\hat{V}_{\pi_\theta}(s)=\sum_{t=0}^{T'}\gamma^{1/2}\cdot R(s_t,a_t)\given s_0=s. 
$ 
%Also, $s'_T$ is sampled from $s'_T\sim \PP(\cdot\given s_T,a_T)$. 
We refer to this subroutine as \textbf{EstV}, which is summarized in Algorithm \ref{alg:est_V}. The reason for these alternate updates is that the off-set term can be used to reduce the variance of estimating the  policy gradient \cite{greensmith2004variance}.

We then establish in the following theorem, which states  that all the stochastic policy gradients $\hat{\nabla}J(\theta),\check{\nabla}J(\theta)$, and $\tilde{\nabla}J(\theta)$ are unbiased estimates of ${\nabla}J(\theta)$ [cf. \eqref{equ:policy_grad_3}].
{Additionally, we can also establish  the boundedness of $\|\hat{\nabla}J(\theta)\|,\|\check{\nabla}J(\theta)\|$, and $\|\tilde{\nabla}J(\theta)\|$, as well as $\|\nabla J(\theta)\|$ for any $\theta\in\Theta$. 
The proof is deferred to Appendix \ref{apx_thm:unbiased_and_bnd_grad_est}.  

%%%%%%%%%%%%%%%%%%%%%%%%%%%%%%%%%%%%%%%%%%%%%%%%%%%%%%%%%%%%%%%%%%%%%%%%%%%%%
%%%%%%%%%%%%%%%%%%%%%%%%%%%%%%%%%%%%%%%%%%%%%%%%%%%%%%%%%%%%%%%%%%%%%%%%%%%%%%%%%%%%%%%%%%%%%%%% T		H	E	O	R	E	M %%%%%%%%%%%%%%%%%%%%%%%%%%%%%%%%%%%%%%%%%%%%%%%%%%%%%%%%%%%%%%%%%%%%%%%%%%%%%%%%%%%%%%%%%%%%%%%%%%%%%%%%%%%%%%%%%%%%%%%%%%%%%%%%%%%%%%%%%%%%%%%%%%%%%%%%%%%%%%%%%%%%%%%%%%%%%%%%%%%%%%%%%%%%%%%%%%%%%%%%%%%
\begin{theorem}[Properties of Stochastic Policy Gradients]\label{thm:unbiased_and_bnd_grad_est}
For any $\theta$, $\hat{\nabla}J(\theta),\check{\nabla}J(\theta)$, and $\tilde{\nabla}J(\theta)$ obtained from  \eqref{equ:SGD_eva}, \eqref{equ:SGD_eva_2}, and \eqref{equ:SGD_eva_3}, respectively, are all unbiased estimates of $\nabla J(\theta)$ in \eqref{equ:policy_grad_3}, i.e., for any $\theta$
	\$
	\mathbb{E}[\hat{\nabla} J(\theta) \given \theta] =\mathbb{E}[\check{\nabla} J(\theta) \given \theta]=\mathbb{E}[\tilde{\nabla} J(\theta) \given \theta]=\nabla J(\theta).
	\$
	where the expectation is with respect to the  random horizon $T'$, the trajectory along $(s_0,a_0,s_1,\cdots,s_{T'},a_{T'})$, and the random sample $(s_T,a_T)$.
	Moreover, the norm of the  policy gradient $\nabla J(\theta)$ is bounded, and its stochastic estimates $\hat{\nabla}J(\theta),\check{\nabla}J(\theta),\tilde{\nabla}J(\theta)$  are all almost surely (a.s.) bounded, i.e., 
	\$
	\|\nabla J(\theta)\|\leq \frac{B_{\Theta}\cdot U_R}{(1-\gamma)^2},\quad \|\hat{\nabla}J(\theta)\|\leq \hat{\ell}\tx{~~a.s.},\quad \|\check{\nabla}J(\theta)\|\leq \check{\ell}\tx{~~a.s.},\quad \|\tilde{\nabla}J(\theta)\|\leq \tilde{\ell}\tx{~~a.s.},
	\$
	for some constants $\hat{\ell},\check{\ell},\tilde{\ell}>0$,  whose values are given in \eqref{equ:def_B_J_1}, \eqref{equ:def_B_J_2}, and \eqref{equ:def_B_J_3} in \S\ref{apx_thm:unbiased_and_bnd_grad_est}.
\end{theorem} 
}

Henceforth in this section and the next, we will mainly focus on the convergence analysis for the RPG algorithm with the stochastic gradient $\hat{\nabla}J(\theta)$ as defined in \eqref{equ:SGD_eva}. The RPG algorithms with  $\check{\nabla} J(\theta)$ and $\tilde{\nabla} J(\theta)$ will be discussed later in \S\ref{sec:conv_second_stationary_point}, where reducing the variance of RPG is of greater interest.

%%%%%%%%%%%%%%%%%%%%%%%%%%%%%%%%%%%%%%%%%%%%%%%%%%%%%%%%%%%%%%%%%%%%%%%%%%%%%
%%%%%%%%%%%%%%%%%%%%%%%%%%%%%%%%%%%%%%%%%%%%%%%%%%%%%%%%%%%%%%%%%%%%%%%%%%%%%%%%%%%%%%%%%%%%%%%% A  L  G  O  R  I  T  H  M%%%%%%%%%%%%%%%%%%%%%%%%%%%%%%%%%%%%%%%%%%%%%%%%%%%%%%%%%%%%%%%%%%%%%%%%%%%%%%%%%%%%%%%%%%%%%%%%%%%%%%%%%%%%%%%%%%%%%%%%%%%%%%%%%%%%%%%%%%%%%%%%%%%%%%%%%%%%%%%%%%%%%%%%%%%%%%%%%%%%%%%%%%%%%%%%%%%%%%%%%%%
\begin{algorithm}[t]
	\caption{~\textbf{RPG:}~Random-horizon  Policy Gradient Algorithm}\label{alg:RPG}
	\centering
	\begin{algorithmic}
		\STATE \textbf{Input:} $s_0$ and $\theta_0$, initialize $k\leftarrow 0$.
		\STATE \textbf{Repeat:}
		\bindent
		\STATE Draw $T_{k+1}$ from the geometric distribution $\text{Geom}(1-\gamma)$.
		\STATE Draw $a_0\sim \pi_{\theta_k}(\cdot\given s_0)$
		\FORALL{$t=0,\cdots,T_{k+1}-1$}
			\STATE Simulate the next state $s_{t+1}\sim \PP(\cdot\given s_t,a_t)$ and action $a_{t+1}\sim \pi_{\theta_k}(\cdot\given s_{t+1})$.
		\ENDFOR		
		\STATE Obtain an estimate of $Q_{\pi_{\theta_k}}(s_{T_{k+1}},a_{T_{k+1}})$ by Algorithm \ref{alg:est_Q}, i.e., 
		\$
		\hat{Q}_{\pi_{\theta_k}}(s_{T_{k+1}},a_{T_{k+1}})\leftarrow\textbf{EstQ}(s_{T_{k+1}},a_{T_{k+1}},\theta_k).
		\$
		\STATE Perform  policy gradient update 
		\$
		\theta_{k+1}\leftarrow\theta_k+\frac{\alpha_k}{1-\gamma}\cdot \hat{Q}_{\pi_{\theta_k}}(s_{T_{k+1}},a_{T_{k+1}})\cdot \nabla\log[\pi_{\theta_k}(a_{T	_{k+1}}\given s_{T_{k+1}})]
		\$
		\STATE Update the iteration counter $k\leftarrow k+1$.
		\eindent  
		\STATE \textbf{Until Convergence}
		\end{algorithmic}
\end{algorithm}

To this end, let $k$ be the iteration index  and $\theta_k$ be the  associated estimate for the policy  parameter. 
Under Theorem \ref{thm:unbiased_and_bnd_grad_est}, the  policy gradient update for step $k+1$ is 
\#\label{equ:RPG_update}
\theta_{k+1}=\theta_k+{\alpha_k}\hat{\nabla}J(\theta_k)=\theta_k+\frac{\alpha_k}{1-\gamma}\cdot \hat{Q}_{\pi_{\theta_k}}(s_{T_{k+1}},a_{T_{k+1}})\cdot \nabla\log[\pi_{\theta_k}(a_{T_{k+1}}\given s_{T_{k+1}})],
\#
where $\{\alpha_k\}$ is the stepsize sequence  that can be either diminishing or constant, and $\{T_k\}$ are drawn i.i.d. from $ 
\text{Geom}(1-\gamma)$. The details of the  policy gradient method, which we refer to as the \emph{random-horizon  policy gradient} algorithm, are summarized in Algorithm \ref{alg:RPG}. 
Note that the estimate of $\hat{Q}_{\pi_{\theta_k}}(s_{T_{k+1}},a_{T_{k+1}})$,  i.e., Algorithm \ref{alg:est_Q}, is conducted in the  \emph{inner-loop} of the stochastic  policy gradient update.

\begin{remark} \label{remark:actor_critic}
We note that in order to estimate the Q-function, it is not very sample-efficient to use Monte-Carlo rollouts to sample  states, actions, and rewards.  
In fact, there exist some methods that can estimate the Q-function in parallel with the policy gradient update, which is usually referred to as \emph{actor-critic} method \cite{konda2000actor,bhatnagar2009natural}. This online policy evaluation update is generally performed via \emph{bootstrapping} algorithms such as temporal difference  learning  \cite{dann2014policy}, which will introduce biases into  the Q-function estimate, and thus the policy gradient estimate. In addition, such  policy evaluation updates in concurrence with the policy improvement will inevitably cause correlation between consecutive stochastic policy gradients. Analyzing the non-asymptotic  convergence performance of such \emph{biased} RPG with \emph{correlated noise} is still open and challenging, which is left as a future research  direction. 
%
%Methods which avoid the need for random sampling try to estimate the $Q$ function online, and are referred to as \emph{bootstrapping} in reinforcement learning. However, due to the fact that as the policy changes, the sampling distribution for the rewards changes, using boostrapping to estimate $Q$ may cause the policy gradient to become biased, and thus yield instability, a phenomenon referred to as the \emph{deadly triad} in \cite{sutton2017reinforcement}. One way to avoid this instability is to ensure the estimation of the $Q$ function is on a slower time-scale than the rate of change of the policy. This intertwined procedure is referred to as \emph{actor-critic} method \cite{konda2000actor,bhatnagar2009natural}, which is beyond the scope of this work.
\end{remark}

In the next sections, we shift focus to analyzing the theoretical properties of the aforementioned policy learning methods, establishing their asymptotic and finite-time performances, as well as stepsize strategies designed to mitigate the challenges of non-convexity when certain reward structure is present.

%!TEX root = Policy_Grad.tex
%%%%%%%%%%%%%%%%%%%%%%%%%%%%%%%%%%%%%%%%%%%%%%%%%%%%%%%%%%%%%%%%%%%%%%%%%%%%%
%%%%%%%%%%%%%%%%%%%%%%%%%%%%%%%%%%%%%%%%%%%%%%%%%%%%%%%%%%%%%%%%%%%%%%%%%%%%%%%%%%%%%%%%%%%%%%%% S  E  C	T  I  O  N%%%%%%%%%%%%%%%%%%%%%%%%%%%%%%%%%%%%%%%%%%%%%%%%%%%%%%%%%%%%%%%%%%%%%%%%%%%%%%%%%%%%%%%%%%%%%%%%%%%%%%%%%%%%%%%%%%%%%%%%%%%%%%%%%%%%%%%%%%%%%%%%%%%%%%%%%%%%%%%%%%%%%%%%%%%%%%%%%%%%%%%%%%%%%%%%%%%%%%%%%%%
\section{Convergence to  Stationary Points}\label{sec:conv_stationary_point}
In this section, we provide convergence analyses for the policy gradient algorithms proposed in \S\ref{sec:alg}.  
%Before proceeding any further, we first lay out the assumptions needed for the analysis.   
We start with the following assumption for the diminishing stepsize $\alpha_k$, which is standard in stochastic approximation.
 
\begin{assumption}\label{assum:RM_Cond}
	The sequence of stepsize $\{\alpha_k\}_{k\geq 0}$   satisfies the  Robbins-Monro condition
\$
\sum_{k=0}^\infty \alpha_k=\infty,\quad\sum_{k=0}^\infty\alpha_k^2<\infty.
\$
\end{assumption}

%\subsection{Asymptotic Convergence of Algorithm \ref{alg:RPG}}
%
We first establish the  convergence of Algorithm \ref{alg:RPG} in the following theorem under the aforementioned technical conditions.

%%%%%%%%%%%%%%%%%%%%%%%%%%%%%%%%%%%%%%%%%%%%%%%%%%%%%%%%%%%%%%%%%%%%%%%%%%%%%
%%%%%%%%%%%%%%%%%%%%%%%%%%%%%%%%%%%%%%%%%%%%%%%%%%%%%%%%%%%%%%%%%%%%%%%%%%%%%%%%%%%%%%%%%%%%%%%%T  H  E  O  R  E  M %%%%%%%%%%%%%%%%%%%%%%%%%%%%%%%%%%%%%%%%%%%%%%%%%%%%%%%%%%%%%%%%%%%%%%%%%%%%%%%%%%%%%%%%%%%%%%%%%%%%%%%%%%%%%%%%%%%%%%%%%%%%%%%%%%%%%%%%%%%%%%%%%%%%%%%%%%%%%%%%%%%%%%%%%%%%%%%%%%%%%%%%%%%%%%%%%%%%%%%%%%%
\begin{theorem}[Asymptotic Convergence of Algorithm \ref{alg:RPG}]\label{thm:alg1_as_conv}
	Let $\{\theta_k\}_{k\geq 0}$ be the sequence of parameters of the policy  $\pi_{\theta_k}$ given by Algorithm \ref{alg:RPG}. Then under Assumptions \ref{assum:regularity} and  \ref{assum:RM_Cond}, we have $\lim_{k\to\infty}\theta_k\in\Theta^*$, where $\Theta^*$ is  the set of stationary points of $J(\theta)$.
\end{theorem}

Theorem \ref{thm:alg1_as_conv}, whose proof is in Appendix \ref{append:proof:thm:alg1_as_conv}, shows that the random-horizon policy gradient update converges to the (first-order) stationary points of $J(\theta)$ almost surely.  
The proof of the theorem is relegated to \S\ref{append:proof:thm:alg1_as_conv}. 
We note that the asymptotic convergence result here is  established from an optimization perspective using supermartingale convergence theorem \cite{robbins1985convergence}, which differs from the existing techniques that show convergence of actor-critic algorithms from dynamical systems theory (or ODE method) \cite{borkar2008stochastic}. Such optimization perspective can be leveraged thanks to the unbiasedness of the stochastic  policy gradients obtained from Algorithm \ref{alg:RPG}.

An additional virtue of this style of analysis is that we can also establish convergence rate  of the policy gradient algorithm without the need for sophisticated concentration inequalities. 
In contrast, the {finite-iteration} analysis for actor-critic algorithms is known to be quite challenging \cite{dalal2017finite,yang18cdc}.  
By convention, we choose the stepsize to be either $
\alpha_k=k^{-a}$ for some parameter $a\in(0,1)$ or constant $\alpha>0$. Note that for the diminishing stepsize,  here we allow a more general choice  than that in Assumption \ref{assum:RM_Cond}. Since   $J(\theta)$ is generally nonconvex,  we consider the convergence rate in terms of a metric of nonstationarity, i.e., the  norm of the gradient $\|\nabla J(\theta_k)\|$.
We  then provide the convergence rates of Algorithm \ref{alg:RPG} for the setting of using  diminishing and constant stepsizes in the  following theorem  and corollary, respectively. 
The proofs of the results are given in \S\ref{apx_thm:alg1_conv_rate}. 

 %%%%%%%%%%%%%%%%%%%%%%%%%%%%%%%%%%%%%%%%%%%%%%%%%%%%%%%%%%%%%%%%%%%%%%%%%%%%%
%%%%%%%%%%%%%%%%%%%%%%%%%%%%%%%%%%%%%%%%%%%%%%%%%%%%%%%%%%%%%%%%%%%%%%%%%%%%%%%%%%%%%%%%%%%%%%%%T  H  E  O  R  E  M %%%%%%%%%%%%%%%%%%%%%%%%%%%%%%%%%%%%%%%%%%%%%%%%%%%%%%%%%%%%%%%%%%%%%%%%%%%%%%%%%%%%%%%%%%%%%%%%%%%%%%%%%%%%%%%%%%%%%%%%%%%%%%%%%%%%%%%%%%%%%%%%%%%%%%%%%%%%%%%%%%%%%%%%%%%%%%%%%%%%%%%%%%%%%%%%%%%%%%%%%%%
\begin{theorem}[Convergence Rate of  Algorithm \ref{alg:RPG} with Diminishing Stepsize]\label{thm:alg1_conv_rate}
	Let $\{\theta_k\}_{k\geq 0}$ be the sequence of parameters of the policy  $\pi_{\theta_k}$ given by Algorithm \ref{alg:RPG}. Let the stepsize  be $\alpha_k=k^{-a}$ where $a\in(0,1)$. Let 
	\$
	K_\epsilon=\min\big\{k:\inf_{0\leq m\leq k}\EE\|\nabla J(\theta_m)\|^2\leq \epsilon \big\}.
	\$
	Then, under Assumption \ref{assum:regularity}, we have $K_\epsilon\leq O(\epsilon^{-1/p})$, where $p$ is defined as $p=\min\{1-a,a\}$. 
	By optimizing the complexity bound over $a$, we obtain $K_\epsilon\leq O(\epsilon^{-2})$ with $a=1/2$.
\end{theorem}
%%
%\textcolor{green}{
%\begin{enumerate}
%%
%\item I was thinking that we should explicitly state the theorem in terms of $\alpha_k = k^{-1/2}$ only. And then explain where this selection comes from in the proof. That way the result will be easier to interpret, and we can draw the connection between this and the stepsize selection one typically has for rate analysis in the convex setting using Polyak-averaging.
%\textcolor{cyan}{I this we can keep this form of exposition, since this gives us more choices of stepsizes.}
%%
%\item We can also use the ascent lemma derived from \eqref{equ:Wk_iter} mentioned at the beginning of the following proof to conduct constant stepsize analysis and convergence in mean to a bounded error neighborhood. I think constant learning rate results are not well understood in reinforcement learning.
%\textcolor{cyan}{I have added a new theorem for constant stepsize.}
%%
%\end{enumerate}
%}
% 

%The proof of Theorem \ref{thm:alg1_conv_rate}  easily gives way to show the convergence rate using constant stepsize, which is outlined in the following corollary.  
%%%%%%%%%%%%%%%%%%%%%%%%%%%%%%%%%%%%%%%%%%%%%%%%%%%%%%%%%%%%%%%%%%%%%%%%%%%%%
%%%%%%%%%%%%%%%%%%%%%%%%%%%%%%%%%%%%%%%%%%%%%%%%%%%%%%%%%%%%%%%%%%%%%%%%%%%%%%%%%%%%%%%%%%%%%%%% C		O	R	O	L	L	A	R	Y%%%%%%%%%%%%%%%%%%%%%%%%%%%%%%%%%%%%%%%%%%%%%%%%%%%%%%%%%%%%%%%%%%%%%%%%%%%%%%%%%%%%%%%%%%%%%%%%%%%%%%%%%%%%%%%%%%%%%%%%%%%%%%%%%%%%%%%%%%%%%%%%%%%%%%%%%%%%%%%%%%%%%%%%%%%%%%%%%%%%%%%%%%%%%%%%%%%%%%%%%%%
\begin{corollary}[Convergence Rate of  Algorithm \ref{alg:RPG} with Constant Stepsize]\label{coro:alg1_conv_rate_const_step}
	Let $\{\theta_k\}_{k\geq 0}$ be the sequence of parameters of the policy  $\pi_{\theta_k}$ given by Algorithm \ref{alg:RPG}. Let the stepsize  be $\alpha_k=\alpha>0$.  Then, under Assumption \ref{assum:regularity}, we have  
	\$
	\frac{1}{k}\sum_{m=1}^k\EE\|\nabla J(\theta_m)\|^2\leq O(\alpha L \hat{\ell}^2),
	\$
	where recall that $L$ is the Lipchitz constant of the policy gradient as defined in \eqref{equ:L_Theta_def} in  Lemma \ref{lemma:lip_policy_grad}. 
\end{corollary}

Theorem \ref{thm:alg1_conv_rate} illustrates that 
when diminishing stepsize is adopted, which essentially establishes  a $1/\sqrt{k}$ convergence rate for the convergence of the expected gradient norm square $\|\nabla J(\theta_k)\|^2$.
Corollary \ref{coro:alg1_conv_rate_const_step} shows that the average of the  gradient norm square will converge to a neighborhood around zero with the rate of $1/k$. The size of the neighborhood  is controlled by the stepsize $\alpha$. Moreover, \eqref{equ:conv_const_step_im} also implies that a smaller stepsize may decrease the size of the neighborhood, at the expense of the convergence speed.
We note that both results are standard and recover the convergence properties of stochastic gradient descent for nonconvex optimization problems \cite{shapiro2009lectures,ghadimi2013stochastic}. In the next  section, we propose modified stepsize rules, which under an appropriate hypothesis on the policy parameterization and reward structure of the problem, yield stronger limiting policies.

%!TEX root = Policy_Grad.tex
%%%%%%%%%%%%%%%%%%%%%%%%%%%%%%%%%%%%%%%%%%%%%%%%%%%%%%%%%%%%%%%%%%%%%%%%%%%%%
%%%%%%%%%%%%%%%%%%%%%%%%%%%%%%%%%%%%%%%%%%%%%%%%%%%%%%%%%%%%%%%%%%%%%%%%%%%%%%%%%%%%%%%%%%%%%%%% S  E  C	T  I  O  N%%%%%%%%%%%%%%%%%%%%%%%%%%%%%%%%%%%%%%%%%%%%%%%%%%%%%%%%%%%%%%%%%%%%%%%%%%%%%%%%%%%%%%%%%%%%%%%%%%%%%%%%%%%%%%%%%%%%%%%%%%%%%%%%%%%%%%%%%%%%%%%%%%%%%%%%%%%%%%%%%%%%%%%%%%%%%%%%%%%%%%%%%%%%%%%%%%%%%%%%%%%
\section{Convergence to Second-Order Stationary Points}\label{sec:conv_second_stationary_point}
In this section, we provide convergence analyses for several modified policy gradient algorithms based on Algorithm \ref{alg:RPG}, which may escape saddle points and thus converge to the approximate  second-order stationary points of the problem.  In short, we propose a custom periodically enlarged stepsize rule, which under an additional hypothesis on the incentive structure of the problem and some other standard conditions (see \S\ref{sec:MRPG}), allow us to attain improved limiting policy parameters (see \S\ref{sec:MRPG_Conv}).

We start with the definition of \edit{(approximate)} second-order stationary points \cite{nesterov2006cubic}\footnote{Note that Definition \ref{def:2nd_stationary_point} is based on the maximization problem  we consider here, which is slightly different from the definition for minimization problems where $\lambda_{\max}[\nabla^2 J(\theta)]\leq \epsilon_h$ is replaced by $\lambda_{\min}[\nabla^2 J(\theta)]\geq -\epsilon_h$.}. 
 
\begin{definition}\label{def:2nd_stationary_point}
	An $(\epsilon_g,\epsilon_h)$-\edit{approximate-}second-order stationary point $\theta$ is defined as
	\$
	\|\nabla J(\theta)\|\leq \epsilon_g,\qquad \lambda_{\max}[\nabla^2 J(\theta)]\leq \epsilon_h. 
	\$
	\edit{If $\epsilon_g=\epsilon_h=0$, the point $\theta$ is a second-order stationary point. }
\end{definition}

The intuition for this definition is that a local maximum is one in which the gradient is null and the Hessian is negative semidefinite. When we relax the first criterion, we obtain the first inequality, whereas when we relax the second one, we mean that the Hessian is near negative semidefinite.

With the further assumption that all saddle points are strict (i.e., for any saddle point $\theta$, $\lambda_{\max}[\nabla^2 J(\theta)]>0$) \edit{\cite{jin2017escape,ge2015escaping}}, all second-order stationary points ($\epsilon_g=\epsilon_h=0$) are local maxima. In this case, converging to \edit{(approximate)} second-order
stationary points is equivalent to converging to \edit{approximate} local minima, which is usually more desirable  than converging to (first-order) stationary  points.

\subsection{Algorithm}\label{sec:MRPG}
The modified RPG (MRPG) algorithms are built upon the RPG algorithm (Algorithm \ref{alg:RPG}) discussed in Section \ref{sec:alg}. These modifications can  yield escape from saddle points under certain conditions, and hence convergence to approximate local extrema.

In order to reduce the variance of the RPG update \eqref{equ:RPG_update}, we employ the stochastic gradients $\check{\nabla}J(\theta)$ and $\tilde{\nabla}J(\theta)$ as defined in \eqref{equ:SGD_eva_2} and \eqref{equ:SGD_eva_3}, respectively. Note that  the evaluations of both  $\check{\nabla}J(\theta)$ and $\tilde{\nabla}J(\theta)$ need  to estimate the state-value function $\hat{V}_{\pi_\theta}(s)$ for any given $\theta$ and $s$. 
Built upon the subroutines \textbf{EstQ} and \textbf{EstV}, we summarize the subroutine for calculating all three types of  stochastic policy gradients as \textbf{EvalPG} in Algorithm \ref{alg:EvalPG}.

In order to converge to the approximate  second-order stationary points, we modify the RPG algorithm, i.e., Algorithm \ref{alg:RPG}, by \emph{periodically enlarging} the \emph{constant}  stepsize of   the update, once  every $k_{\tx{thre}}$ steps. The larger stepsize can amplify the variance along the eigenvector  corresponding to the largest eigenvalue of the Hessian, which provides a direction for the update to escape at the saddle points.  This idea was first introduced in \cite{daneshmand2018escaping}  for general stochastic gradient methods, and is outlined in Algorithm \ref{alg:MRPG}. 
%Different from the algorithm there, we propose to return from the visited iterates whose indices satisfies $k\tx{~mod~} k_{\tx{thre}}=0$, i.e., from the set $\hat{\Theta}^*$ in Algorithm \ref{alg:MRPG}. 
Note that  $\alpha$ and $\beta$ are the constant stepsizes with $\beta>\alpha>0$, whose values will be given in  \S\ref{sec:MRPG_Conv}  to obtain certain convergence rates. To design this behavior while avoiding unnecessarily large variance, we propose updates that make use of the advantage function, i.e., $\check{\nabla}J(\theta)$ and $\tilde{\nabla}J(\theta)$. The resulting algorithm, with periodically enlarged stepsizes, and stochastic policy gradients that use advantage functions, is summarized as Algorithm \ref{alg:MRPG}. Subsequently, we shift focus to characterizing its policy learning performance analysis.

\begin{algorithm}[t]
	\caption{~\textbf{EvalPG:}~Calculating the Three Types of  Stochastic Policy Gradients}\label{alg:EvalPG}
	\centering
	\begin{algorithmic}
		\STATE \textbf{Input:} $s,a$, $\theta$ and the gradient type $\diamondsuit$. 
%		\STATE Draw  $T$ from the geometric distribution $\text{Geom}(1-\gamma)$. 
%		\FORALL{$t=0,\cdots,T-1$}
%			\STATE Simulate the next state $s_{t+1}\sim \PP(\cdot\given s_t,a_t)$ and action $a_{t+1}\sim \pi_{\theta_k}(\cdot\given s_{t+1})$.
%		\ENDFOR	
		\IF {gradient type $\diamondsuit=\ \hat{} \ $}
		\STATE Obtain an estimate    $
		\hat{Q}_{\pi_{\theta}}(s,a)\leftarrow\textbf{EstQ}(s,a,\theta)$. 
		\STATE Calculate  $\hat{\nabla}J(\theta)$, i.e., let 
		\$
		{g}_\theta\leftarrow\frac{1}{1-\gamma}\cdot \hat{Q}_{\pi_{\theta}}(s,a)\cdot \nabla\log\pi_{\theta}(a\given s).
		\$
		\ELSIF {gradient type $\diamondsuit=\ \check{} \ $}
		\STATE Obtain  estimates    $
		\hat{Q}_{\pi_{\theta}}(s,a)\leftarrow\textbf{EstQ}(s,a,\theta)$ and $ 
		\hat{V}_{\pi_{\theta}}(s)\leftarrow\textbf{EstV}(s,\theta)$. 
		\STATE Calculate  $\check{\nabla}J(\theta)$, i.e., let 
		\$
		{g}_\theta\leftarrow\frac{1}{1-\gamma}\cdot [\hat{Q}_{\pi_{\theta}}(s,a)-\hat{V}_{\pi_{\theta}}(s)]\cdot \nabla\log\pi_{\theta}(a\given s).
		\$
		\ELSIF {gradient type $\diamondsuit=\ \tilde{} \ $}  
		\STATE Simulate the next state:  $s'\sim \PP(\cdot\given s,a)$. 
		\STATE Obtain  estimates    $ 
		\hat{V}_{\pi_{\theta}}(s)\leftarrow\textbf{EstV}(s,\theta)$ and $\hat{V}_{\pi_{\theta}}(s')\leftarrow\textbf{EstV}(s',\theta)$. 
		\STATE Calculate  $\tilde{\nabla}J(\theta)$, i.e., let 
		\$
		{g}_\theta\leftarrow\frac{1}{1-\gamma}\cdot [R(s,a)+\gamma\cdot\hat{V}_{\pi_{\theta}}(s')-\hat{V}_{\pi_{\theta}}(s)]\cdot \nabla\log\pi_{\theta}(a\given s). 
		\$
		\ENDIF 
		\RETURN  Stochastic policy gradient  ${g}_\theta$
		\end{algorithmic}
\end{algorithm} 
 
%%%%%%%%%%%%%%%%%%%%%%%%%%%%%%%%%%%%%%%%%%%%%%%%%%%%%%%%%%%%%%%%%%%%%%%%%%%%%
%%%%%%%%%%%%%%%%%%%%%%%%%%%%%%%%%%%%%%%%%%%%%%%%%%%%%%%%%%%%%%%%%%%%%%%%%%%%%%%%%%%%%%%%%%%%%%%% A  L  G  O  R  I  T  H  M%%%%%%%%%%%%%%%%%%%%%%%%%%%%%%%%%%%%%%%%%%%%%%%%%%%%%%%%%%%%%%%%%%%%%%%%%%%%%%%%%%%%%%%%%%%%%%%%%%%%%%%%%%%%%%%%%%%%%%%%%%%%%%%%%%%%%%%%%%%%%%%%%%%%%%%%%%%%%%%%%%%%%%%%%%%%%%%%%%%%%%%%%%%%%%%%%%%%%%%%%%%
\begin{algorithm}[th]
	\caption{~\textbf{MRPG:}~Modified Random-horizon  Policy Gradient Algorithm}\label{alg:MRPG}
	\centering
	\begin{algorithmic}
		\STATE \textbf{Input:} $s_0$, $\theta_0$, and the gradient type $\diamondsuit$, initialize $k\leftarrow 0$, return set $\hat{\Theta}^*\leftarrow \emptyset$.
		\STATE \textbf{Repeat:}
		\bindent
		\STATE Draw $T_{k+1}$ from the geometric distribution $\text{Geom}(1-\gamma)$, and draw $a_0\sim \pi_{\theta_k}(\cdot\given s_0)$.
	\FORALL{$t=0,\cdots,T_{k+1}-1$}
			\STATE Simulate the next state $s_{t+1}\sim \PP(\cdot\given s_t,a_t)$ and action $a_{t+1}\sim \pi_{\theta_k}(\cdot\given s_{t+1})$.
		\ENDFOR		
		\STATE Calculate the stochastic  gradient   $g_k\leftarrow\textbf{EvalPG}(s_{T_{k+1}},a_{T_{k+1}},\theta_k,\diamondsuit)$. 
		\IF {$(k\tx{~mod~}k_{\tx{thre}})=0$}
		\STATE 
		\$
		\hat{\Theta}^*\leftarrow \hat{\Theta}^*\cup\{\theta_{k}\}
		\\
		\theta_{k+1}\leftarrow\theta_k+\beta\cdot g_k 
		\$
		\ELSE 
		\STATE
		\$
		\theta_{k+1}\leftarrow\theta_k+\alpha\cdot g_k
		\$
		\ENDIF
		\STATE Update the iteration counter $k=k+1$.
		\eindent  
		\STATE \textbf{Until Convergence}
		\RETURN{ {$\theta$ uniformly at random} from the set $\hat{\Theta}^*$.}
		\end{algorithmic}
\end{algorithm}

\subsection{Convergence Analysis}\label{sec:MRPG_Conv}

In this subsection, we provide {a finite-iteration} convergence result for the modified RPG algorithm, i.e., Algorithm \ref{alg:MRPG}. To this end, we first  
introduce the following condition, built upon Assumption \ref{assum:regularity}, which is  required in the sequel.

\begin{assumption}\label{assum:regularity_for_CNC}
The MDP and the parameterized policy $\pi_\theta$ satisfy the following conditions:
\begin{enumerate}[label=(\roman*)]
	\item The reward $R(s,a)$ is either positive or negative for any $(s,a)\in\cS\times\cA$. Thus, $|R(s,a)|\in[L_R,U_R]$  with some $L_R>0$. \label{as:reward_strict_sign}
	\item The score function $\nabla \log \pi_\theta$ exists, and its norm is   bounded by $\|\nabla \log \pi_\theta\|\leq B_\Theta$ for any $\theta$. 
	Also, the Jacobian of  $\nabla \log \pi_\theta$  has bounded norm and is Lipschitz continuous, i.e., there exist constants $\rho_\Theta>0$ and $L_\Theta<\infty$ such that for any  $(s,a)\in\cS\times\cA$ \label{as:smooth_policy}
	\$
	&\big\|\nabla^2\log \pi_{\theta^1}(a\given s)-\nabla^2\log \pi_{\theta^2}(a\given s)\big\|\leq \rho_\Theta\cdot\|\theta^1-\theta^2\|,\text{~~for any~~} \theta^1,\theta^2,\\
	&\big\|\nabla^2\log \pi_{\theta}(a\given s)\big\|\leq L_\Theta,\text{~~for any~~} \theta.
	\$
	\item The integral of the Fisher information matrix induced by $\pi_\theta(\cdot\given s)$ is positive-definite uniformly  for any  $\theta\in\RR^d$, i.e., there exists a constant $L_I>0$ such that \label{as:fischer}
	\#\label{equ:Fisher_lower_bnd}
	\int_{s\in\cS,a\in\cA}\rho_\theta(s,a)\cdot\nabla\log  \pi_\theta(\cdot\given s)\cdot [\nabla\log \pi_\theta(\cdot\given s)]^\top da ds\succeq  L_I\cdot\bI,~~\text{for~all~~}\theta\in\RR^d.
	\#
%	\item Either the policy $\pi_\theta$ is log-concave with respect to $\theta$; or, the policy parameterization has constant curvature with respect to the action space, i.e., for any $(s,a)\in\cS\times\cA$, the matrix $\nabla^2\log[\pi_\theta(a\given s)]$ does not depend upon the action $a$.
\end{enumerate}
\end{assumption}

We note that Assumption \ref{assum:regularity_for_CNC} is indeed standard, and can be readily satisfied in practice. First, the strict positivity (or negativity) of the reward function in Assumption \ref{assum:regularity_for_CNC}\ref{as:reward_strict_sign} can be easily satisfied by adding (or subtracting) an offset to the original non-negative and upper-bounded reward. In fact, it can be justified in the following lemma that adding any offset does not change the optimal policy of the original MDP.

\begin{lemma}\label{lemma:offset_reward_nochange}
	Given any MDP $\cM=(\cS,\cA,\PP,R,\gamma)$, let $\tilde{\cM}$ be a modified MDP of  $\cM$, such that $\cM=(\cS,\cA,\PP,\tilde{R},\gamma)$ and $\tilde{R}(s,a)=R(s,a)+C$, for any $s\in\cS,a\in\cA$, and  $C\in\RR$. Then, the sets of optimal policies for the two  MDPs, $\cM$ and $\tilde{\cM}$, are equal. 
\end{lemma}

The proof of the lemma is deferred to Appendix \S\ref{apx_offset}.  
The  positivity (or negativity) of the rewards ensures that the absolute value of the Q-function is also lower-bounded, by the value of $L_R/(1-\gamma)$, which will benefit the convergence of the MRPG algorithm as to be specified shortly. 
Interestingly, such a \emph{reshape} of the reward  function can be shown to yield better convergence results. To our knowledge, our work appears to be the first theoretical study on the effect of reward-reshaping on the convergence property of policy gradient methods, although reward-reshaping is known to be useful for improving learned policies in practice.
%\kznote{This is a good point. I have added this sentence to emphasize our contribution here. I think in the Introduction part, we do need to review some empirical "reward-reshaping" papers. }

On the other hand, 
we  note that such positivity of  $|Q_{\pi_\theta}|$ will cause a relatively large variance in the original RPG update  \eqref{equ:RPG_update}. This makes the RPG with baseline, i.e., the use of $\check{\nabla}J(\theta)$ and $\tilde{\nabla}J(\theta)$, beneficial for variance reduction.

The latter conditions \ref{as:smooth_policy}-\ref{as:fischer} in Assumption \ref{assum:regularity_for_CNC} can also be satisfied easily  by commonly used policies such as Gaussian policies and Gibbs policies. For example, for a Gaussian policy, $\nabla^2\log \pi_{\theta}(a\given s)$ reduces to the matrix $\phi(s)\phi(s)^\top/\sigma^2$, which is a constant function of $\theta$ and thus satisfies    condition   \ref{as:smooth_policy}.  
Such a condition is used  to show the Lipschitz continuity of the Hessian matrix of the objective  function $J(\theta)=V_{\pi_\theta}(s_0)$, which is standard  in establishing the convergence to approximate second-order stationary points in the nonconvex optimization literature \cite{ge2015escaping,jin2017escape,daneshmand2018escaping,xu2017newton}. 
Formally, the Lipschitz continuity of the Hessian is substantiated in the following lemma, whose proof is relegated to Appendix \S\ref{apx_Hessian}.  

\begin{lemma}\label{lemma:Hessian_Lip}
	The Hessian matrix of the objective function  $\cH(\theta)$ is Lipschitz continuous, i.e.,  with some constant $\rho>0$,
%	, i.e., for any $\theta^1,\theta^2\in\RR^d$ 
	\$
	\big\|\cH(\theta^1)-\cH(\theta^2)\big\|\leq \rho \cdot \|\theta^1-\theta^2\|,~~\text{for~~any}~~\theta^1,\theta^2\in\RR^d.
	\$
	The value of the Lipschitz constant $\rho$ is given in \eqref{equ:value_rho} in \S\ref{apx_Hessian}. 
\end{lemma}

The third condition \ref{as:fischer} in Assumption \ref{assum:regularity_for_CNC} holds for many \emph{regular} policy parameterizations, and has been assumed in prior works on natural policy gradient \cite{kakade2002natural} and actor-critic algorithms \cite{bhatnagar2009natural}. 
%In fact, since \eqref{equ:def_Fisher} holds for any $s\in\cS$ and $\theta\in\RR^d$, and both $\cS$ and $\Theta$ are compact, by extreme value theorem, there exists a uniform bound $L_I>0$, such that 
%\#\label{equ:Fisher_lower_bnd}
%	\int_{a\in\cA}\pi_\theta(\cdot\given s)\cdot\nabla\log \pi_\theta(\cdot\given s)\cdot [\nabla\log \pi_\theta(\cdot\given s)]^\top da \succeq L_I\cdot\bI,~~\text{for~all~~}s\in\cS,\theta\in\Theta.
%\#
%
We  note that  Assumption \ref{assum:regularity_for_CNC}   implies   Assumption \ref{assum:regularity}. More specifically, the  condition on the   reward function  in Assumption \ref{assum:regularity} only requires boundedness, without further requirement on its positivity/ negativity; the boundedness of the norm $\|\nabla^2\log \pi_{\theta}(a\given s)\|$ in Assumption \ref{assum:regularity_for_CNC} implies the $L_\Theta$-Lipschitz continuity of the score function $\nabla\log \pi_{\theta}(a\given s)$ in Assumption \ref{assum:regularity}. And the condition on the positive-definiteness of the Fisher information matrix is additional. We will show shortly that these  stricter assumptions enable stronger convergence guarantees.

% The original third condition
%The third condition  in Assumption \ref{assum:regularity_for_CNC} is  satisfied by common policy parameterizations including the Gibbs policy, i.e., $\pi_\theta(a\given s)\propto \exp[\theta^\top \phi(s,a)]$, where $\phi(s,a)$ is a vector of features that depends on the state-action $(s,a)$; and the Gaussian policy in continuous action spaces, i.e., $\pi_\theta(a\given s)\propto \exp\{-[a-\theta^\top\phi(s)]^\top\Sigma^{-1}[a-\theta^\top\phi(s)]\}$, where $\phi(s)$ is a feature vector related to the state $s$ \cite{furmston2016approximate}. 

Now we  show that  all the three stochastic policy gradients $\hat{\nabla} J(\theta)$, $\check{\nabla} J(\theta)$, and $\tilde{\nabla} J(\theta)$ satisfy the so-termed correlated negative curvature (CNC) condition \cite{daneshmand2018escaping}, which is crucial in the ensuing analysis. The proof of Lemma \ref{lemma:CNC_verify} is deferred to Appendix \S\ref{apx_cnc}.

%%%%%%%%%%%%%%%%%%%%%%%%%%%%%%%%%%%%%%%%%%%%%%%%%%%%%%%%%%%%%%%%%%%%%%%%%%%%%
%%%%%%%%%%%%%%%%%%%%%%%%%%%%%%%%%%%%%%%%%%%%%%%%%%%%%%%%%%%%%%%%%%%%%%%%%%%%%%%%%%%%%%%%%%%%%%%% L		E	M	M	A	%%%%%%%%%%%%%%%%%%%%%%%%%%%%%%%%%%%%%%%%%%%%%%%%%%%%%%%%%%%%%%%%%%%%%%%%%%%%%%%%%%%%%%%%%%%%%%%%%%%%%%%%%%%%%%%%%%%%%%%%%%%%%%%%%%%%%%%%%%%%%%%%%%%%%%%%%%%%%%%%%%%%%%%%%%%%%%%%%%%%%%%%%%%%%%%%%%%%%%%%%%%
\begin{lemma}\label{lemma:CNC_verify}
	Under Assumption   \ref{assum:regularity_for_CNC}, all the three stochastic policy gradients $\hat{\nabla} J(\theta)$, $\check{\nabla} J(\theta)$, and $\tilde{\nabla} J(\theta)$ satisfy the correlated negative curvature 
	 condition, i.e., letting $\vb_\theta$ be the unit-norm eigenvector  corresponding to the maximum  eigenvalue of the Hessian matrix $\cH(\theta)$,  there exist constants $\hat{\eta},\check{\eta},\tilde{\eta}>0$ such that  for any $\theta\in\RR^d$
	\$
	\EE\big\{[\vb_\theta^\top \hat{\nabla}J(\theta)]^2\biggiven \theta\big\}\geq \hat{\eta},\quad \EE\big\{[\vb_\theta^\top \check{\nabla}J(\theta)]^2\biggiven \theta\big\}\geq \check{\eta},\quad \EE\big\{[\vb_\theta^\top \tilde{\nabla}J(\theta)]^2\biggiven \theta\big\}\geq \tilde{\eta}.
	\$
\end{lemma} 

%\textcolor{blue}{insert a sentence or two explaining the intuition and mathematical utility of the CNC condition: basically, we continue to have a valid descent direction in a neighborhood of the saddle point.}
The CNC condition, established in Appendix \ref{apx_cnc}, basically illustrates that the perturbation caused by the stochastic gradient is guaranteed to have variance along the direction with positive curvature, i.e., the escaping direction of the objective \cite{daneshmand2018escaping}. 
Such an escaping direction is dictated by  the eigenvectors associated with the maximum eigenvalue of the Hessian matrix $\cH(\theta)$. 
The CNC condition here can be satisfied thanks to Assumption \ref{assum:regularity_for_CNC}, primarily due to the strict positivity of the absolute value of the reward, and the positive-definiteness of the Fisher information matrix. To be more specific, recall the formula of stochastic policy gradients in \eqref{equ:SGD_eva}-\eqref{equ:SGD_eva_3}, such two conditions ensure: i) the square of the Q-value/advantage function estimates is strictly positive and uniformly lower-bounded; ii) thus the expectation of the outer-product of the stochastic policy gradients is strictly positive-definite, which gives the lower bound in Lemma \ref{lemma:CNC_verify}.  
The argument will be detailed in the proof of the lemma  in Appendix \S\ref{apx_cnc}. 

Now we are ready to lay out the following convergence guarantees of the modified RPG algorithm, i.e., Algorithm \ref{alg:MRPG}. 
The  values of  parameters used in the analysis are  specified in   {Table \ref{table:parameters}.

\begin{table}
\begin{tabular}{|c|c|c|c|c|c|}
\hline
Param.       & Value                                                                                                  & Order         & Constraint                                                                              & Equation                                       & Const.       \\ \hline
$\beta$         & $c_1\epsilon^2/(2\ell^2L)$                                                                             & $\cO(\epsilon^2)$                    & $\leq\epsilon^2/(2\ell^2L)$                                                             & \eqref{equ:large_grad_immed_1p5}  & $c_1=1$        \\ \hline
$\beta$         & ''                                                                                                     & ''                                  & $\leq [J_{\tx{thre}}\delta/(2L\ell^2)]^{1/2}$                                           & \eqref{equ:dec_bnd_near_SP}       & ''             \\ \hline
$\beta$         & ''                                                                                                     & $\cO(\epsilon)$                      & $\leq \eta\lambda^2/(24L\ell^3\rho)$                                                    & \eqref{equ:contra_immed_1}         & ''             \\ \hline
$J_{\tx{thre}}$ & $c_2\eta\epsilon^4/(2\ell^2L)$                                                                         & $\cO(\epsilon^{4})$                  & $\leq \beta \epsilon^2/2$                                                              & \eqref{equ:large_grad_immed_2}    & $c_2=c_1/2$    \\ \hline
$J_{\tx{thre}}$ & ''                                                                                                     & ''                                   & $\leq {\eta\beta\lambda^2}/{(48\ell\rho)}$                                              & \eqref{equ:contra_immed_2}          & ''             \\ \hline
$\alpha$        & $c_1\epsilon^2/(2\ell^2L\sqrt{k_{\tx{thre}}})$                                                                                                     & $\cO(\epsilon^{9/2})$                                  & $\leq \beta/\sqrt{k_{\tx{thre}}}$                                                       & \eqref{equ:large_grad_immed_1p25} &              \\ \hline
$\alpha$        & ''                                                                   & ''                & $\leq c'\eta\beta\lambda^3/(24L\ell^3\rho)$                                             & \eqref{equ:contra_immed_3}         & '' \\ \hline
$k_{\tx{thre}}$ & $c_4\frac{\log[{L\ell_g}/{(\eta\beta\alpha\sqrt{\rho\epsilon})}]}{\alpha(\rho\epsilon)^{1/2}}$ & $\Omega(\epsilon^{-5}\log(1/\epsilon))$ & $\geq {c}\frac{\log[{L\ell_g}/({\eta\beta\alpha\lambda})]}{\alpha\lambda}$ & \eqref{equ:contra_immed_4}         & $c_4=c$        \\ \hline
$K$             & $c_5\frac{[{J^*-J(\theta_0)}]k_{\tx{thre}}}{{\delta J_{\tx{thre}}}}$                                         & $\Omega(\epsilon^{-9}\log(1/\epsilon))$ & $\geq 2\frac{[{J^*-J(\theta_0)}]k_{\tx{thre}}}{{\delta J_{\tx{thre}}}}$                       & \eqref{equ:K_lower_bnd}            & $c_5=2$        \\ \hline
\end{tabular}
\caption{List of parameter values used in the convergence analysis.}
\label{table:parameters}
\end{table}

\begin{theorem}\label{thm:conv_MRPG}
	Under Assumption  \ref{assum:regularity_for_CNC}, Algorithm \ref{alg:MRPG} returns an $(\epsilon,\sqrt{\rho\epsilon})$-approximate {second-order stationary point} policy with probability at least $(1-\delta)$ after
	\#\label{equ:final_res}
	\cO\bigg(\Big(\frac{\rho^{3/2}L\epsilon^{-9}}{\delta\eta}\Big)\log\Big(\frac{\ell_g L}{\epsilon\eta\rho}\Big)\bigg),
	\#
	steps, where $\delta\in(0,1)$,  $\ell_g^2:=2\ell^2+{2B^2_\Theta U^2_R}/{(1-\gamma)^4}$, $B_\Theta,U_R$ are as defined in Assumption \ref{assum:regularity_for_CNC}, $\rho$ is the  Lipschitz constant of the Hessian in Lemma \ref{lemma:Hessian_Lip},  
	$\ell$ and $\eta$ take the values of $\hat{\ell},\check{\ell},\tilde{\ell}$ in Theorem \ref{thm:unbiased_and_bnd_grad_est} and $\hat{\eta},\check{\eta},\tilde{\eta}$ in Lemma \ref{lemma:CNC_verify}, when the stochastic policy gradients $\hat{\nabla}J(\theta),\check{\nabla}J(\theta)$, and $\tilde{\nabla}J(\theta)$ are used, respectively.   
\end{theorem}

The proof of Theorem \ref{thm:conv_MRPG} originates but improves the proof techniques in \cite{daneshmand2018escaping}\footnote{Our convergence result corresponds to Theorem $2$ in \cite{daneshmand2018escaping}. However, we have identified and informed the authors, and have been acknowledged,  that there is a flaw in their proof, which breaks the convergence rate claimed in the original version of the paper (personal communication). At the time the current manuscript is prepared, the authors of \cite{daneshmand2018escaping} have corrected the proof in the Arxiv version using a similar idea to what we proposed in the personal communication. }, and is relegated to Appendix \S\ref{apx_sosp}. 
Note that we follow the convention of using  $(\epsilon,\sqrt{\rho\epsilon})$ as the  convergence criterion for approximate  second-order stationary points  \cite{nesterov2006cubic,zhang2017hitting,jin2017escape}, which  reflects the natural relation between the
gradient and the Hessian. 
Theorem \ref{thm:conv_MRPG} concludes that it is possible for the  policy gradient algorithm to escape the saddle points efficiently and retrieve  an approximate  second-order stationary point in a polynomial number of steps\footnote{Note that the number of steps here in \eqref{equ:final_res} corresponds to the notion of \emph{iteration complexity} in the literation of optimization, which is not the total \emph{sample complexity} since each step of our algorithm requires two rollouts with random but finite horizon. Thus, the \emph{expected} number of samples, i.e., state-action-reward tuples, equals $1/(1-\gamma)+1/(1-\gamma^{1/2})$ times the expression in \eqref{equ:final_res}}.  Additionally, if all saddle points are \emph{strict} (cf. definition in \cite{ge2015escaping}), the modified  RPG algorithm will converge  to an actual \emph{local-optimal} policy.   In the next section, we experimentally investigate the validity of our algorithms proposed in this section and the previous section, and probe whether reward-shaping to mitigate challenges of non-convexity is borne out empirically.

%!TEX root = Policy_Grad.tex
%%%%%%%%%%%%%%%%%%%%%%%%%%%%%%%%%%%%%%%%%%%%%%%%%%%%%%%%%%%%%%%
%%%%%%%%%%%%%%%%%%%%%%%%%%%%%%%%%%%%%%%%%%%%%%%%%%%%%%%%%%%%%%%%%%%%%%%%%%%%%%%%%%%%%%%%%%%%%%%% S  E  C  T  I  O  N%%%%%%%%%%%%%%%%%%%%%%%%%%%%%%%%%%%%%%%%%%%%%%%%%%%%%%%%%%%%%%%%%%%%%%%%%%%%%%%%%%%%%%%%%%%%%%%%%%%%%%%%%%%%%%%%%%%%%%%%%%%%%%%%%%%%%%%%%%%%%%%%%%%%%%%%%%%%%%%%%%%%%%%%%%%%%%%%%%%%%%%%%%%%%%%%%%%%%%%%%%%
%%%%%%%%%%%%%%%%%%%%%%%%%%%%%%%%%%%%%%%%%%%%%%%%%%%%%%%%%%%%%%%
%%%%%%%%%%%%%%%%%%%%%%%%%%%%%%%%%%%%%%%%%%%%%%%%%%%%%%%%%%%%%%%%%%%%%%%%%%%%%%%%%%%%%%%%%%%%%%%% F   I   G   U   R   E %%%%%%%%%%%%%%%%%%%%%%%%%%%%%%%%%%%%%%%%%%%%%%%%%%%%%%%%%%%%%%%%%%%%%%%%%%%%%%%%%%%%%%%%%%%%%%%%%%%%%%%%%%%%%%%%%%%%%%%%%%%%%%%%%%%%%%%%%%%%%%%%%%%%%%%%%%%%%%%%%%%%%%%%%%%%%%%%%%%%%%%%%%%%%%%%%%%%%%%%%%%
\begin{figure}[t]	
\vspace{4mm}
%centerline instead of centering in order to manipulate vertical margin
	\centerline{\includegraphics[width=.45\columnwidth,clip = true]{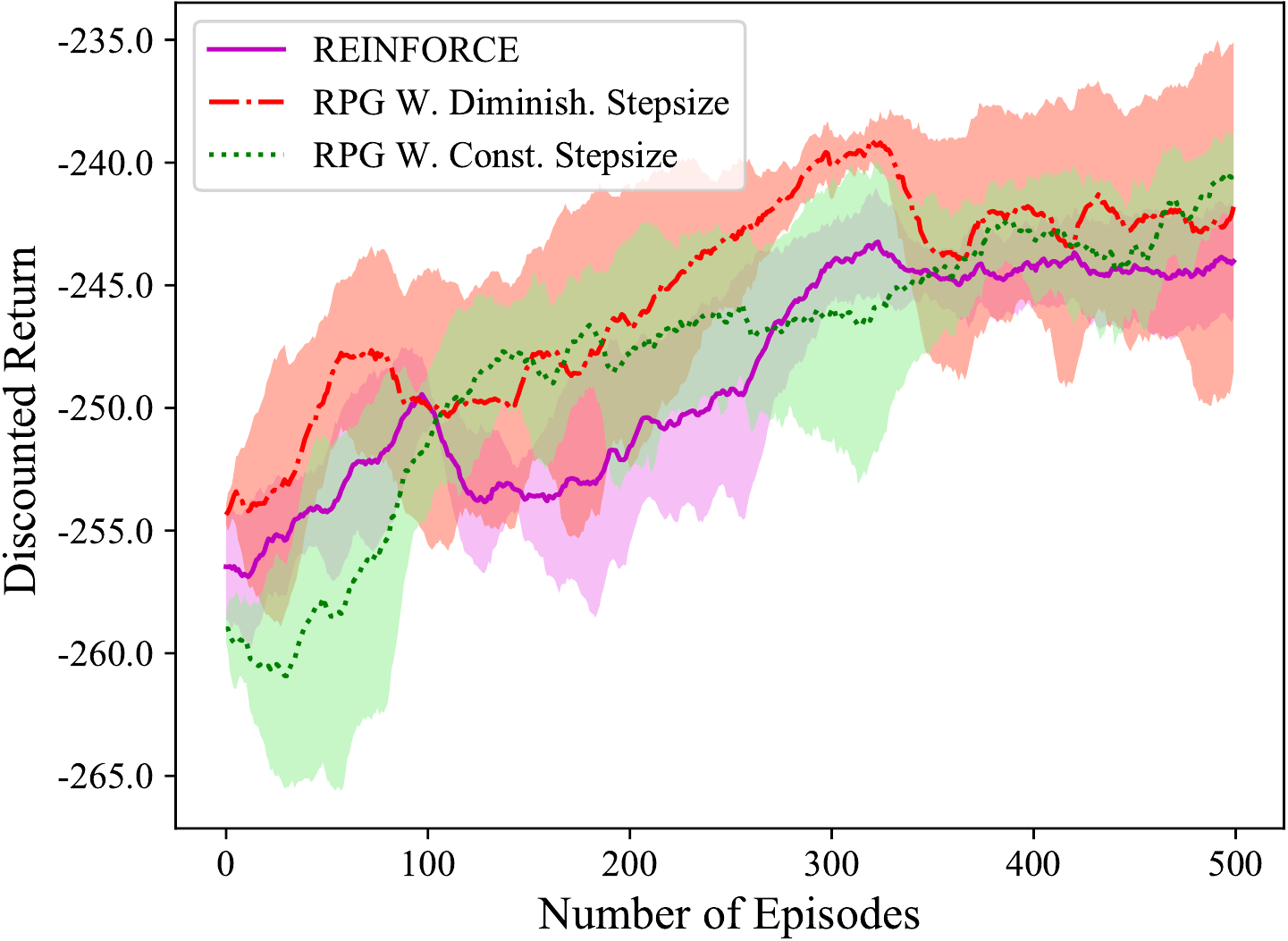}\hspace{4mm}
	\includegraphics[width=.45\columnwidth,clip = true]{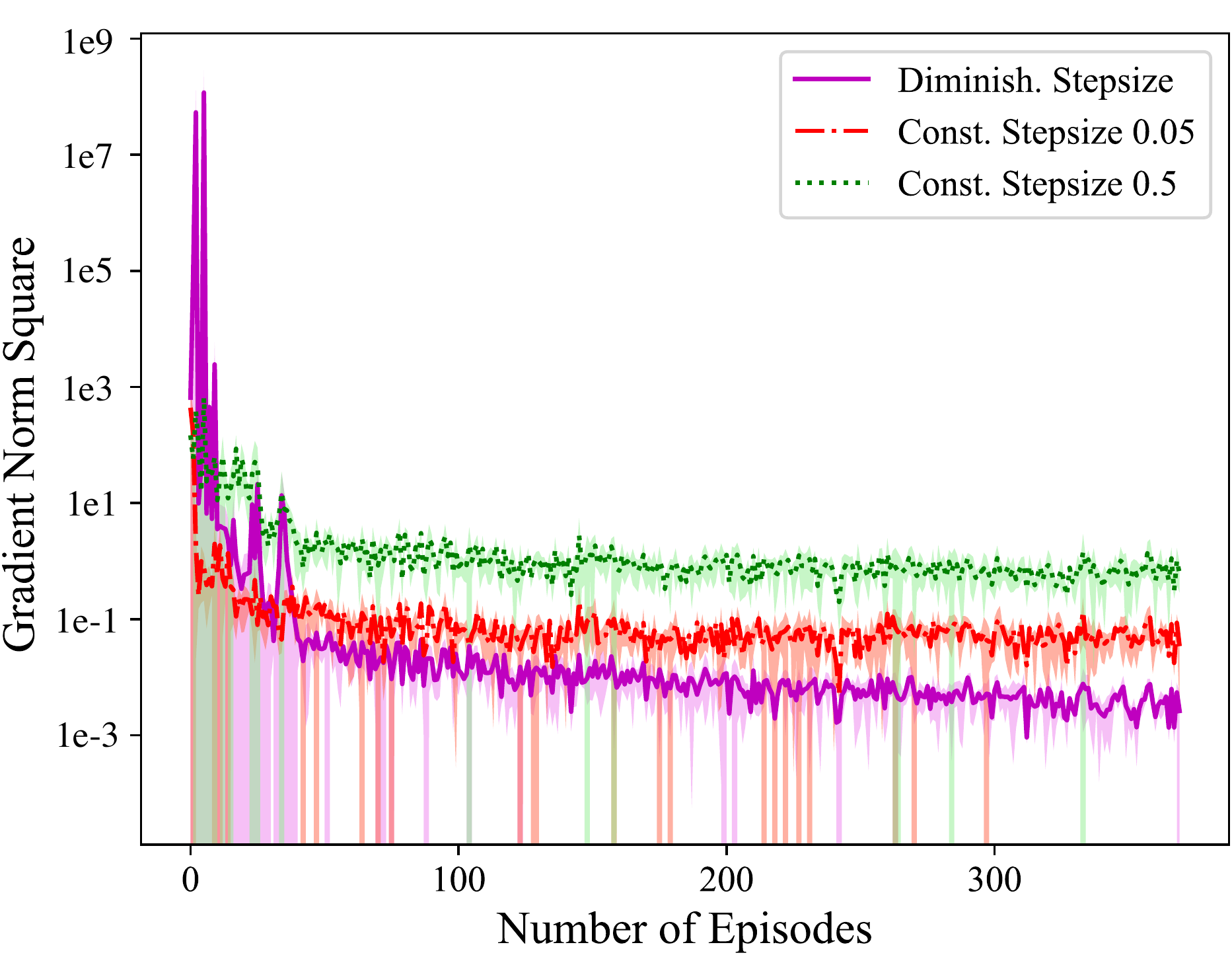}}
	\vspace{0mm}
	%\centerline{\hspace*{0.6cm}\footnotesize(a)\hspace*{4.3cm}(b)}
	{{\centering
	\caption{Left:The convergence of discounted return $J(\theta)$ when REINFORCE and the proposed RPG (Algorithm \ref{alg:RPG})   are used. Both  diminishing and constant stepsizes are used for RPG. Right: The convergence of gradient norm square $\EE\|\nabla J(\theta_m)\|^2$ when the proposed RPG (Algorithm \ref{alg:RPG}) with different stepsizes are used. } \label{fig:discount_return}}}\vspace{0mm}
\end{figure}
%
%%%%%%%%%%%%%%%%%%%%%%%%%%%%%%%%%%%%%%%%%%%%%%%%%%%%%%%%%%%%%%%%
%%%%%%%%%%%%%%%%%%%%%%%%%%%%%%%%%%%%%%%%%%%%%%%%%%%%%%%%%%%%%%%%%%%%%%%%%%%%%%%%%%%%%%%%%%%%%%%%% F   I   G   U   R   E %%%%%%%%%%%%%%%%%%%%%%%%%%%%%%%%%%%%%%%%%%%%%%%%%%%%%%%%%%%%%%%%%%%%%%%%%%%%%%%%%%%%%%%%%%%%%%%%%%%%%%%%%%%%%%%%%%%%%%%%%%%%%%%%%%%%%%%%%%%%%%%%%%%%%%%%%%%%%%%%%%%%%%%%%%%%%%%%%%%%%%%%%%%%%%%%%%%%%%%%%%% 
%\begin{figure}[h]
%	%\centering
%	\vspace{4mm}
%	\centerline{\includegraphics[width=.9\columnwidth,clip = true]{Final_res_grad_norm.pdf}}
%	\vspace{0mm}   
%	%\centerline{\hspace*{0.6cm}\footnotesize(a)\hspace*{4.3cm}(b)}
%	{{\centering  
%	\caption{The convergence of gradient norm square $\EE\|\nabla J(\theta_m)\|^2$ when the proposed RPG (Algorithm \ref{alg:RPG}) with different stepsizes are used. } \label{fig:grad_norm}}}\vspace{-4mm} 
%\end{figure} 

 \section{Simulations}\label{sec:simulations}
In this section, we present several experiments to corroborate the results of the previous two sections.    
Focused on the discounted infinite-horizon setting, we use the Pendulum environment in the OpenAI gym \cite{Brockman2016open} as the test environment. In particular, the pendulum starts in a random position, and the goal is to swing it up so that it stays upright. The state is a vector of dimension three, i.e.,  $s_t=(\cos(\theta_t),\sin(\theta_t),\dot{\theta}_t)^\top$, where $\theta_t$ is the angle between the pendulum and the upright direction, and $\dot{\theta}_t$ is the derivative of $\theta_t$. The action $a_t$ is a one-dimensional scalar representing the joint effort. In addition, the reward $R(s_t,a_t)$ is defined as
\#\label{equ:reward_Pendulum}
R(s_t,a_t):=-(\theta^2 + 0.1*\dot{\theta}^2 + 0.001*a_t^2)-0.5,
\#
which lies in $[-17.1736044,-0.5]$, since $\theta$ is normalized between $[-\pi,\pi]$ and $a_t$ lies in $[-20,20]$.
Different from the reward in the original Pendulum environment, we shift the reward by $-0.5$, so that the negativity of $R(s,a)$ in Assumption \ref{assum:regularity_for_CNC} is satisfied, i.e., $L_R=0.5$. 
The transition probability follows the physical rules of Newton's Second Law.
We choose the discounted factor $\gamma$ to be $0.97$. 
We use Gaussian policy $\pi_\theta$ truncated over   the support $[-20,20]$, which is parameterized as $\pi_\theta(\cdot\given s)=\cN(\mu_\theta(s),\sigma^2)$, where $\sigma=1.0$ and $\mu_\theta(s):\cS\to \cA$ is a neural network with two hidden layers. Each hidden layer contains $10$ neurons and uses \emph{softmax} as activation functions. The output layer of $\mu_\theta(s)$ uses $\tanh$ as the activation function. One can verify  that such  parameterization satisfies Assumption \ref{assum:regularity_for_CNC}.

%%%%%%%%%%%%%%%%%%%%%%%%%%%%%%%%%%%%%%%%%%%%%%%%%%%%%%%%%%%%%%%
%%%%%%%%%%%%%%%%%%%%%%%%%%%%%%%%%%%%%%%%%%%%%%%%%%%%%%%%%%%%%%%%%%%%%%%%%%%%%%%%%%%%%%%%%%%%%%%% F   I   G   U   R   E %%%%%%%%%%%%%%%%%%%%%%%%%%%%%%%%%%%%%%%%%%%%%%%%%%%%%%%%%%%%%%%%%%%%%%%%%%%%%%%%%%%%%%%%%%%%%%%%%%%%%%%%%%%%%%%%%%%%%%%%%%%%%%%%%%%%%%%%%%%%%%%%%%%%%%%%%%%%%%%%%%%%%%%%%%%%%%%%%%%%%%%%%%%%%%%%%%%%%%%%%%%
\begin{figure}[t]
	\centering
	\vspace{-3pt}	\includegraphics[width=0.45\columnwidth,clip = true]{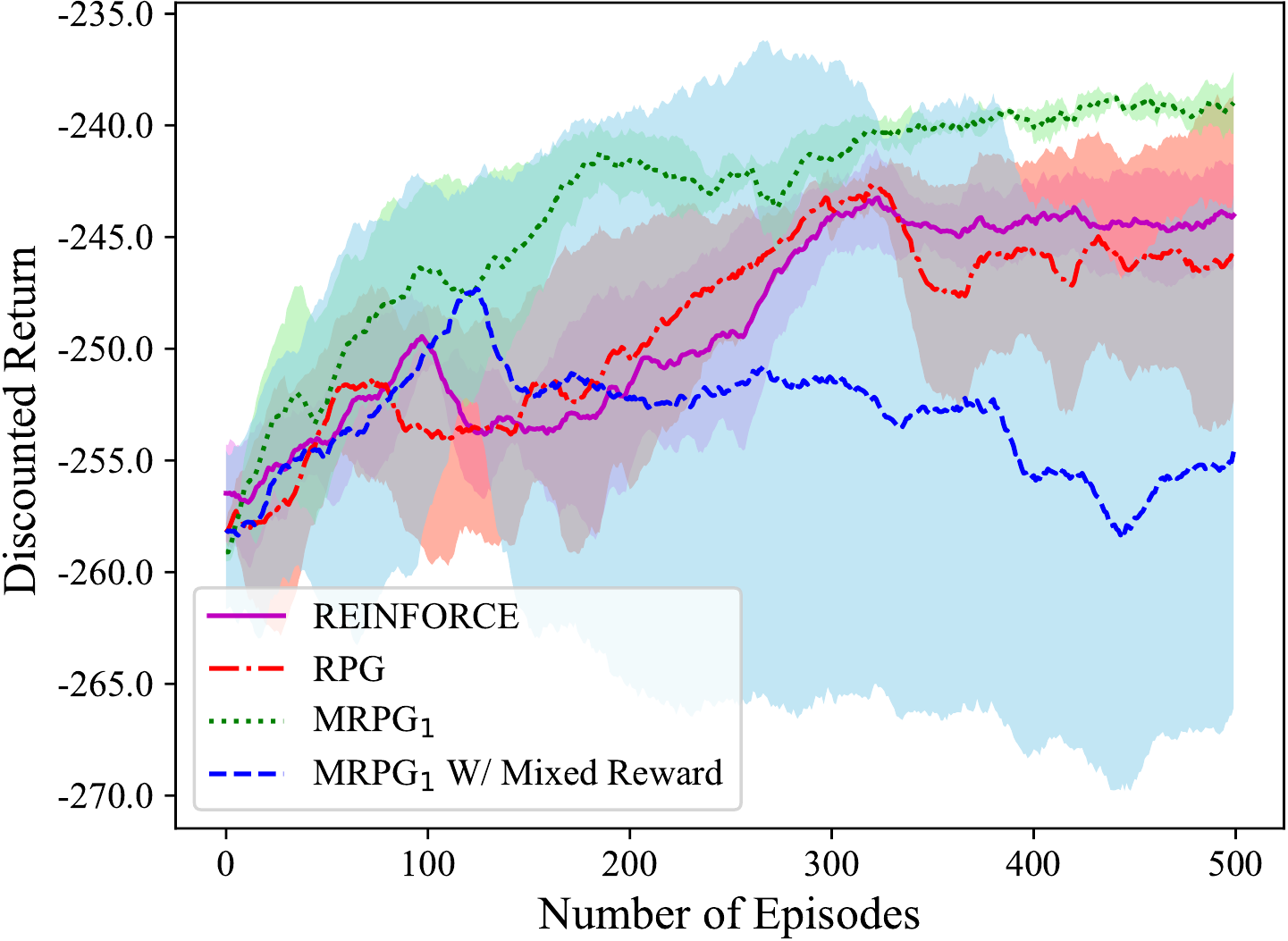}\hspace{4mm}
	\includegraphics[width=0.45\columnwidth,clip = true]{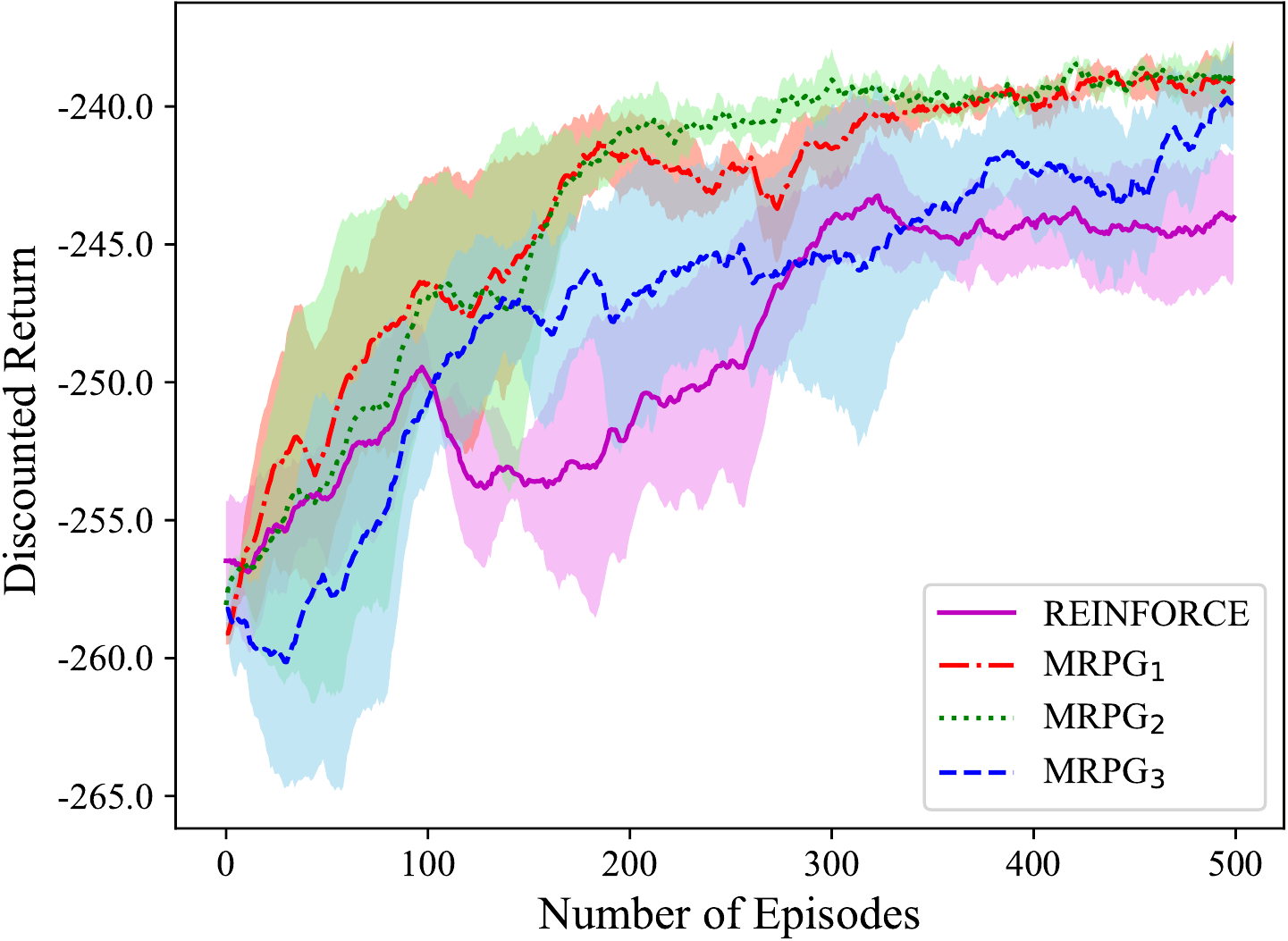}
	\vspace{0mm}
	%\centerline{\hspace*{0.6cm}\footnotesize(a)\hspace*{4.3cm}(b)}
	{{\centering
	\caption{Left:The convergence of discounted return $J(\theta)$ when REINFORCE, the proposed RPG (Algorithm \ref{alg:RPG}), and MRPG$_1$  (Algorithm \ref{alg:MRPG}) algorithms are used. The MRPG$_1$ algorithm is also evaluated for the setting with mixed reward, i.e., the reward can be both negative and positive. Right: the convergence of discounted return $J(\theta)$ when REINFORCE and the proposed  MRPG (Algorithm \ref{alg:MRPG}) algorithms with three types of policy gradients are used }\label{fig:offset_comp}}}
\end{figure}

We first compare the performance of our algorithms with that of the popular REINFORCE algorithm \cite{williams1992simple}. To make the comparison fair, we choose the length of the rollout horizon of REINFORCE to be the expected value of the {geometric distribution with success probability $1-\gamma^{1/2}$, i.e., $T=\gamma^{1/2}/(1-\gamma^{1/2})=66$. Recall that the length of the rollout horizon for Q-function  estimate in our algorithm is drawn from $\text{Geom}(1-\gamma^{1/2})$.} After each rollout, i.e., one episode, the policy parameter $\theta_k$ is updated and then evaluated by calculating the value of $J(\theta)$ using the Monte-Carlo method.

First, we compare the performance of RPG ( Algorithm \ref{alg:RPG}) with that of the popular REINFORCE algorithm \cite{williams1992simple}. 
Recall that REINFORCE creates bias in the policy gradient estimate. To make a fair comparison, we set  the rollout horizon of REINFORCE to be the expected value of the {geometric distribution with success probability $1-\gamma^{1/2}$, the same distribution that the rollout horizon for Q-function  estimate in Algorithm \ref{alg:est_Q} is drawn from,  i.e., $T=\gamma^{1/2}/(1-\gamma^{1/2})=66$. For RPG, we test both diminishing and constant stepsizes, where the former is set as $\alpha_k=1/\sqrt{k}$ and the latter is set as $\alpha_k=0.05$ for all $k\geq 0$.  

Fig. \ref{fig:discount_return}(left) plots the discounted return obtained along the iterations of REINFORCE and our proposed RPG algorithms. The return is estimated by  running the algorithms $30$ times. The bar areas represent the standard deviation region calculated using the $30$ simulations. It is shown that  our proposed algorithms perform slightly better than REINFORCE  in terms of discounted return, but with higher variance. This is expected since our policy gradient estimates are unbiased, compared to REINFORCE. Moreover, the higher variance possibly comes from the additional randomness of the rollout horizon in RPG. 

We also evaluate the convergence of the expected gradient norm square studied in Theorem \ref{thm:alg1_conv_rate} and Corollary \ref{coro:alg1_conv_rate_const_step}. 
Fig. \ref{fig:discount_return}(right)  plots the empirical estimates of $\EE\|\nabla J(\theta_m)\|^2$ after $30$ runs of the algorithms. It is verified that  using  diminishing stepsize results in convergence of the gradient norm   to zero a.s. (the curve keeps decreasing), while using constant stepsizes leads to an error that is lower-bounded above zero (the curves stay mostly unchanged after certain episodes). Moreover, it is shown that  a smaller constant stepsize indeed creates a smaller size of the error neighborhood. Convergence rates  under  both diminishing and constant stepsize choices are sublinear, as identified in our theoretical results.

We further evaluate the performance of Algorithm \ref{alg:MRPG} that uses intermittently larger stepsizes with stochastic policy gradient $\hat{\nabla}J(\theta)$ as \emph{MRPG$_1$}, which theoretically we expect to yield favorable performance under appropriately designed incentive structure. Thus, in order to verify  the significance of  the CNC condition  in escaping saddle points, we also test the MRPG$_1$ algorithm in the environment that has mixed reward, i.e., the reward can be both positive and negative. We generate such an environment by  adding a constant $10.0$ onto the reward defined in  \eqref{equ:reward_Pendulum}. Each learning curve in Figure \ref{fig:offset_comp} is run for $30$ times, and the bar area in the figure represents plus or minus one sample standard deviation of $30$ trajectories.

First, it can be seen from Figure \ref{fig:offset_comp}(left) that RPG achieves almost identical performance as REINFORCE, which shows that the \emph{unbiasedness} of the RPG update seems to not hold great advantages over the biased PG obtained from REINFORCE, in finding the first-order stationary points. On the other hand, Figure \ref{fig:offset_comp}(left) illustrates that MRPG$_1$ achieves greater return than RPG, substantiating the necessity of finding approximate second-order stationary points than first-order ones. To the best of our knowledge, this appears to be the first empirical observation in RL that saddle-escaping techniques may benefit the policy learning. Interestingly, when the  reward is  ``mixed'',  the MRPG$_1$ algorithm suffers  from  lower  discounted return and larger variance across $30$ trajectories. This may be explained by the fact that different trajectories may converge to different saddle points or stationary points that may be of very different qualities. This observation also justifies the necessity of escaping undesirable saddle points  for policy gradient updates.

We have also evaluated the performance of the other two MRPG algorithms that use  the policy gradients   $\check{\nabla}J(\theta)$ and $\tilde{\nabla}J(\theta)$, which we refer to as \emph{MRPG$_2$} and \emph{MRPG$_3$}, respectively, in Figure \ref{fig:offset_comp}(right). Recall that the key differences of these alternative gradient updates is that they subtract a baseline or use Bellman's evaluation equation, respectively, to replace the $Q$ function that multiplies the score function with the advantage function. As shown in Figure \ref{fig:offset_comp}(right), the update with baselines does not always benefit the variance reduction, at least in this experiment. In particular, the policy gradient $\check{\nabla} J(\theta)$ that uses $V(s)$ as the baseline indeed outperforms the MRPG$_1$ algorithm; however, the policy gradient $\check{\nabla} J(\theta)$ that uses TD error to estimate the advantage function performs even worse. Even so, all the MRPG algorithms beat  the REINFORCE algorithm in terms of discounted return, and MRPG$_1$ and MRPG$_2$ also beat REINFORCE in terms of variance.

%!TEX root = Policy_Grad.tex
%%%%%%%%%%%%%%%%%%%%%%%%%%%%%%%%%%%%%%%%%%%%%%%%%%%%%%%%%%%%%%%%%%%%%%%%%%%%%
%%%%%%%%%%%%%%%%%%%%%%%%%%%%%%%%%%%%%%%%%%%%%%%%%%%%%%%%%%%%%%%%%%%%%%%%%%%%%%%%%%%%%%%%%%%%%%%% S  E  C	T  I  O  N%%%%%%%%%%%%%%%%%%%%%%%%%%%%%%%%%%%%%%%%%%%%%%%%%%%%%%%%%%%%%%%%%%%%%%%%%%%%%%%%%%%%%%%%%%%%%%%%%%%%%%%%%%%%%%%%%%%%%%%%%%%%%%%%%%%%%%%%%%%%%%%%%%%%%%%%%%%%%%%%%%%%%%%%%%%%%%%%%%%%%%%%%%%%%%%%%%%%%%%%%%%
\section{Conclusions}\label{sec:conclusions}

Despite its tremendous popularity, policy gradient methods in RL have rarely been investigated in terms of their global convergence, i.e., there seems to be a gap in the literature regarding the limiting properties of policy search and how this is a function of the initialization. Motivated by this gap, we have adopted the perspective and tools from nonconvex optimization to clarify and partially overcome some of the challenges of policy search for MDPs over continuous spaces.
In particular, we have developed a series of random-horizon policy gradient algorithms, which generate \emph{unbiased} estimates of the policy gradient for the \emph{infinite-horizon} setting. Under standard assumptions for RL, we have first recovered the convergence to stationary-point policies for such first-order optimization algorithms. Moreover, by virtue of the recent results in nonconvex optimization, we have proposed the modified  RPG algorithms by introducing periodically enlarged stepsizes, which are shown to be able to escape saddle points and converge to actual \emph{local optimal} policies under mild conditions that are satisfied for most modern reinforcement learning applications. Specifically, we have  given an optimization-based  explanation of why reward-reshaping is beneficial: it improves the curvature profile of the problem in neighborhoods of saddle points. On the inverted pendulum balancing task, we have experimentally corroborated our theoretical findings. Many enhancements are possible  for future research directions via the link between policy search and nonconvex optimization: rate improvements through acceleration, trust region methods, variance reduction, and Quasi-Newton methods.

%\input{Appendix_Lemmas}

%!TEX root = Policy_Grad.tex
%%%%%%%%%%%%%%%%%%%%%%%%%%%%%%%%%%%%%%%%%%%%%%%%%%%%%%%%%%%%%%%%%%%%%%%%%%%%%
%%%%%%%%%%%%%%%%%%%%%%%%%%%%%%%%%%%%%%%%%%%%%%%%%%%%%%%%%%%%%%%%%%%%%%%%%%%%%%%%%%%%%%%%%%%%%%%% S  E  C T  I  O  N%%%%%%%%%%%%%%%%%%%%%%%%%%%%%%%%%%%%%%%%%%%%%%%%%%%%%%%%%%%%%%%%%%%%%%%%%%%%%%%%%%%%%%%%%%%%%%%%%%%%%%%%%%%%%%%%%%%%%%%%%%%%%%%%%%%%%%%%%%%%%%%%%%%%%%%%%%%%%%%%%%%%%%%%%%%%%%%%%%%%%%%%%%%%%%%%%%%%%%%%%%%
\clearpage
\appendix 

\section{Detailed Proofs}\label{append:proof}
We  provide in this appendix the proofs of some  of the results stated  in the main body of the paper. 

\subsection{Proof of Lemma \ref{lemma:lip_policy_grad}}\label{apx_lemma:lip_policy_grad}
%%%%%%%%%%%%%%%%%%%%%%%%%%%%%%%%%%%%%%%%%%%%%%%%%%%%%%%%%%%%%%%%%%%%%%%%%%%%%
%%%%%%%%%%%%%%%%%%%%%%%%%%%%%%%%%%%%%%%%%%%%%%%%%%%%%%%%%%%%%%%%%%%%%%%%%%%%%%%%%%%%%%%%%%%%%%%%  P  R  O  O  F %%%%%%%%%%%%%%%%%%%%%%%%%%%%%%%%%%%%%%%%%%%%%%%%%%%%%%%%%%%%%%%%%%%%%%%%%%%%%%%%%%%%%%%%%%%%%%%%%%%%%%%%%%%%%%%%%%%%%%%%%%%%%%%%%%%%%%%%%%%%%%%%%%%%%%%%%%%%%%%%%%%%%%%%%%%%%%%%%%%%%%%%%%%%%%%%%%%%%%%%%%%
\begin{proof}
%\textcolor{green}{This proof is hard to follow, and we have some pretty messy expressions to deal with. Give me a little outline of the steps we are going to take before we proceed. Any time we have an equation reference, say what it is referencing.  For example: \\} 

The proof proceeds by expanding the expression of $\nabla J(\theta)$, and upper-bounding the norm $\|\nabla J(\theta^1)-\nabla J(\theta^2)\|$ for any $\theta^1,\theta^2\in\RR^d$ by multiples of $\|\theta^1-\theta^2\|$. To this end, we first substitute the  definition of $Q_{\pi_\theta}$ into the expression of policy gradient  in \eqref{equ:policy_grad_3}, which gives
	\#\label{eq:expanded_policy_gradient}
	\nabla J(\theta)=\sum_{t=0}^\infty\sum_{\tau=0}^\infty\gamma^{t+\tau}\cdot\int R(s_{t+\tau},a_{t+\tau})\cdot\nabla\log\pi_{\theta}(a_t\given s_t)\cdot p_{\theta,0:t+\tau}\cdot ds_{1:{t+\tau}}da_{0:{t+\tau}},
	\#
%	\textcolor{blue}{Also explain where does the double sum come from?} 
	where  for brevity we have introduced  
	\#\label{equ:def_trajec_dense}
	p_{\theta,0:t+\tau}=
	\bigg[\prod_{u=0}^{t+\tau-1}p(s_{u+1}\given s_u,a_u)\bigg]\cdot\bigg[\prod_{u=0}^{t+\tau}\pi_\theta(a_u\given s_u)\bigg]
	\#
	to represent the probability density of the trajectory $(s_0,a_0,\cdots,s_{t+\tau},a_{t+\tau})$.  Note that \eqref{equ:def_trajec_dense} follows from the Markov property of the trajectory. 
%	Moreover, we also define XX 	
	Hence, for any $\theta^1,\theta^2\in\RR^d$, we can analyze the difference of gradients through \eqref{eq:expanded_policy_gradient} as:
	\#\label{equ:lip_im_1}
	&\quad\|\nabla J(\theta^1)-\nabla J(\theta^2)\|\notag\\
	&=\bigg\|\sum_{t=0}^\infty\sum_{\tau=0}^\infty\gamma^{t+\tau}\cdot\bigg( \int R(s_{t+\tau},a_{t+\tau})\cdot\big\{\nabla\log\pi_{\theta^1}(a_t\given s_t)-\nabla\log\pi_{\theta^2}(a_t\given s_t)\big\}\cdot p_{\theta^1,0:t+\tau}\notag\\
	&\quad+\int R(s_{t+\tau},a_{t+\tau})\cdot\nabla\log\pi_{\theta^2}(a_t\given s_t)\cdot (p_{\theta^1,0:t+\tau}-p_{\theta^2,0:t+\tau})\bigg)ds_{1:{t+\tau}}da_{0:{t+\tau}}\bigg\|\notag\\
	&\leq \sum_{t=0}^\infty\sum_{\tau=0}^\infty \gamma^{t+\tau}\cdot \bigg\{\underbrace{\int \big|R(s_{t+\tau},a_{t+\tau})\big|\cdot \big\|\nabla\log\pi_{\theta^1}(a_t\given s_t)-\nabla\log\pi_{\theta^2}(a_t\given s_t)\big\|  {p}_{\theta^1,0:t+\tau}ds_{1:{t+\tau}}da_{0:{t+\tau}}}_{I_1}\notag\\
	&\quad +\underbrace{\int \big|R(s_{t+\tau},a_{t+\tau})\big|\cdot\big\|\nabla\log\pi_{\theta^2}(a_t\given s_t)\big\|\cdot \big|p_{\theta^1,0:t+\tau}-p_{\theta^2,0:t+\tau}\big|ds_{1:{t+\tau}}da_{0:{t+\tau}}\bigg\|}_{I_2}\bigg\},
	\#
	where the first equality comes from  adding and subtracting the term $\sum_{t=0}^\infty\sum_{\tau=0}^\infty\gamma^{t+\tau}\cdot\int R(s_{t+\tau},a_{t+\tau})\cdot\nabla\log\pi_{\theta^2}(a_t\given s_t)\cdot p_{\theta^1,0:t+\tau}$, and the  inequality follows from Cauchy-Schwarz inequality.
%	 to pull the norm inside the summand. 
The first term $I_1$ inside the summand on the right-hand side of \eqref{equ:lip_im_1} depends on a difference of score functions, whereas the second term $I_2$ depends on a difference between distributions induced by different policy parameters. We establish that both terms depend only on the norm of the difference between policy parameters. 
%Before doing so, for notational brevity, define the term $\check{p}_{\theta,0:t+\tau}$ in \eqref{equ:lip_im_1} as follows
%		\$
%	\check{p}_{\theta,0:t+\tau}=
%	\bigg[\prod_{u=0}^{t+\tau-1}p(s_{u+1}\given s_u,a_u)\bigg]\cdot\bigg[\prod_{u\neq t,u=0}^{t+\tau}\pi_\theta(a_u\given s_u)\bigg],
%	\$ 
%	which is defined by eliminating the term $\pi_\theta(a_t\given s_t)$ in  ${p}_{\theta,0:t+\tau}$. Note that $\check{p}_{\theta,0:t+\tau}$ satisfies that $\int \check{p}_{\theta,0:t+\tau} ds_{1:{t+\tau}}da_{0:{t+\tau}}=|\cA|$, since   $\check{p}_{\theta,0:t+\tau} \cdot {|\cA|}^{-1}$ is a valid probability density function whose full integral is $1$. 
%	By Cauchy-Schwarz inequality, we can further bound  \eqref{equ:lip_im_1} by 
%	\#\label{equ:lip_im_2}
%	&\quad\|\nabla J(\theta^1)-\nabla J(\theta^2)\|\notag\\
%	&\leq \sum_{t=0}^\infty\sum_{\tau=0}^\infty \gamma^{t+\tau}\cdot \bigg\{\underbrace{\int  |R(s_t,a_t)|\cdot\big\|\big\{\nabla\pi_{\theta^1}(a_t\given s_t)-\nabla\pi_{\theta^2}(a_t\given s_t)\big\}\big\|\cdot \check{p}_{\theta^1,0:t+\tau}ds_{1:{t+\tau}}da_{0:{t+\tau}}}_{I_1}\notag\\
%	&\quad +\underbrace{\int |R(s_t,a_t)|\cdot\big\|\nabla\log\pi_{\theta^2}(a_t\given s_t)\big\|\cdot \big|p_{\theta^1,0:t+\tau}-p_{\theta^2,0:t+\tau}\big|ds_{1:{t+\tau}}da_{0:{t+\tau}}}_{I_2}\bigg\},
%	\#

	By Assumption \ref{assum:regularity}, we have  $|R(s,a)|\leq U_R$ for any $(s,a)$, and 
	\$
	\|\nabla\log\pi_{\theta^1}(a_t\given s_t)-\nabla\log\pi_{\theta^2}(a_t\given s_t)\|\leq L_\Theta\cdot \|\theta^1-\theta^2\|.
	\$ Hence, we can bound the term $I_1$ in \eqref{equ:lip_im_1} as 
	\#\label{equ:lip_bnd_I_1}
	I_1\leq U_R\cdot L_\Theta\cdot \|\theta^1-\theta^2\|.
	\#
	\begin{comment}
	we define \textcolor{blue}{what inequality is used in the preceding expression, i.e., we are using the triangle inequality somewhere to replace a square of a sum by a sum of squares? Also, the left-hand side is not the simple substitution of the expansion \eqref{eq:expanded_policy_gradient} but there seems to be some grouping of terms going on. Like, where did the last term $+\int R(s_{t+\tau},a_{t+\tau})\cdot\nabla\log[\pi_{\theta^2}(a_t\given s_t)]\cdot (p_{\theta^1,0:t+\tau}-p_{\theta^2,0:t+\tau})\bigg)ds_{1:{t+\tau}}da_{0:{t+\tau}} $ come from in the middle equality? Explain which groupings/associations are happening.}

	\textcolor{blue}{We should explain the difference between $\check{p}_{\theta,0:t+\tau}$ and ${p}_{\theta,0:t+\tau}$, and why we have to define $\check{p}_{\theta,0:t+\tau}$.}
	Let $|\cA|$ denote the carnality of the action set $\cA$. 
	Then we have $\int \check{p}_{\theta,0:t+\tau} ds_{1:{t+\tau}}da_{0:{t+\tau}}=|\cA|$, since   $\check{p}_{\theta,0:t+\tau} \cdot {|\cA|}^{-1}$ is a valid probability density function whose full integral is $1$. 	
	Together with  $\|\nabla\pi_{\theta^1}(a_t\given s_t)-\nabla\pi_{\theta^2}(a_t\given s_t)\|\leq L_\Theta\cdot \|\theta^1-\theta^2\|$ in Assumption \ref{assum:regularity},  we can upper bound $I_1$ as
	$
	I_1\leq L_\Theta\cdot |\cA|\cdot \|\theta^1-\theta^2\|.
	$ \textcolor{blue}{Make this upper-estimate a display equation.}
	\end{comment}
	%
	To bound  the term $I_2$, let $\cU_{t+\tau}=\{u: u=0,\cdots,t+\tau\}$; we first have
	\#\label{equ:lip_im_I_2_0}
	p_{\theta^1,0:t+\tau}-p_{\theta^2,0:t+\tau}=\bigg[\prod_{u=0}^{t+\tau-1}p(s_{u+1}\given s_u,a_u)\bigg]\cdot\bigg[\prod_{u\in \cU_{t+\tau}}\pi_{\theta^1}(a_u\given s_u)-\prod_{u\in \cU_{t+\tau}}\pi_{\theta^2}(a_u\given s_u)\bigg]. 
	\#
	By Taylor expansion of $\prod_{u\in \cU_{t+\tau}}\pi_{\theta}(a_u\given s_u)$,  we have
	\#\label{equ:lip_im_I_2_1}
	&\Bigg|\prod_{u\in\cU_{t+\tau}}\pi_{\theta^1}(a_u\given s_u)-\prod_{u\in\cU_{t+\tau}}\pi_{\theta^2}(a_u\given s_u)\Bigg|=\Bigg|(\theta^1-\theta^2)^\top \bigg[\sum_{m\in\cU_{t+\tau}}\nabla \pi_{\tilde{\theta}}(a_m\given s_m)\prod_{u\in\cU_{t+\tau},u\neq m}\pi_{\ttheta}(a_u\given s_u)\bigg]\Bigg|\notag\\
	&\quad\leq \|\theta^1-\theta^2\| \cdot\sum_{m\in\cU_{t+\tau}}\|\nabla \log\pi_{\tilde{\theta}}(a_m\given s_m)\|\cdot \prod_{u\in\cU_{t+\tau}}\pi_{\ttheta}(a_u\given s_u)\notag\\
	&\quad\leq \|\theta^1-\theta^2\| \cdot (t+\tau+1)\cdot B_{\Theta}\cdot \prod_{u\in\cU_{t+\tau}}\pi_{\ttheta}(a_u\given s_u),
	\#
	where $\ttheta$ is a vector lying between $\theta^1$ and $\theta^2$, i.e., there exists some $\lambda\in[0,1]$ such that $\ttheta=\lambda \theta^1+(1-\lambda)\theta^2$. Therefore, we can  upper bound $|p_{\theta^1,0:t+\tau}-p_{\theta^2,0:t+\tau}|$ by substituting \eqref{equ:lip_im_I_2_1} into \eqref{equ:lip_im_I_2_0}, which further upper-bounds the term $I_2$  in \eqref{equ:lip_im_1} by 
	\#\label{equ:I_2_bnd}
	I_2&\leq \|\theta^1-\theta^2\| \cdot U_R\cdot B^2_{\Theta}\cdot \int\bigg[\prod_{u=0}^{t+\tau-1}p(s_{u+1}\given s_u,a_u)\bigg]\cdot (t+\tau+1)\cdot \prod_{u\in\cU_{t+\tau}}\pi_{\ttheta}(a_u\given s_u)ds_{1:{t+\tau}}da_{0:{t+\tau}}\notag\\
%	&= \|\theta^1-\theta^2\|\cdot U_R \cdot B^2_{\Theta}\cdot \sum_{u\in\cU_{t+\tau}} \int\bigg[\prod_{m=0}^{t+\tau-1}p(s_{u+1}\given s_u,a_u)\bigg]\cdot \prod_{u\in\cU_{t+\tau}}\pi_{\ttheta}(a_u\given s_u)ds_{1:{t+\tau}}da_{0:{t+\tau}}\notag\\
	&=  \|\theta^1-\theta^2\| \cdot U_R\cdot B^2_{\Theta}\cdot (t+\tau+1),
	\# 
	where the last equality follows from the fact that $\prod_{u=0}^{t+\tau-1}p(s_{u+1}\given s_u,a_u)\cdot \prod_{u\in\cU_{t+\tau}}\pi_{\ttheta}(a_u\given s_u)$ is a valid probability density function.
	
	Combining the bounds for $I_1$ in \eqref{equ:lip_bnd_I_1} and that for $I_2$ in \eqref{equ:I_2_bnd}, we have 
	\$
	\|\nabla J(\theta^1)-\nabla J(\theta^2)\|&\leq \sum_{t=0}^\infty\sum_{\tau=0}^\infty \gamma^{t+\tau}\cdot U_{R}\cdot \big[L_\Theta+B^2_{\Theta}\cdot (t+\tau+1) \big]\cdot \|\theta^1-\theta^2\|\\
	&\leq \bigg[\frac{1}{(1-\gamma)^2}\cdot U_{R}\cdot L_\Theta+\frac{1+\gamma}{(1-\gamma)^3}\cdot U_{R}\cdot B^2_{\Theta}\bigg]\cdot \|\theta^1-\theta^2\|,
	\$
	where the last inequality uses the expression for the limit of a geometric series, given that $\gamma\in(0,1)$:
	\$
	\sum_{t=0}^\infty\sum_{\tau=0}^\infty \gamma^{t+\tau}=\frac{1}{(1-\gamma)^2},\quad \sum_{t=0}^\infty\sum_{\tau=0}^\infty \gamma^{t+\tau}\cdot (t+\tau+1)=\frac{1+\gamma}{(1-\gamma)^3}.
	\$
	Hence, we define the Lipschitz constant $L$ as follows 
	\$
	L:=\frac{U_{R}\cdot L_\Theta}{(1-\gamma)^2} +\frac{(1+\gamma)\cdot U_R\cdot B^2_{\Theta}}{(1-\gamma)^3},
	\$
	which completes the proof.
\end{proof}

%%%%%%%%%%%%%%%%%%%%%%%%%%%%%%%%%%%%%%%%%%%%%%%%%%%%%%%%%%%%%%%%%%%%%%%%%%%%%
%%%%%%%%%%%%%%%%%%%%%%%%%%%%%%%%%%%%%%%%%%%%%%%%%%%%%%%%%%%%%%%%%%%%%%%%%%%%%%%%%%%%%%%%%%%%%%%%  P  R  O  O  F %%%%%%%%%%%%%%%%%%%%%%%%%%%%%%%%%%%%%%%%%%%%%%%%%%%%%%%%%%%%%%%%%%%%%%%%%%%%%%%%%%%%%%%%%%%%%%%%%%%%%%%%%%%%%%%%%%%%%%%%%%%%%%%%%%%%%%%%%%%%%%%%%%%%%%%%%%%%%%%%%%%%%%%%%%%%%%%%%%%%%%%%%%%%%%%%%%%%%%%%%%%
\subsection{Proof of Theorem \ref{thm:unbiased_and_bnd_grad_est}}\label{apx_thm:unbiased_and_bnd_grad_est}
\begin{proof}
We first establish  unbiasedness of the stochastic estimates of the policy gradient. 
	We start by showing   unbiasedness of the Q-estimate, i.e., for any $(s,a)\in\cS\times\cA$ and $\theta\in\RR^d$, $\EE[\hat{Q}_{\pi_\theta}(s,a)\given \theta,s,a]={Q}_{\pi_\theta}(s,a)$. In particular, from the  definition of $\hat{Q}_{\pi_{\theta}}(s,a)$, we have {
	\#\label{equ:grad_hat}
	\EE[\hat{Q}_{\pi_{\theta}}(s,a)\given \theta,s,a]&=\EE\bigg[\sum_{t=0}^{T'}\gamma^{t/2}\cdot R(s_t,a_t)\bigggiven \theta,s_0=s,a_0=a\bigg]\notag\\
	&=\EE\bigg[\sum_{t=0}^{\infty}\mathbbm{1}_{T'\geq t\geq 0}\cdot \gamma^{t/2}\cdot R(s_t,a_t)\bigggiven \theta,s_0=s,a_0=a\bigg]
	\#
	where we have replaced $T'$ by $\infty$ since we use the indicator function $\mathbbm{I}$ such that the summand for $t\geq T'$ is null.} 
%	Now, let's exchange the order of integration in \eqref{equ:grad_hat} for the product measure of $T',(s_{1:T'},a_{1:T'})$ by Fubini's Theorem:
%	%
%		\#\label{equ:grad_hat22}
%	\EE_{T',(s_{1:T'},a_{1:T'})}&\bigg[\sum_{t=0}^{\infty}\mathbbm{1}_{T'\geq t} R(s_t,a_t)\bigggiven \theta,s_0=s,a_0=a\bigg] \notag\\
%\quad=&	\EE_{(s_{1:T'},a_{1:T'}),T'}\bigg[\sum_{t=0}^{\infty}\mathbbm{1}_{T'\geq t} R(s_t,a_t)\bigggiven \theta,s_0=s,a_0=a\bigg] 
%	\#

	Now we show that the inner-expectation over $T'$ and summation in \eqref{equ:grad_hat} can be interchanged. In fact, by Assumption \ref{assum:regularity} regarding the boundedness of the reward, for any $N>0$, we have
	\#\label{equ:thm_21_im}
	\EE_{T'} \bigg(\bigg|\sum_{t=0}^N \mathbbm{1}_{0\leq t\leq T'}\cdot \gamma^{t/2}\cdot R_t\bigg|\bigg)\leq U_{R}\cdot \EE_{T'} \bigg(\sum_{t=0}^N \mathbbm{1}_{0\leq t\leq T'}\cdot \gamma^{t/2}\bigg).
	\#
	Note that on the right-hand side of \eqref{equ:thm_21_im}, the random variable in the expectation is monotonically increasing and the limit as $N\to\infty$ exists. Thus, 
%	we have
%	\$
%	\EE_{T'}\bigg(\sum_{t=0}^N \mathbbm{1}_{t\leq T'}\cdot \gamma^{t/2}\bigg)\leq \EE_{T'} \bigg(\sum_{t=0}^\infty \mathbbm{1}_{t\leq T'}\cdot \gamma^{t/2}\bigg)=\frac{1}{1-\gamma},
%	\$
%	where we use the fact that $\PP({t\leq T'})=\gamma^{t/2}$ as $T'\sim\tx{Geom}(1-\gamma^{t/2})$. 
%	Moreover, since $\big|\sum_{t=0}^N \mathbbm{1}_{t\geq T'}\cdot \gamma^{t/2}\cdot R_t\big|$ is dominated by $ \big|\sum_{t=0}^N \mathbbm{1}_{t\geq T'}\cdot\gamma^{t/2}\big|\cdot U_{R}$ for any $N$, 
	by the  Monotone Convergence Theorem \cite{yeh2006real},  we can  interchange the limit with the integral, i.e., the sum and inner-expectation in \eqref{equ:grad_hat} as follows
	\#\label{equ:Q_unbiased22}
	\EE\bigg\{\bigg[&\sum_{t=0}^{\infty}\mathbbm{1}_{T'\geq t\geq 0} \cdot \gamma^{t/2}\cdot R(s_t,a_t)\bigggiven \theta,s_0=s,a_0=a\bigg]\bigg\}\notag\\
	&=\sum_{t=0}^{\infty}\EE\bigg[\EE_{T'}(\mathbbm{1}_{T'\geq t\geq 0})\cdot \gamma^{t/2}\cdot R(s_t,a_t)\given \theta,s_0=s,a_0=a\bigg]\notag\\
	& =\sum_{t=0}^{\infty}\EE\bigg[ \gamma^{t}\cdot R(s_t,a_t)\given \theta,s_0=s,a_0=a\bigg],
	\#
	where we have also used in the first equality  the fact that $T'$ is drawn independently of the system evolution $(s_{1:T'},a_{1:T'})$, and d in the second  equality  the fact that $T'\sim\tx{Geom}(1-\gamma^{1/2})$ and thus $\EE_{T'}(\mathbbm{1}_{T'\geq t\geq 0}) = \mathbb{P}(T' \geq t\geq 0) = \gamma^{t/2}$  in the second equality.
	Furthermore, since $|\sum_{t=0}^N\gamma^t R(s_t,a_t)|\leq \sum_{t=0}^N\gamma^t U_R$, and $\lim_{N\to\infty}\EE(\sum_{t=0}^N\gamma^t U_R)$ exists, by the Dominated   Convergence Theorem \cite{bartle2014elements}, the right-hand side of \eqref{equ:Q_unbiased22} can be written as
	\$
	\sum_{t=0}^{\infty}\EE\bigg[ \gamma^{t}\cdot R(s_t,a_t)\given \theta,s_0=s,a_0=a\bigg]=\EE\bigg[\sum_{t=0}^{\infty}\gamma^{t}\cdot R(s_t,a_t)\given \theta,s_0=s,a_0=a\bigg]=Q_{\pi_\theta}(s,a),
	\$
	which completes the proof of the  unbiasedness of $\hat{Q}_{\pi_{\theta}}(s,a)$.

	Similar logic allows us to establish that $\hat{V}_{\pi_\theta}(s)$ is an unbiased estimate  of ${V}_{\pi_\theta}(s)$, i.e., for any  $s\in\cS$ and $\theta\in\RR^d$, 
	\$ 
	\EE\bigg[\sum_{t=0}^{\infty}\mathbbm{1}_{T'\geq t\geq 0} \cdot \gamma^{t/2}\cdot  R(s_t,a_t)\bigggiven \theta,s_0=s\bigg]=\EE[\hat{V}_{\pi_\theta}(s)\given \theta,s]={V}_{\pi_\theta}(s),
	\$
	where the expectation is taken along the trajectory as well as with respect to  the random horizon  $T'\sim\tx{Geom}(1-\gamma^{t/2})$. 
	Therefore, if $s'\sim \PP(\cdot\given s,a)$ and $a'\sim \pi_\theta(\cdot\given s')$,  we have
	\#\label{equ:A_unbiased}
	\EE[\hat{Q}_{\pi_\theta}(s,a)-\hat{V}_{\pi_\theta}(s)\given \theta,s,a]=\EE[R(s,a)+\gamma\hat{V}_{\pi_\theta}(s')-\hat{V}_{\pi_\theta}(s)\given \theta,s,a]=A_{\pi_\theta}(s,a).
	\#
	That is, $\hat{Q}_{\pi_\theta}(s,a)-\hat{V}_{\pi_\theta}(s)$ and $R(s,a)+\gamma\hat{V}_{\pi_\theta}(s')-\hat{V}_{\pi_\theta}(s)$ are both unbiased estimates of the advantage function $A_{\pi_\theta}(s,a)$.

	Now we are ready to show  unbiasedness of the stochastic gradients $\hat{\nabla} J(\theta), \check{\nabla} J(\theta)$, and $\tilde{\nabla} J(\theta)$. First for $\hat{\nabla} J(\theta)$, we have from \eqref{equ:Q_unbiased22} that
	\#\label{equ:RPG_unbiased}
	&\EE[\hat{\nabla} J(\theta)\given \theta]=\EE_{T,(s_T,a_T)}\big\{\EE_{T',(s_{1:T'},a_{1:T'})}[\hat{\nabla} J(\theta)\given \theta,s_T=s,a_T=a]\biggiven \theta\big\}\notag\\
	&\quad =\EE_{T,(s_T,a_T)}\bigg(\EE_{T',(s_{1:T'},a_{1:T'})}\bigg\{\frac{1}{1-\gamma}\cdot \hat{Q}_{\pi_\theta}(s_T,a_T)\cdot \nabla\log[\pi_{\theta}(a_T\given s_T)]\biggiven \theta,s_T=s,a_T=a\bigg\}\bigggiven \theta\bigg)\notag\\
	&\quad =\EE_{T,(s_T,a_T)}\bigg\{\frac{1}{1-\gamma}\cdot {Q}_{\pi_\theta}(s_T,a_T)\cdot \nabla\log[\pi_{\theta}(a_T\given s_T)]\bigggiven \theta\bigg\}.
	\#
	By using the identity function $\mathbbm{1}_{t=T}$, \eqref{equ:RPG_unbiased} can be further written as
	\#\label{equ:RPG_exchange_0}
	\EE[\hat{\nabla} J(\theta)\given \theta]=\frac{1}{1-\gamma}\cdot \EE_{T,(s_T,a_T)}\bigg\{\sum_{t=0}^\infty\mathbbm{1}_{t=T}\cdot {Q}_{\pi_\theta}(s_t,a_t)\cdot \nabla\log[\pi_{\theta}(a_t\given s_t)]\bigggiven \theta\bigg\}.
	\#
	Note that by Assumption   \ref{assum:regularity}, $\|\hat{\nabla} J(\theta)\|$ is directly bounded by $(1-\gamma)^{-2}\cdot U_{R}\cdot B_\Theta$, since there is only one nonzero term in the summation in \eqref{equ:RPG_exchange_0}. Thus, by the Dominated Convergence Theorem, we can interchange the  summation and expectation in \eqref{equ:RPG_exchange_0} and obtain 
	\#
	&\EE[\hat{\nabla} J(\theta)\given \theta]=\sum_{t=0}^\infty\frac{P(t=T)}{1-\gamma}\cdot\EE\bigg\{{Q}_{\pi_\theta}(s_t,a_t)\cdot \nabla\log[\pi_{\theta}(a_t\given s_t)]\bigggiven \theta\bigg\}\notag\\
	&\quad=\sum_{t=0}^\infty\gamma^t\cdot\EE\bigg\{{Q}_{\pi_\theta}(s_t,a_t)\cdot \nabla\log[\pi_{\theta}(a_t\given s_t)]\bigggiven \theta\bigg\}\label{equ:RPG_exchange_2}\\
	&\quad=\sum_{t=0}^\infty\gamma^t\cdot\int_{s\in\cS,a\in\cA}p(s_t=s,a_t=a\given s_0,\pi_\theta)\cdot{Q}_{\pi_\theta}(s,a)\cdot \nabla\log[\pi_{\theta}(a\given s)] dsda,\label{equ:RPG_exchange_3}
	\#
	where \eqref{equ:RPG_exchange_2} is due to the fact that $T\sim\tx{Geom}(1-\gamma)$ and thus  $P(t=T)=(1-\gamma)\gamma^t$, and in \eqref{equ:RPG_exchange_3} we define $p(s_t=s,a_t=a\given s_0,\pi_\theta)=p(s_t=s\given s_0,\pi_\theta)\cdot \pi_\theta(a_t=a\given s_t)$, with  $p(s_t=s\given s_0,\pi_\theta)$ being the probability of state $s_t=s$ given initial state $s_0$ and policy $\pi_\theta$.
 	By  the	Dominated Convergence Theorem, we can further re-write \eqref{equ:RPG_exchange_3} by interchanging the summation and the integral, i.e., 
	\#
	\EE[\hat{\nabla} J(\theta)\given \theta]&=\int_{s\in\cS,a\in\cA}\sum_{t=0}^\infty\gamma^t\cdot p(s_t=s\given s_0,\pi_\theta)\cdot{Q}_{\pi_\theta}(s,a)\cdot \nabla \pi_{\theta}(a\given s) dsda.\label{equ:RPG_exchange_4}
	\#
	Note that the expression in \eqref{equ:RPG_exchange_4} coincides with the policy gradient given in \eqref{equ:policy_grad_1}, 	which completes the proof of   unbiasedness of $\hat{\nabla} J(\theta)$.

	For $\check{\nabla} J(\theta)$, we have the following identity similar to \eqref{equ:RPG_unbiased}: 
		\#\label{equ:RPG_unbiased_2}
	&\EE[\check{\nabla} J(\theta)\given \theta]=\EE_{T,(s_T,a_T)}\big\{\EE_{T',(s_{1:T'},a_{1:T'})}[\check{\nabla} J(\theta)\given \theta,s_T=s,a_T=a]\biggiven \theta\big\}\notag\\
	&\quad =\EE_{T,(s_T,a_T)}\bigg(\EE_{T',(s_{1:T'},a_{1:T'})}\bigg\{\frac{\hat{Q}_{\pi_\theta}(s,a)-\hat{V}_{\pi_\theta}(s)}{1-\gamma}\cdot \nabla\log[\pi_{\theta}(a_T\given s_T)]\biggiven \theta,s_T=s,a_T=a\bigg\}\bigggiven \theta\bigg)\notag\\
	&\quad =\EE_{T,(s_T,a_T)}\bigg\{\frac{1}{1-\gamma}\cdot {A}_{\pi_\theta}(s_T,a_T)\cdot \nabla\log[\pi_{\theta}(a_T\given s_T)]\bigggiven \theta\bigg\}.
	\#
	From the definition of ${A}_{\pi_\theta}(s_T,a_T)={Q}_{\pi_\theta}(s_T,a_T)-{V}_{\pi_\theta}(s_T)$, \eqref{equ:RPG_unbiased_2} further implies
		\#
	&\EE[\check{\nabla} J(\theta)\given \theta]=\int_{s\in\cS,a\in\cA}\sum_{t=0}^\infty\gamma^t\cdot p(s_t=s\given s_0,\pi_\theta)\cdot[{Q}_{\pi_\theta}(s,a)-{V}_{\pi_\theta}(s)]\cdot \nabla \pi_{\theta}(a\given s) dsda,\label{equ:RPG_exchange_2_2}
	\#
	which follows from  similar  arguments as   in \eqref{equ:RPG_exchange_0}-\eqref{equ:RPG_exchange_4}. Note that \eqref{equ:RPG_exchange_2_2} also coincides with the policy gradient given in \eqref{equ:policy_grad_1}, since $\int_{a\in\cA}{V}_{\pi_\theta}(s)\cdot\nabla \pi_{\theta}(a\given s) da=0$. Similar arguments also hold for the stochastic policy gradient $\tilde{\nabla} J(\theta)$, since \eqref{equ:RPG_unbiased_2} can also be obtained from $\EE[\tilde{\nabla} J(\theta)\given \theta]$. This proves    unbiasedness of $\check{\nabla} J(\theta)$ and $\tilde{\nabla} J(\theta)$.
	 
	{
	Now we  establish  almost sure  boundedness of the stochastic policy gradients $\hat{\nabla}J(\theta),\check{\nabla}J(\theta)$, and $\tilde{\nabla}J(\theta)$. In particular, from the definition of $\hat{\nabla}J(\theta)$ in \eqref{equ:SGD_eva}
	\#
	\|\hat{\nabla}J(\theta)\|&=\bigg\|\frac{1}{1-\gamma}\cdot \hat{Q}_{\pi_\theta}(s_T,a_T)\cdot \nabla\log[\pi_{\theta}(a_T\given s_T)]\bigg\|\leq \frac{B_{\Theta}}{1-\gamma}\bigg|\sum_{t=0}^{T'}\gamma^{t/2}\cdot R(s_t,a_t)\bigg|\notag\\
	&\leq  \frac{B_{\Theta}}{1-\gamma}\sum_{t=0}^{T'}\gamma^{t/2}\cdot U_R\leq \frac{B_{\Theta}}{1-\gamma}\sum_{t=0}^{\infty}\gamma^{t/2}\cdot U_R=\frac{B_{\Theta}U_R}{(1-\gamma)(1-\gamma^{1/2})}=:\hat{\ell},\label{equ:def_B_J_1}
	\#
	where we have used Assumption \ref{assum:regularity}, namely that $|R(s,a)|\leq U_R$ and $\|\nabla\log\pi_{\theta}(a\given s)\|\leq B_{\Theta}$ for any $s,a$ and $\theta$. Similarly, we arrive at  the following bounds	
		\#
	&\|\check{\nabla}J(\theta)\|
	\leq  \frac{2B_{\Theta}}{1-\gamma}\sum_{t=0}^{\infty}\gamma^{t/2}\cdot U_R=\frac{2B_{\Theta}U_R}{(1-\gamma)(1-\gamma^{1/2})}=:\check{\ell},\label{equ:def_B_J_2}\\
	&\|\tilde{\nabla}J(\theta)\|
	\leq  \frac{B_{\Theta}}{1-\gamma}\bigg[1+\bigg(\gamma+1\bigg)\bigg(\sum_{t=0}^{\infty}\gamma^{t/2}\bigg)\bigg]\cdot U_R\leq \frac{(2+\gamma-\gamma^{1/2})B_{\Theta}U_R}{(1-\gamma)(1-\gamma^{1/2})}=:\tilde{\ell},\label{equ:def_B_J_3}
	\#
	which completes the proof. 
	}

\end{proof}

\subsection{Proof of Theorem \ref{thm:alg1_as_conv}}\label{append:proof:thm:alg1_as_conv}
\begin{proof}
Recall that the  policy gradient method  follows \eqref{equ:RPG_update}. At each iteration $k$, we define the random horizon used in estimating $\hat{Q}_{\pi_{\theta_k}}(s_{T_{k+1}},a_{T_{k+1}})$ in the inner-loop of Algorithm \ref{alg:est_Q} as $T'_{k+1}$. 
We then introduce a probability measure space $(\Omega,\cF,P)$ and 
 let $\{\cF_k\}_{k\geq 0}$ denote a sequence  of increasing sigma-algebras  $\cF_0\subset\cF_1\subset\cdots\cF_{\infty}\subset\cF$, where  
 \$
 \cF_k=\sigma\big(\{\theta_\tau\}_{\tau=0:k},\{T_\tau\}_{\tau=0:k},\{(s_\tau,a_\tau)\}_{\tau=T_0:T_k},\{T'_\tau\}_{\tau=0:k},\big\{\{(s_\tau,a_\tau)\}_{\tau=0:T'_0},\cdots,\{(s_\tau,a_\tau)\}_{\tau=0:T'_k}\big\}\big).
 \$ %\textcolor{green}{should $\mathcal{F}_k$ also include realizations of $s_t, a_t$?}
   We also define the following auxiliary random variable $W_k$, which is essential to the analysis of Algorithm \ref{alg:RPG}
	\#\label{eq:submartingale_seq_definition}
	W_k=J(\theta_k)-L \hat{\ell}^2\sum_{j=k}^\infty\alpha_k^2,
	\# 
	where we recall that $L$ is the Lipchitz constant of $\nabla J(\theta)$ as defined in \eqref{equ:L_Theta_def}, and $\hat{\ell}$ is the upper bound of $\|\hat{\nabla}J(\theta_k)\|$ in Theorem \ref{thm:unbiased_and_bnd_grad_est}.   
	 Noting that $J(\theta)$ is bounded and $\{\alpha_k\}$ is square-summable, we conclude that $W_k$ is bounded for any $k\geq 0$. In fact, we can show that $\{W_k\}$ is a bounded submartingale, as stated in the following lemma.
 %%%%%%%%%%%%%%%%%%%%%%%%%%%%%%%%%%%%%%%%%%%%%%%%%%%%%%%%%%%%%%%%%%%%%%%%%%%%%
%%%%%%%%%%%%%%%%%%%%%%%%%%%%%%%%%%%%%%%%%%%%%%%%%%%%%%%%%%%%%%%%%%%%%%%%%%%%%%%%%%%%%%%%%%%%%%%%L	E	M	M	A%%%%%%%%%%%%%%%%%%%%%%%%%%%%%%%%%%%%%%%%%%%%%%%%%%%%%%%%%%%%%%%%%%%%%%%%%%%%%%%%%%%%%%%%%%%%%%%%%%%%%%%%%%%%%%%%%%%%%%%%%%%%%%%%%%%%%%%%%%%%%%%%%%%%%%%%%%%%%%%%%%%%%%%%%%%%%%%%%%%%%%%%%%%%%%%%%%%%%%%%%%%
\begin{lemma}\label{lemma:submartingale}
	The objective function sequence defined by Algorithm \ref{alg:RPG}  satisfies the following stochastic ascent property:
	\#
	\EE[J(\theta_{k+1})\given \cF_k]\geq J(\theta_k)+\EE[(\theta_{k+1}-\theta_{k})\given \cF_k]^\top \nabla J(\theta_k)-L\alpha_k^2\hat{\ell}^2\label{equ:ascent_lemma}
	\#
	Moreover, the sequence $\{W_k\}$ defined in \eqref{eq:submartingale_seq_definition} is a bounded submartingale. 
	\#
	\EE(W_{k+1}\given \cF_k)\geq W_k+\alpha_k\|\nabla J(\theta_k)\|^2.\label{equ:submartingale}
	\#
\end{lemma}	
%%%%%%%%%%%%%%%%%%%%%%%%%%%%%%%%%%%%%%%%%%%%%%%%%%%%%%%%%%%%%%%%%%%%%%%%%%%%%
%%%%%%%%%%%%%%%%%%%%%%%%%%%%%%%%%%%%%%%%%%%%%%%%%%%%%%%%%%%%%%%%%%%%%%%%%%%%%%%%%%%%%%%%%%%%%%%%  P  R  O  O  F %%%%%%%%%%%%%%%%%%%%%%%%%%%%%%%%%%%%%%%%%%%%%%%%%%%%%%%%%%%%%%%%%%%%%%%%%%%%%%%%%%%%%%%%%%%%%%%%%%%%%%%%%%%%%%%%%%%%%%%%%%%%%%%%%%%%%%%%%%%%%%%%%%%%%%%%%%%%%%%%%%%%%%%%%%%%%%%%%%%%%%%%%%%%%%%%%%%%%%%%%%%
\begin{proof}
	Note that $W_k$  is adapted to the sigma-algebra $\cF_k$. Consider the first-order Taylor expansion of $J(\theta_{k+1})$ at $\theta_k$. Then there exists some	 $\ttheta_k=\lambda \theta_k+(1-\lambda)\theta_{k+1}$ for some $\lambda\in[0,1]$ such that $W_{k+1}$ can be written as  
%	\textcolor{blue}{Why is the first $J(\theta)$ in the following expression $J(\theta_k)$ and not $J(\ttheta_k)$? In the preceding sentence, we explicitly say we are evaluating $J(\theta)$ at $\ttheta_k$ so it stands to reason that $J(\ttheta_k)$ should appear somewhere. What is going on here? This is odd or a mistake. \emph{It needs to be explained in specific detail.} It doesn't pass first scrutiny of a read.}
	
%%	 By Taylor expansion, let $\ttheta_k=\lambda \theta_k+(1-\lambda)\theta_{k+1}$ for some $\lambda\in[0,1]$, we have
	\$
	W_{k+1}
	&=J(\theta_k)+(\theta_{k+1}-\theta_{k})^\top \nabla J(\ttheta_k)-L \hat{\ell}^2\sum_{j=k+1}^\infty \alpha_k^2\\
	&= J(\theta_k)+(\theta_{k+1}-\theta_{k})^\top \nabla J(\theta_k)+(\theta_{k+1}-\theta_{k})^\top [\nabla J(\ttheta_k)-\nabla J(\theta_k)]-L \hat{\ell}^2\sum_{j=k+1}^\infty \alpha_k^2\\
	&\geq J(\theta_k)+(\theta_{k+1}-\theta_{k})^\top \nabla J(\theta_k)-L\|\theta_{k+1}-\theta_{k}\|^2-L \hat{\ell}^2\sum_{j=k+1}^\infty \alpha_k^2
	\$
	where the {second} equality comes from adding and subtracting $(\theta_{k+1}-\theta_{k})^\top \nabla J(\theta_k)$, and the inequality follows from applying  Lipschitz continuity of the gradient (Lemma \ref{lemma:lip_policy_grad}), {i.e. 
	\$
	&(\theta_{k+1}-\theta_{k})^\top [\nabla J(\ttheta_k)-\nabla J(\theta_k)]\geq -\|\theta_{k+1}-\theta_{k}\|\cdot \|\nabla J(\ttheta_k)-\nabla J(\theta_k)\|\geq -\|\theta_{k+1}-\theta_{k}\|\cdot L \|\ttheta_k-\theta_k\|\\
	&\quad =-\|\theta_{k+1}-\theta_{k}\|\cdot L (1-\lambda)\cdot\|\theta_{k+1}-\theta_k\|\geq -L\cdot \|\theta_{k+1}-\theta_{k}\|^2,
	\$
	}with the constant $L$ being defined in \eqref{equ:L_Theta_def}. By taking conditional expectation over $\cF_k$ on both sides, we further obtain
	\#
	\EE[W_{k+1}\given \cF_k]&\geq J(\theta_k)+\EE[(\theta_{k+1}-\theta_{k})\given \cF_k]^\top \nabla J(\theta_k)-L\EE(\|\theta_{k+1}-\theta_{k}\|^2\given \cF_k)-L \hat{\ell}^2\sum_{j=k+1}^\infty \alpha_k^2\notag\\
	&=J(\theta_k)+\EE[(\theta_{k+1}-\theta_{k})\given \cF_k]^\top \nabla J(\theta_k)-L\alpha_k^2\EE(\|\hat{\nabla}J(\theta_k)\|^2\given \cF_k)-L \hat{\ell}^2\sum_{j=k+1}^\infty \alpha_k^2\notag\\
	&\geq J(\theta_k)+\EE[(\theta_{k+1}-\theta_{k})\given \cF_k]^\top \nabla J(\theta_k)-L\alpha_k^2\hat{\ell}^2-L \hat{\ell}^2\sum_{j=k+1}^\infty \alpha_k^2,\label{equ:Wk_iter}
	\#
	where the first inequality comes from substituting $\theta_{k+1}-\theta_{k} = \alpha_k \hat{\nabla}J(\theta_k)$ and the second one uses the fact that $\mathbb{E}[\|\hat{\nabla}J(\theta_k)\|^2 ] \leq \hat{\ell}^2$. 
	By definition of $J(\theta)$, we have
	\$
	\EE[J(\theta_{k+1})\given \cF_k]&\geq J(\theta_k)+\EE[(\theta_{k+1}-\theta_{k})\given \cF_k]^\top \nabla J(\theta_k)-L\alpha_k^2\hat{\ell}^2,
	\$
	which  establishes the first argument of the lemma.

	In addition, note that 
	\$
	\EE[(\theta_{k+1}-\theta_{k})\given \cF_k]=\alpha_k\EE(\hat{\nabla}J(\theta_k)\given \cF_k)=\alpha_k\nabla J(\theta_k),
	\$
	which we may substitute into the right-hand side of \eqref{equ:Wk_iter}, and upper-bound the negative constant terms by null to obtain
	\$
	\EE(W_{k+1}\given \cF_k)&\geq W_k+ \alpha_k\|\nabla J(\theta_k)\|^2.
	\$
	This concludes the proof.
\end{proof}

Now we are in a position to show that $\|\nabla J(\theta_k)\|$ converges to zero as $k\to\infty$. In particular, by definition, we have the boundedness of $W_k$, i.e., $W_k\leq J^*$, where $J^*$ is the global maximum of $J(\theta)$. Thus, \eqref{equ:submartingale} can be written as 
\$
\EE(J^*-W_{k+1}\given \cF_k)\leq (J^*-W_k)-\alpha_k\|\nabla J(\theta_k)\|^2,
\$
where $\{J^*-W_{k}\}$ is a nonnegative sequence of random variables. By applying the  supermartingale convergence theorem \cite{robbins1985convergence}, we have
\#\label{equ:sum_bnd_almost_sure}
\sum_{k=1}^\infty\alpha_k\|\nabla J(\theta_k)\|^2<\infty,~~\text{a.s.}. 
\#

Note that by Assumption \ref{assum:RM_Cond}, the stepsize $\{\alpha_k\}$ is non-summable. Therefore, the only way that \eqref{equ:sum_bnd_almost_sure} may be valid is if the following holds:
\#\label{equ:liminf_bnd}
\liminf_{k\to\infty}~\|\nabla J(\theta_k)\|=0.
\#

%\textcolor{green}{It would help the reader to ``preview" the logic that's going to come next in deriving the contradiction, just to help keep it straight. It gets quite technical and down in the weeds here.}
%\textcolor{cyan}{
From here, we proceed to show that $\limsup_{k\to\infty}\|\nabla J(\theta_k)\|=0$ by contradiction. To this end, we construct a  sequence of $\{\theta_k\}$ that has two sub-sequences lying in two disjoint sets. We aim to establish a contradiction on the sum of the distances between the points in the two sets. 
Specifically,  suppose that  for some random realization $\omega\in\Omega$,   we have% }
\#\label{equ:limsup_bnd}
\limsup_{k\to\infty}\|\nabla J(\theta_k)\|={\epsilon}>0.
\#
Then  it must hold that $\|\nabla J(\theta_k)\|\geq 2\epsilon/3$ for infinitely many $k$. Moreover, \eqref{equ:liminf_bnd} implies  that $\|\nabla J(\theta_k)\|\leq \epsilon/3$ for infinitely many $k$. We thus can define the following sets $\cN_1$ and $\cN_2$ as
\$
\cN_1=\{\theta_k: \|\nabla J(\theta_k)\|\geq 2\epsilon/3\},\quad \cN_2=\{\theta_k: \|\nabla J(\theta_k)\|\leq \epsilon/3\}.
\$
Note that since $\|\nabla J(\theta)\|$ is continuous by Lemma \ref{lemma:lip_policy_grad}, 
both sets are closed  in the 
Euclidean space. We define the distance between  the two sets as
\$
D(\cN_1,\cN_2)=\inf_{\theta^1\in\cN_1}\inf_{\theta^2\in\cN_2}\|\theta^1-\theta^2\|.
\$
Then $D(\cN_1,\cN_2)$ must be a positive number since the sets $\cN_1$ and $\cN_2$ are disjoint and closed. Moreover, since both $\cN_1$ and $\cN_2$ are infinite sets, there exists an index set $\cI$ such that the subsequence $\{\theta_k\}_{k\in\cI}$ of $\{\theta_k\}_{k\geq 0}$ crosses the two sets infinitely often. In particular, there exist two sequences of indices $\{s_i\}_{i\geq 0}$ and $\{t_i\}_{i\geq 0}$ such that 
%\textcolor{blue}{Should this be indices $s_i+1,\dots, t_i-1$? Originally the sum in \eqref{equ:Dis_inf} was from $s_i$ to $t_i-1$, rather than $s_i+ 1$. I corrected it but maybe I misunderstood. Please clarify.}
\$
\{\theta_k\}_{k\in\cI}=\{\theta_{s_i},\cdots,\theta_{t_i-1}\}_{i\geq 0},
\$
with $
\{\theta_{s_i}\}_{i\geq 0}\subseteq \cN_1, \{\theta_{t_i}\}_{i\geq 0}\subseteq \cN_2
$, 
and for any indices $k=s_i+1,\cdots,t_i-1 \in \cI$ (not including $s_i$) in between the indices $\{s_i\}$ and $\{t_i\}$, we have
\$
\frac{\epsilon}{3}\leq \|\nabla J(\theta_k)\|\leq \frac{2\epsilon}{3}\leq \|\nabla J(\theta_{s_i})\|.
\$
Setting aside this expression for now, let us analyze the norm-difference of iterates $\theta_{k}$ associated with indices in $\cI$. By the triangle inequality, we may write
\#\label{equ:Dis_inf}
\sum_{k\in\cI}\|\theta_{k+1}-\theta_k\|=\sum_{i=0}^\infty\sum_{k=s_i}^{t_i-1}\|\theta_{k+1}-\theta_k\|\geq \sum_{i=0}^\infty \|\theta_{s_i}-\theta_{t_i}\|\geq \sum_{i=0}^\infty D(\cN_1,\cN_2)=\infty.
\#
Moreover,  \eqref{equ:sum_bnd_almost_sure} implies that
\$
\infty>\sum_{k\in\cI}\alpha_k\|\nabla J(\theta_k)\|^2\geq \sum_{k\in\cI}\alpha_k\cdot\frac{\epsilon^2}{9},
\$
using the definition of $\epsilon$ in \eqref{equ:limsup_bnd}. We may therefore conclude that $\sum_{k\in\cI}\alpha_k<\infty$.
Also from Theorem \ref{thm:unbiased_and_bnd_grad_est}, we have that the stochastic policy gradient has a  finite  first  moment: $\EE(\|\hat{\nabla} J(\theta_k)\|)<\infty$. Taken together, we therefore have 
\$
\sum_{k\in\cI}\EE(\|\theta_{k+1}-\theta_{k}\|)=\sum_{k\in\cI}\alpha_k\EE(\|\hat{\nabla} J(\theta_k)\|)<\infty. 
\$
The monotone convergence theorem then implies that $\sum_{k\in\cI}\|\theta_{k+1}-\theta_k\|<\infty$ almost surely, which contradicts \eqref{equ:Dis_inf}. Therefore, \eqref{equ:Dis_inf} must be false, which implies that the hypothesis that the limsup is bounded away from zero, as in \eqref{equ:limsup_bnd}, is invalid. As a consequence, its negation must be true: the set of sample paths for which this condition holds has measure zero. This allows us to conclude
\$
\limsup_{k\to\infty}\|\nabla J(\theta_k)\|=0, ~~a.s.
\$
This statement together with \eqref{equ:liminf_bnd} allows us to conclude that $\lim_{k\to\infty}\|\nabla J(\theta_k)\|=0$ a.s., which completes the proof.
\end{proof}

%%%%%%%%%%%%%%%%%%%%%%%%%%%%%%%%%%%%%%%%%%%%%%%%%%%%%%%%%%%%%%%%%%%%%%%%%%%%%
%%%%%%%%%%%%%%%%%%%%%%%%%%%%%%%%%%%%%%%%%%%%%%%%%%%%%%%%%%%%%%%%%%%%%%%%%%%%%%%%%%%%%%%%%%%%%%%%  P  R  O  O  F %%%%%%%%%%%%%%%%%%%%%%%%%%%%%%%%%%%%%%%%%%%%%%%%%%%%%%%%%%%%%%%%%%%%%%%%%%%%%%%%%%%%%%%%%%%%%%%%%%%%%%%%%%%%%%%%%%%%%%%%%%%%%%%%%%%%%%%%%%%%%%%%%%%%%%%%%%%%%%%%%%%%%%%%%%%%%%%%%%%%%%%%%%%%%%%%%%%%%%%%%%%
\subsection{Proofs of Theorem \ref{thm:alg1_conv_rate} and Corollary \ref{coro:alg1_conv_rate_const_step}}\label{apx_thm:alg1_conv_rate}
\begin{proof}
	By the stochastic ascent property, i.e., \eqref{equ:ascent_lemma} in  Lemma \ref{lemma:submartingale}, we can write %\textcolor{green}{I think the statement that we use from \eqref{equ:Wk_iter} should be stated as an additional part of Lemma \ref{lemma:submartingale}. Part of the lemma can be used for diminishing stepsize proof and part can be used for constant stepsize here. That'll make it easier/more understandable when it's referred to here.}
	\#
	\EE[J(\theta_{k+1})\given \cF_k]&\geq J(\theta_k)+\EE[(\theta_{k+1}-\theta_{k})\given \cF_k]^\top \nabla J(\theta_k)-L\alpha_k^2\hat{\ell}^2\notag\\
	&=J(\theta_k)+\alpha_k\|\nabla J(\theta_k)\|^2-L\alpha_k^2\hat{\ell}^2.\label{equ:Jk_iter}
	\#
	Let $U(\theta)=J^*-J(\theta)$, where $J^*$ is the global optimum\footnote{Such an optimum is assumed to always exist for the parameterization $\pi_\theta$.} of $J(\theta)$. Then, we immediately have $0\leq U(\theta)\leq 2U_R/(1-\gamma)$ since $|J(\theta)|\leq U_R/(1-\gamma)$  for any $\theta$. Moreover,  we may write \eqref{equ:Jk_iter} as 
	\#
	\EE[U(\theta_{k+1})\given \cF_k]\leq U(\theta_k)-\alpha_k\|\nabla J(\theta_k)\|^2+L\alpha_k^2\hat{\ell}^2.\label{equ:Uk_iter}
	\#
	%
%	\textcolor{blue}{The conditioning on $s_0$ starts to get quite confusing in the subsequent analysis. Sometimes we condition on $s_T, a_T$ and sometimes $s_0$. Please clarify throughout the subsequent expressions.}
	Let $N>0$ be an arbitrary positive integer. By re-ordering the terms in \eqref{equ:Uk_iter} and summing over $k-N,\cdots,k$, we have
	\#\label{equ:rate_im_1}
	&\sum_{m=k-N}^k\EE\|\nabla J(\theta_m)\|^2\leq \sum_{m=k-N}^k\frac{1}{\alpha_m}\cdot\big\{\EE[U(\theta_{m})]-\EE[U(\theta_{m+1})]\big\}+\sum_{m=k-N}^kL\alpha_m\hat{\ell}^2\\
	&\quad= \sum_{m=k-N}^k\bigg(\frac{1}{\alpha_m}-\frac{1}{\alpha_{m-1}}\bigg)\cdot\EE[U(\theta_{m})]-\frac{1}{\alpha_{k}}\cdot\EE[U(\theta_{k+1})]+\frac{1}{\alpha_{k-N-1}}\cdot\EE[U(\theta_{k-N})]+\sum_{m=k-N}^kL\alpha_m\hat{\ell}^2 \notag
	\#
	where the equality follows from adding and subtracting an additional term ${\alpha_{k-N-1}}^{-1}\cdot\EE[U(\theta_{k-N})]$. 
	Now, using the fact that the value sub-optimality is bounded by $0\leq U(\theta)\leq 2U_R/(1-\gamma)$, we can further bound the right-hand side of \eqref{equ:rate_im_1} as
\#\label{equ:rate_im_2}
	& \sum_{m=k-N}^k\bigg(\frac{1}{\alpha_m}-\frac{1}{\alpha_{m-1}}\bigg)\cdot\EE[U(\theta_{m})]-\frac{1}{\alpha_{k}}\cdot\EE[U(\theta_{k+1})]+\frac{1}{\alpha_{k-N-1}}\cdot\EE[U(\theta_{k-N})]+\sum_{m=k-N}^kL\alpha_m\hat{\ell}^2\notag\\
	&\quad\leq  \sum_{m=k-N}^k\bigg(\frac{1}{\alpha_m}-\frac{1}{\alpha_{m-1}}\bigg)\cdot\frac{2U_R}{1-\gamma}+\frac{1}{\alpha_{k-N-1}}\cdot\frac{2U_R}{1-\gamma}+\sum_{m=k-N}^kL\alpha_m\hat{\ell}^2\notag\\
	&\quad\leq  \frac{1}{\alpha_k}\cdot\frac{2U_R}{1-\gamma}+\sum_{m=k-N}^kL\alpha_m\hat{\ell}^2,
	\#
	%
	%where the second inequality follows from the fact that $U(\theta)\in[0,2U_R/(1-\gamma)]$ for any $\theta$.
	%
%	\textcolor{green}{we are grouping terms somehow to throw away the first term on the left-hand side. explain explicitly the grouping.} 
	where we drop the nonpositive term $-\EE[U(\theta_{k+1})]/{\alpha_{k}}$ and upper-bound $\EE[U(\theta_{m})]$ by $2U_R/(1-\gamma)$ for all $m=k-N,\cdots,k$. We use the fact that the stepsize is non-increasing $\alpha_m \leq \alpha_{m-1}$, such that $1/{\alpha_m} \geq  1/{\alpha_{m-1}}$. By substituting $\alpha_k=k^{-a}$ into \eqref{equ:rate_im_2} and then \eqref{equ:rate_im_1}, we further have
	\#\label{equ:rate_imme_1}
	\sum_{m=k-N}^k\EE\|\nabla J(\theta_m)\|^2\leq O\bigg(k^a\cdot\frac{2U_R}{1-\gamma}+L\hat{\ell}^2\cdot[k^{1-a}-(k-N)^{1-a}]\bigg),
	\#
	where we use the fact that 
	\$
	\sum_{m=k-N}^k  {m^{-a}}\leq {k^{1-a}}-{(k-N)^{1-a}}
	\$
	for $a\in(0,1)$.
	Setting $N=k-1$ and dividing   by $k$ on both sides of \eqref{equ:rate_imme_1}, we obtain 
	\#\label{equ:rate_imme_2}
	\frac{1}{k}\sum_{m=1}^k\EE\|\nabla J(\theta_m)\|^2\leq O\bigg(k^{a-1}\cdot\frac{2U_R}{1-\gamma}+L\hat{\ell}^2\cdot[k^{-a}-k^{-1}]\bigg)\leq O(k^{-p}),
	\# 
	%
%	\textcolor{green}{I am a little confused by the presence of $N$ in the preceding expression. Shouldn't it have gone away with our substitution? Then the result should make an expression involving $p$ pop out somewhere more overtly.} 
	where $p=\min\{1-a,a\}$. By definition of $K_\epsilon$, we have
	\$
	\EE\|\nabla J(\theta_k)\|^2>\epsilon,\quad~\text{for any~~}k<K_\epsilon,
	\$
	which together with \eqref{equ:rate_imme_2} gives us
	\$
	\epsilon\leq \frac{1}{K_\epsilon}\sum_{m=1}^{K_\epsilon}\EE\|\nabla J(\theta_m)\|^2\leq O(K_\epsilon^{-p}).
	\$
	This shows that $K_\epsilon\leq O(\epsilon^{-1/p})$.
	Note that $\max_{a\in(0,1)}p(a)=1/2$ with $a=1/2$, 	 which  concludes the proof of Theorem \ref{thm:alg1_conv_rate}.

	From \eqref{equ:rate_im_1} in the proof of Theorem \ref{thm:alg1_conv_rate}, we obtain  that for any $k>0$ and $0\leq N< k$,  
	\#\label{equ:coro_med_1}
	&\sum_{m=k-N}^k\EE\|\nabla J(\theta_m)\|^2\leq \sum_{m=k-N}^k\frac{1}{\alpha}\cdot\big\{\EE[U(\theta_{m})]-\EE[U(\theta_{m+1})]\big\}+\sum_{m=k-N}^kL\alpha \hat{\ell}^2\notag\\
	&\quad= \frac{1}{\alpha}\cdot\big\{\EE[U(\theta_{k})]-\EE[U(\theta_{k-N+1})]\big\}+\sum_{m=k-N}^kL\alpha \hat{\ell}^2\leq \frac{1}{\alpha}\cdot \frac{2U_R}{1-\gamma}+(N+1)\cdot L\alpha \hat{\ell}^2,
	\#
	where the equality follows from telescope cancellation and the second inequality follows from the fact that $0\leq U(\theta)\leq 2U_R/(1-\gamma)$. 
	By choosing $N=k-1$ and dividing both sides of \eqref{equ:coro_med_1} by $k$, we obtain 
	\#\label{equ:conv_const_step_im}
	\frac{1}{k}\sum_{m=1}^k\EE\|\nabla J(\theta_m)\|^2\leq \frac{1}{k\alpha}\cdot \frac{2U_R}{1-\gamma}+L\alpha \hat{\ell}^2\leq O(\alpha L \hat{\ell}^2),
	\#
	which completes  the proof of Corollary \ref{coro:alg1_conv_rate_const_step}.
\end{proof}

%%%%%%%%%%%%%%%%%%%%%%%%%%%%%%%%%%%%%%%%%%%%%%%%%%%%%%%%%%%%%%%%%%%%%%%%%%%%%
%%%%%%%%%%%%%%%%%%%%%%%%%%%%%%%%%%%%%%%%%%%%%%%%%%%%%%%%%%%%%%%%%%%%%%%%%%%%%%%%%%%%%%%%%%%%%%%%  P  R  O  O  F %%%%%%%%%%%%%%%%%%%%%%%%%%%%%%%%%%%%%%%%%%%%%%%%%%%%%%%%%%%%%%%%%%%%%%%%%%%%%%%%%%%%%%%%%%%%%%%%%%%%%%%%%%%%%%%%%%%%%%%%%%%%%%%%%%%%%%%%%%%%%%%%%%%%%%%%%%%%%%%%%%%%%%%%%%%%%%%%%%%%%%%%%%%%%%%%%%%%%%%%%%%
\subsection{Proof of Lemma \ref{lemma:offset_reward_nochange}}\label{apx_offset}
\begin{proof}
	A policy $\pi$ is an optimal policy for  the MDP if and only if the corresponding Q-function  satisfies the  Bellman equation \cite{bellman1954theory}, namely, for any $(s,a)\in\cS\times\cA$
	\$
	Q_{\pi}(s,a)=R(s,a)+\gamma\cdot \EE_{s'}\bigg[\max_{a'\in\cA} Q_{\pi}(s',a')\bigg].
	\$
	For any $C\in\RR$, by adding  ${C}/{(1-\gamma)}$ to both sides, we obtain
	\$
	Q_{\pi}(s,a)+\frac{C}{1-\gamma}&=R(s,a)+C+\gamma\cdot \EE_{s'}\bigg[\max_{a'\in\cA} Q_{\pi}(s',a')+ \frac{C}{1-\gamma}\bigg]\\
	&=\tilde{R}(s,a)+\gamma\cdot \EE_{s'}\bigg[\max_{a'\in\cA} \tilde{Q}_{\pi}(s',a')\bigg],
	\$
	where $\tilde{Q}_{\pi}(s',a')=Q_{\pi}(s',a')+ {C}/{(1-\gamma)}$ is the Q-function corresponding to $\tilde{R}$ under policy $\pi$. Since $C\in\RR$ can be any value, we conclude the proof for the opposite direction. 
\end{proof} 

%%%%%%%%%%%%%%%%%%%%%%%%%%%%%%%%%%%%%%%%%%%%%%%%%%%%%%%%%%%%%%%%%%%%%%%%%%%%%
%%%%%%%%%%%%%%%%%%%%%%%%%%%%%%%%%%%%%%%%%%%%%%%%%%%%%%%%%%%%%%%%%%%%%%%%%%%%%%%%%%%%%%%%%%%%%%%%  P  R  O  O  F %%%%%%%%%%%%%%%%%%%%%%%%%%%%%%%%%%%%%%%%%%%%%%%%%%%%%%%%%%%%%%%%%%%%%%%%%%%%%%%%%%%%%%%%%%%%%%%%%%%%%%%%%%%%%%%%%%%%%%%%%%%%%%%%%%%%%%%%%%%%%%%%%%%%%%%%%%%%%%%%%%%%%%%%%%%%%%%%%%%%%%%%%%%%%%%%%%%%%%%%%%%
\subsection{Proof of Lemma \ref{lemma:Hessian_Lip}}\label{apx_Hessian}
\begin{proof}
	First, from Theorem $3$ in  \cite{furmston2016approximate}, we know that 
	 the Hessian $\cH(\theta)$ of $J(\theta)$ takes the form
	\#\label{equ:H_def}
	\cH(\theta)=\cH_1(\theta)+\cH_2(\theta)+\cH_{12}(\theta)+\cH^\top_{12}(\theta),
	\#
	where the matrices  $\cH_1$, $\cH_2$, and $\cH_{12}$ have the form
	\#
	\cH_1(\theta)&=\int_{s\in\cS,a\in\cA}\rho_\theta(s,a)\cdot Q_{\pi_\theta}(s,a)\cdot\nabla\log\pi_\theta(a\given s)\cdot\nabla\log\pi_\theta(a\given s)^\top dads
	\label{equ:H_1_def}
	\\
	\cH_2(\theta)&=\int_{s\in\cS,a\in\cA}\rho_\theta(s,a)\cdot Q_{\pi_\theta}(s,a)\cdot\nabla^2\log\pi_\theta(a\given s) dads
	\label{equ:H_2_def}
\\
	\cH_{12}(\theta)&=\int_{s\in\cS,a\in\cA}\rho_\theta(s,a)\cdot\nabla\log\pi_\theta(a\given s)\cdot\nabla Q_{\pi_\theta}(s,a)^\top dads,
	\label{equ:H_12_def}
	\#
	and $\nabla Q_{\pi_\theta}(s,a)$ here is the gradient of $Q_{\pi_\theta}(s,a)$ with respect to $\theta$. Recall that $\rho_\theta(s,a)=\rho_{\pi_\theta}(s)\cdot \pi_\theta(a\given s)$ and  $\rho_{\pi_\theta}(s)=(1-\gamma)\sum_{t=0}^\infty\gamma^t p(s_t=s\given s_0,\pi_\theta)$   is the  discounted state-occupancy
measure over $\cS$.  
	
	 $\cH_1$ is the Fisher information of the policy scaled by its value in expectation with respect to the  discounted state-occupancy
measure over $\cS$. $\cH_2$ is the Hessian of the log-likelihood of the policy, i.e., the gradient of the score function, again scaled by its value in expectation with respect to the  discounted state-occupancy
measure over $\cS$. $\cH_{12}$ contains a product between the score function and the derivative of the action-value function with respect to the policy scaled in expectation with respect to the  discounted state-occupancy
measure over $\cS$.
	
	For any $\theta$ and $(s,a)$,	we define the function $f_\theta(s,a)$ as
\#\label{equ:f_def}
f_\theta(s,a)&:=\underbrace{Q_{\pi_\theta}(s,a)\cdot\nabla\log\pi_\theta(a\given s)\cdot\nabla\log\pi_\theta(a\given s)^\top}_{f_\theta^1}+\underbrace{Q_{\pi_\theta}(s,a)\cdot\nabla^2\log\pi_\theta(a\given s)}_{f_\theta^2}\notag\\
&\qquad+\underbrace{\nabla\log\pi_\theta(a\given s)\cdot\nabla Q_{\pi_\theta}(s,a)^\top+\nabla Q_{\pi_\theta}(s,a)\cdot\nabla\log\pi_\theta(a\given s)^\top}_{f_\theta^3}.
\#
For notational convenience, we separate   the   terms in $f_\theta$ into   $f_\theta^1,f_\theta^2$, and $f_\theta^3$ as defined above, which are the terms inside the integrand of $\cH_1$, $\cH_2$, and $\cH_{12} + \cH_{12}^\top$.

Note that by definition, 
\$
\cH(\theta)=\int_{s\in\cS,a\in\cA}\rho_\theta(s,a)\cdot f_\theta(s,a) dads. 
\$ 
Then, for any $\theta^1,\theta^2$, we obtain from \eqref{equ:H_def}-\eqref{equ:H_12_def} that 
\#\label{equ:H_lip_immed_1}
&\big\|\cH(\theta^1)-\cH(\theta^2)\big\|\leq \int\bigg\|\rho_{\theta^1}(s,a)\cdot f_{\theta^1}(s,a)-\rho_{\theta^2}(s,a)\cdot f_{\theta^2}(s,a)\bigg\| dads\notag\\
&\quad\leq \int\Big[\big|\rho_{\theta^1}(s,a)-\rho_{\theta^2}(s,a)\big|\cdot \big\|f_{\theta^1}(s,a)\big\|+\big|\rho_{\theta^2}(s,a)\big|\cdot \big\|f_{\theta^1}(s,a)-f_{\theta^2}(s,a)\big\| \Big]dads,
\#
where  the second inequality  follows from adding and subtracting $\rho_{\theta^2}(s,a)\cdot f_{\theta^1}(s,a)$, and  applying the Cauchy-Schwarz inequality. 
Now we proceed our proof by first establishing the boundedness and Lipschitz continuity of  $f_{\theta}(s,a)$. 
	To this end, we need the following technical lemma.

	\begin{lemma}\label{lemma:Lip_Func_Immed_2}
	 For any $(s,a)$, $Q_{\pi_\theta}(s,a)$ and  $\nabla Q_{\pi_\theta}(s,a)$ are both Lipschitz continuous, with constants $L_{Q}:=U_R\cdot B_\Theta\cdot{\gamma}/{(1-\gamma)^2}$ and 
	 \$
	 L_{QGrad}:=U_R\cdot \bigg[B_\Theta^2\cdot \frac{\gamma(1+\gamma)}{(1-\gamma)^3}+L_\Theta \cdot \frac{\gamma}{(1-\gamma)^2}\bigg],
	 \$
	 respectively. Further, the norm of  $\nabla  Q_{\pi_\theta}(s,a)$ is also uniformly  bounded by $L_{Q}$.  
	\end{lemma}
	\begin{proof}
	By the definition of $Q_{\pi_\theta}(s,a)$, we have
	\$
	Q_{\pi_\theta}(s,a)=\sum_{t=0}^\infty \int \gamma^t R(s_t,a_t)\cdot p_{\theta}(h_t\given s_0=s,a_0=a)ds_{1:t}da_{1:t},
	\$
	where  $h_t=(s_0,a_0,s_1,a_1,\cdots,s_t,a_t)$ denotes the trajectory until time $t$, and $p_{\theta}(h_t\given s,a)$ is defined as 
	\#\label{equ:p_theta_def}
	p_{\theta}(h_t\given s_0,a_0)=\bigg[\prod_{u=0}^{t-1}p(s_{u+1}\given s_u,a_u)\bigg]\cdot\bigg[\prod_{u=1}^{t}\pi_\theta(a_u\given s_u)\bigg].
	\#
	Therefore, the gradient $\nabla  Q_{\pi_\theta}(s,a)$ has the following form  
	\#
	&\nabla  Q_{\pi_\theta}(s,a)=\nabla \sum_{t=0}^\infty \int \gamma^t R(s_t,a_t)\cdot p_{\theta}(h_t\given s_0=s,a_0=a)ds_{1:t}da_{1:t}\notag\\
	&\quad =\sum_{t=1}^\infty \int \gamma^t R(s_t,a_t)\cdot p_{\theta}(h_t\given s_0=s,a_0=a)\cdot \sum_{u=1}^t \nabla \log\pi_\theta(a_u\given s_u) ds_{1:t}da_{1:t},\label{equ:grad_Q_form}
%	\\
%	&\quad =\sum_{t=1}^\infty \sum_{u=1}^t\int \gamma^t R(s_t,a_t)\cdot p_{\theta}(h_t\given s_0=s,a_0=a) \nabla \log\pi_\theta(a_u\given s_u) ds_{1:t}da_{1:t},
	\#
	where \eqref{equ:grad_Q_form} is due to the facts that: i) the first term in the summation $ R(s_0,a_0)$ does not depend on $\theta$; ii) for any $t\geq 1$, 
	\#
	&\nabla p_{\theta}(h_t\given s_0,a_0)=\bigg[\prod_{u=0}^{t-1}p(s_{u+1}\given s_u,a_u)\bigg]\cdot\nabla\bigg[\prod_{u=1}^{t}\pi_\theta(a_u\given s_u)\bigg]\notag\\
	&\quad =\bigg[\prod_{u=0}^{t-1}p(s_{u+1}\given s_u,a_u)\bigg]\cdot\sum_{\tau=1}^t\bigg[\prod_{u\neq \tau, u=1}^{t}\pi_\theta(a_u\given s_u)\nabla \pi_\theta(a_\tau\given s_\tau)\bigg]\notag\\
	&\quad =\bigg[\prod_{u=0}^{t-1}p(s_{u+1}\given s_u,a_u)\bigg]\cdot\bigg[\prod_{u=1}^{t}\pi_\theta(a_u\given s_u)\bigg]\cdot \sum_{\tau=1}^t\nabla\log  \pi_\theta(a_\tau\given s_\tau)\notag\\
	&\quad =p_{\theta}(h_t\given s_0,a_0)\cdot \sum_{\tau=1}^t\nabla\log  \pi_\theta(a_\tau\given s_\tau).\label{equ:grad_p_theta_immed}
	\#
	Hence, from \eqref{equ:grad_Q_form} we immediately have that for any $(s,a)$ and $\theta$, 
	\#
	&\big\|\nabla  Q_{\pi_\theta}(s,a)\big\|\leq	 \sum_{t=1}^\infty \int \gamma^t |R(s_t,a_t)|\cdot p_{\theta}(h_t\given s_0=s,a_0=a)\cdot \Bigg\|\sum_{u=1}^t \nabla \log\pi_\theta(a_u\given s_u) \Bigg\|ds_{1:t}da_{1:t}\notag\\
	&\quad \leq	 \sum_{t=1}^\infty \int \gamma^t\cdot  U_R\cdot p_{\theta}(h_t\given s_0=s,a_0=a)\cdot \sum_{u=1}^t \big\|\nabla \log\pi_\theta(a_u\given s_u) \big\|ds_{1:t}da_{1:t}\label{equ:grad_Q_immed_1}\\
	&\quad \leq	 \sum_{t=1}^\infty \int \gamma^t\cdot  U_R\cdot p_{\theta}(h_t\given s_0=s,a_0=a)\cdot B_\Theta \cdot t\cdot  ds_{1:t}da_{1:t}=U_R\cdot B_\Theta\sum_{t=1}^\infty  \gamma^t\cdot  t,\label{equ:grad_Q_immed_2}
	\#
	where \eqref{equ:grad_Q_immed_1}  and \eqref{equ:grad_Q_immed_2} are due to the boundedness of $|R(s,a)|$ and $\|\nabla \log\pi_\theta(a_u\given s_u) \|$, respectively. Let $S=\sum_{t=1}^\infty \gamma^t \cdot t$;   then
	\#\label{equ:seq_prod_trick}
	(1-\gamma) \cdot S=\gamma+\sum_{t=2}^\infty \gamma^t=\frac{\gamma}{1-\gamma}~~\Longrightarrow~~S= \frac{\gamma}{(1-\gamma)^2}. 
	\#
Combining \eqref{equ:grad_Q_immed_2} and \eqref{equ:seq_prod_trick}, we  further establish that 
	\#\label{equ:grad_Q_bnd_norm}
	\big\|\nabla  Q_{\pi_\theta}(s,a)\big\|\leq U_R\cdot B_\Theta\cdot \frac{\gamma}{(1-\gamma)^2},
	\#
	which proves that  $\nabla  Q_{\pi_\theta}(s,a)$ has norm uniformly bounded by $U_R\cdot B_\Theta\cdot {\gamma}/{(1-\gamma)^2}$. Moreover, \eqref{equ:grad_Q_bnd_norm} also implies that $Q_{\pi_\theta}(s,a)$ is Lipschitz continuous with constant $U_R\cdot B_\Theta\cdot {\gamma}/{(1-\gamma)^2}$.

Now we proceed to show the Lipschitz continuity of $\nabla  Q_{\pi_\theta}(s,a)$. 
For any $\theta^1,\theta^2\in\RR^d$, we obtain from \eqref{equ:grad_Q_form} that 
	\#\label{equ:grad_Q_Lip_immed_1}
	&\big|\nabla  Q_{\pi_{\theta^1}}(s,a)-\nabla  Q_{\pi_{\theta^2}}(s,a)\big|\leq \sum_{t=1}^\infty \int \gamma^t |R(s_t,a_t)|\cdot \Bigg|p_{\theta^1}(h_t\given s_0=s,a_0=a)\cdot \sum_{u=1}^t \nabla \log\pi_{\theta^1}(a_u\given s_u)\notag\\
	&\qquad\qquad\qquad\qquad\qquad\qquad-p_{\theta^2}(h_t\given s_0=s,a_0=a)\cdot \sum_{u=1}^t \nabla \log\pi_{\theta^2}(a_u\given s_u)\Bigg| ds_{1:t}da_{1:t} \notag\\
	&\quad\leq \sum_{t=1}^\infty \int \gamma^t U_R\cdot \bigg[\underbrace{\big|p_{\theta^1}(h_t\given s_0=s,a_0=a)-p_{\theta^2}(h_t\given s_0=s,a_0=a)\big|\cdot\bigg\| \sum_{u=1}^t \nabla \log\pi_{\theta^1}(a_u\given s_u)\bigg\|}_{I_1}\notag\\
	&\qquad + \underbrace{p_{\theta^2}(h_t\given s_0=s,a_0=a)\cdot\bigg\| \sum_{u=1}^t \big[\nabla \log\pi_{\theta^1}(a_u\given s_u)- \nabla \log\pi_{\theta^2}(a_u\given s_u)\big]\bigg\|}_{I_2}\bigg]ds_{1:t}da_{1:t}.
	\#
Now we upper bound $I_1$ and $I_2$ separately as follows. 
	By Taylor expansion of $\prod_{u=1}^t\pi_{\theta}(a_u\given s_u)$,  we have
	\#\label{equ:grad_Q_Lip_immed_2}
	&\Bigg|\prod_{u=1}^t\pi_{\theta^1}(a_u\given s_u)-\prod_{u=1}^t\pi_{\theta^2}(a_u\given s_u)\Bigg|=\Bigg|(\theta^1-\theta^2)^\top \bigg[\sum_{m=1}^t\nabla \pi_{\tilde{\theta}}(a_m\given s_m)\prod_{u\neq m,u=1}^t\pi_{\ttheta}(a_u\given s_u)\bigg]\Bigg|\notag\\
	&\quad\leq \|\theta^1-\theta^2\| \cdot\sum_{m=1}^t\|\nabla \log\pi_{\tilde{\theta}}(a_m\given s_m)\|\cdot \prod_{u=1}^t\pi_{\ttheta}(a_u\given s_u)\notag\\
	&\quad\leq \|\theta^1-\theta^2\| \cdot t\cdot B_{\Theta}\cdot \prod_{u=1}^t\pi_{\ttheta}(a_u\given s_u),
	\#
	where $\ttheta$ is a vector lying between $\theta^1$ and $\theta^2$, i.e., there exists some $\lambda\in[0,1]$ such that $\ttheta=\lambda \theta^1+(1-\lambda)\theta^2$. Therefore, \eqref{equ:grad_Q_Lip_immed_2},  combined with  \eqref{equ:p_theta_def}, yields   
	\#\label{equ:grad_Q_Lip_immed_3}
	\big|p_{\theta^1}(h_t\given s_0,a_0)-p_{\theta^2}(h_t\given s_0,a_0)\big|&\leq \bigg[\prod_{u=0}^{t-1}p(s_{u+1}\given s_u,a_u)\bigg]\cdot \|\theta^1-\theta^2\| \cdot t\cdot B_{\Theta}\cdot \prod_{u=1}^t\pi_{\ttheta}(a_u\given s_u)\notag\\
	& = \|\theta^1-\theta^2\| \cdot t\cdot B_{\Theta}\cdot p_{\ttheta}(h_t\given s_0,a_0). 
	\#
	Therefore, the term $I_1$ can be bounded as follows by substituting \eqref{equ:grad_Q_Lip_immed_3} 
	\#\label{equ:grad_Q_Lip_immed_4}
	I_1&\leq \|\theta^1-\theta^2\| \cdot t\cdot B_{\Theta}\cdot p_{\ttheta}(h_t\given s_0=s,a_0=a)\cdot \bigg\| \sum_{u=1}^t \nabla \log\pi_{\theta^1}(a_u\given s_u)\bigg\|\notag\\
	&\leq \|\theta^1-\theta^2\| \cdot t\cdot B_{\Theta}\cdot p_{\ttheta}(h_t\given s_0=s,a_0=a)\cdot t\cdot B_\Theta.
	\#
	In addition, $I_2$ can be bounded using the $L_\Theta$-Lipschitz continuity of $\nabla \log\pi_{\theta}(a\given s)$, i.e.,
	\#\label{equ:grad_Q_Lip_immed_5}
	I_2&\leq p_{\theta^2}(h_t\given s_0=s,a_0=a)\cdot \sum_{u=1}^t \big\|\nabla \log\pi_{\theta^1}(a_u\given s_u)- \nabla \log\pi_{\theta^2}(a_u\given s_u)\big\|\notag\\
	&\leq p_{\theta^2}(h_t\given s_0=s,a_0=a)\cdot t\cdot L_\Theta\cdot  \big\|\theta^1- \theta^2\big\|. 
	\#
	Substituting \eqref{equ:grad_Q_Lip_immed_4} and \eqref{equ:grad_Q_Lip_immed_5} into \eqref{equ:grad_Q_Lip_immed_1}, we obtain that
	\#\label{equ:grad_Q_Lip_immed_6}
	&\big\|\nabla  Q_{\pi_{\theta^1}}(s,a)-\nabla  Q_{\pi_{\theta^2}}(s,a)\big\|\leq \sum_{t=1}^\infty \int \gamma^t U_R\cdot \bigg[\|\theta^1-\theta^2\| \cdot t^2\cdot B^2_{\Theta}\cdot p_{\ttheta}(h_t\given s_0=s,a_0=a)\notag\\
	&\qquad\qquad\qquad\qquad\qquad\qquad + p_{\theta^2}(h_t\given s_0=s,a_0=a)\cdot t\cdot L_\Theta\cdot  \big\|\theta^1- \theta^2\big\|\bigg]ds_{1:t}da_{1:t}\notag\\
	&\quad=  \sum_{t=1}^\infty  \gamma^t U_R\cdot \Big( t^2\cdot B^2_{\Theta}+  t\cdot L_\Theta\Big)\cdot \|\theta^1-\theta^2\|\notag\\
	&\quad=U_R\cdot \bigg[B_\Theta^2\cdot \frac{\gamma(1+\gamma)}{(1-\gamma)^3}+L_\Theta \cdot \frac{\gamma}{(1-\gamma)^2}\bigg]\cdot \|\theta^1- \theta^2\big\|,
	\#
	where the first equality follows from that $\int p_{\theta}(h_t\given s_0=s,a_0=a)ds_{1:t}da_{1:t}=1$ for any $\theta$, and the last equality is due to \eqref{equ:seq_prod_trick} plus  the fact that 
	\$
	\sum_{t=1}^\infty \gamma^t \cdot t^2=\frac{1}{1-\gamma}\sum_{t=0}^\infty (1-\gamma)\gamma^t \cdot t^2=\frac{1}{1-\gamma}\cdot\EE T^2=\frac{1}{1-\gamma}\cdot \frac{\gamma(1+\gamma)}{(1-\gamma)^2}. 
	\$
	Note that $T$ is a random variable following geometric distribution with success probability $1-\gamma$.  Hence, \eqref{equ:grad_Q_Lip_immed_6} shows the uniform Lipschitz continuity of $\nabla  Q_{\pi_{\theta}}(s,a)$ for any $(s,a)$, with the desired constant $L_{QGrad}$ claimed in the lemma. This completes  the  proof. 
\end{proof}
	
Using Lemma \ref{lemma:Lip_Func_Immed_2}, we can easily  obtain the boundedness and Lipschitz continuity of $f_\theta(s,a)$ (cf. definition in \eqref{equ:f_def}). In particular,  to show that the norm of $f_\theta(s,a)$ is bounded, we have 
\#\label{equ:f_bnd}
&\|f_\theta(s,a)\|\leq \big|Q_{\pi_\theta}(s,a)\big|\cdot\big\|\nabla\log\pi_\theta(a\given s)\cdot\nabla\log\pi_\theta(a\given s)^\top+\nabla^2\log\pi_\theta(a\given s)\big\|\notag\\
&\qquad\qquad\qquad\qquad\qquad+\big\|\nabla\log\pi_\theta(a\given s)\cdot\nabla Q_{\pi_\theta}(s,a)^\top+\nabla Q_{\pi_\theta}(s,a)\cdot\nabla\log\pi_\theta(a\given s)^\top\big\|\notag\\
&\quad\leq \frac{U_R}{1-\gamma}\cdot\big[\big\|\nabla\log\pi_\theta(a\given s)\big\|^2+\big\|\nabla^2\log\pi_\theta(a\given s)\big\|\big]+{2\cdot \big\|\nabla\log\pi_\theta(a\given s)\big\|\cdot\big\|\nabla Q_{\pi_\theta}(s,a)\big\|}\notag\\
&\quad\leq \frac{U_R}{1-\gamma}\cdot(B_\Theta^2+L_\Theta)+{2\cdot B_\Theta\cdot L_Q}=\underbrace{\frac{U_R(B_\Theta^2+L_\Theta)}{1-\gamma}+\frac{2  U_R B^2_\Theta{\gamma}}{(1-\gamma)^2}}_{B_f},
\#
where the second inequality follows from the fact\footnote{Note that by definition, for any two vectors $a,b\in\RR^d$, $\|ab^\top\|=\sup_{\|v\|=1}\sqrt{v^\top \cdot{ba^\top ab^\top} \cdot v}=\sup_{\|v\|=1}\sqrt{\|v^\top b\|^2\cdot \|a\|^2}\leq \|a\|\cdot\|b\|$. Specially, if $a=b$, $\|aa^\top\|\leq \|a\|^2$. } that for any vector $a,b\in\RR^d$, $\|ab^\top\|\leq \|a\|\cdot\|b\|$, and $|Q_{\pi_\theta}|\leq U_R/{(1-\gamma)}$.
We use $B_f$ to denote the bound of the norm $\|f_\theta(s,a)\|$.

To show the Lipschitz continuity of $f_\theta(s,a)$, we need the following straightforward but useful lemma.

	\begin{lemma}\label{lemma:Lip_Func_Immed}
		For any two functions $f_1,f_2:\RR^{d}\to \RR^{m\times n}$, if, for  $i=1,2$, $f_i$ has norm bounded by $C_i$ and is $L_i$-Lipschitz continuous, then $f_1+f_2$ is $L_m$-Lipschitz continuous, and $f_1\cdot f_2^\top$ is $\tilde{L}_m$-Lipschitz continuous, with $L_m=\max\{C_1,C_2\}$ and $\tilde{L}_m=\max\{C_1 L_2,C_2L_1\}$.
	\end{lemma}
	\begin{proof}
	The proof is straightforward, and is thus  omitted here. 
	\end{proof}

By Lemma \ref{lemma:Lip_Func_Immed}, we immediately have that $\nabla\log\pi_\theta(a\given s)\cdot\nabla\log\pi_\theta(a\given s)^\top$ is $B_\Theta L_\Theta$-Lipschitz continuous.  Also, note that the norm of $\nabla\log\pi_\theta(a\given s)\cdot\nabla\log\pi_\theta(a\given s)^\top$ is bounded by $B_\Theta^2$. 
Thus, recalling the definition in \eqref{equ:f_def}, we further obtain from Lemmas  \ref{lemma:Lip_Func_Immed_2} and \ref{lemma:Lip_Func_Immed} that  for any $\theta^1,\theta^2$, 
\#\label{equ:Lip_f_1}
\|f_{\theta^1}^1(s,a)-f_{\theta^2}^1(s,a)\|\leq \max\bigg\{\frac{U_R}{1-\gamma}\cdot B_\Theta L_\Theta,\frac{U_R B_\Theta \gamma}{(1-\gamma)^2}\cdot B_\Theta^2\bigg\}\cdot \|\theta^1-\theta^2\|.
\#
Similarly, we establish  the Lipschitz continuity of $f_{\theta^1}^2(s,a)$ and $f_{\theta^1}^3(s,a)$   as follows
\#
\|f_{\theta^1}^2(s,a)-f_{\theta^2}^2(s,a)\|\leq \max\bigg\{\frac{U_R}{1-\gamma}\cdot \rho_\Theta,\frac{U_R B_\Theta \gamma}{(1-\gamma)^2}\cdot L_\Theta\bigg\}\cdot \|\theta^1-\theta^2\|,\label{equ:Lip_f_2}\\
\|f_{\theta^1}^3(s,a)-f_{\theta^2}^3(s,a)\|\leq 2\cdot\max\bigg\{B_\Theta\cdot L_{QGrad},\frac{U_R B_\Theta \gamma}{(1-\gamma)^2}\cdot L_\Theta\bigg\}\cdot \|\theta^1-\theta^2\|,\label{equ:Lip_f_3}
\#
where \eqref{equ:Lip_f_2} is due to  $|Q_{\pi_\theta}(s,a)|$ being  ${U_R}/{1-\gamma}$-bounded   and ${U_R\cdot B_\Theta\cdot \gamma}/{(1-\gamma)^2}$-Lipschitz, and $\nabla^2 \log\pi_\theta(a\given s)$ being  $L_\Theta$-bounded and $\rho_\Theta$-Lipschitz; \eqref{equ:Lip_f_3} is due to  $|\nabla Q_{\pi_\theta}(s,a)|$ being ${U_R\cdot B_\Theta\cdot \gamma}/{(1-\gamma)^2}$-bounded and $L_{QGrad}$-Lipschitz, and $\nabla \log\pi_\theta(a\given s)$ being  $B_\Theta$-bounded and $L_\Theta$-Lipschitz.
Combining \eqref{equ:Lip_f_1}-\eqref{equ:Lip_f_3} and the definition in \eqref{equ:f_def}, we finally obtain the Lipschitz continuity of $f_\theta(s,a)$ with constant $L_f$, i.e.,
\#\label{equ:Lip_f_total}
&\|f_{\theta^1}(s,a)-f_{\theta^2}(s,a)\|\leq \underbrace{\frac{U_RB_\Theta}{1-\gamma}\cdot \max\bigg\{{L_\Theta},\frac{B^2_\Theta \gamma}{1-\gamma},\frac{\rho_\Theta}{B_\Theta},\frac{L_\Theta \gamma}{1-\gamma}, \frac{B_\Theta^2(1+\gamma)+L_\Theta(1-\gamma)\gamma}{(1-\gamma)^2}\bigg\}}_{L_f}\cdot \|\theta^1-\theta^2\|. 
\#
 
By substituting \eqref{equ:f_bnd} and \eqref{equ:Lip_f_total} into \eqref{equ:H_lip_immed_1}, we arrive at  
\#\label{equ:H_lip_immed_2}
&\big\|\cH(\theta^1)-\cH(\theta^2)\big\|\leq \int\Big[\big|\rho_{\theta^1}(s,a)-\rho_{\theta^2}(s,a)\big|\cdot \big\|f_{\theta^1}(s,a)\big\|+\big|\rho_{\theta^2}(s,a)\big|\cdot \big\|f_{\theta^1}(s,a)-f_{\theta^2}(s,a)\big\| \Big]dads\notag\\
&\quad\leq \int \big|\rho_{\theta^1}(s,a)-\rho_{\theta^2}(s,a)\big|\cdot B_f dads+L_f\cdot\|\theta^1-\theta^2\|\cdot\int \rho_{\theta^2}(s,a) dads\notag\\
&\quad=B_f\cdot\int \big|\rho_{\theta^1}(s,a)-\rho_{\theta^2}(s,a)\big| dads+L_f\cdot\|\theta^1-\theta^2\|.
\#
Now it suffices to show the Lipschitz continuity of $\int \big|\rho_{\theta^1}(s,a)-\rho_{\theta^2}(s,a)\big| dads$.  By definition,  we have
	\#\label{equ:rho_Lip_immed_1}
	\rho_{\theta}(s,a)&=(1-\gamma)\cdot\sum_{t=0}^\infty\gamma^tp(s_t=s\given s_0,\pi_{\theta})\pi_{\theta}(a\given s)\notag\\
	&=(1-\gamma)\cdot\sum_{t=0}^\infty\gamma^tp(s_t=s,a_t=a\given s_0,\pi_{\theta}). 
	\#
	Note that 
	\#\label{equ:p_sa_def}
	p(s_t,a_t\given s_0,\pi_{\theta})=\int \underbrace{\bigg[\prod_{u=0}^{t-1}p(s_{u+1}\given s_u,a_u)\bigg]\cdot\bigg[\prod_{u=0}^{t}\pi_\theta(a_u\given s_u)\bigg]}_{p_\theta(h_t\given s_0)}ds_{1:t-1}da_{0:t-1},
	\#
	where we define $p_\theta(h_t\given s_0)$ similarly to  $p_\theta(h_t\given s_0,a_0)$ in \eqref{equ:p_theta_def}. 
	Hence, for any $\theta^1,\theta^2\in\RR^d$,  \eqref{equ:rho_Lip_immed_1} yields 
	\#
	&\int \big|\rho_{\theta^1}(s,a)-\rho_{\theta^2}(s,a)\big|dsda\notag\\
	&\quad=(1-\gamma)\cdot\sum_{t=0}^\infty\gamma^t\int \big|p(s_t=s,a_t=a\given s_0,\pi_{\theta^1})-p(s_t=s,a_t=a\given s_0,\pi_{\theta^2})\big|dsda\notag\\
	&\quad\leq (1-\gamma)\cdot\sum_{t=0}^\infty\gamma^t\int \big| p_{\theta^1}(h_t\given s_0)-p_{\theta^2}(h_t\given s_0)\big|ds_{1:t-1}da_{0:t-1}ds_tda_t,\label{equ:rho_Lip_immed_2}
	\#
	where the first equality interchanges the sum and the integral due to the monotone convergence theorem; the inequality follows by substituting \eqref{equ:p_sa_def} and applying the Cauchy-Schwarz inequality. 
	Now it suffices to bound $| p_{\theta^1}(h_t\given s_0)-p_{\theta^2}(h_t\given s_0)|$.
	Then, we can apply the  same argument from \eqref{equ:grad_Q_Lip_immed_2} to \eqref{equ:grad_Q_Lip_immed_3} that bounds $|p_{\theta^1}(h_t\given s_0,a_0)-p_{\theta^2}(h_t\given s_0,a_0)|$. Note that  the only difference between the definitions of $p_{\theta}(h_t\given s_0)$ and $p_{\theta}(h_t\given s_0,a_0)$ is one additional multiplication of $\pi_\theta (a_0\given s_0)$. 
	 Thus, we will first have 
	\$
	&\Bigg|\prod_{u=0}^t\pi_{\theta^1}(a_u\given s_u)-\prod_{u=0}^t\pi_{\theta^2}(a_u\given s_u)\Bigg|\leq \|\theta^1-\theta^2\| \cdot (t+1)\cdot B_{\Theta}\cdot \prod_{u=0}^t\pi_{\ttheta}(a_u\given s_u),
	\$
	where $\ttheta$ is some vector lying between $\theta^1$ and $\theta^2$. Then, the bound for  $| p_{\theta^1}(h_t\given s_0)-p_{\theta^2}(h_t\given s_0)|$ has the form of
	\#\label{equ:rho_Lip_immed_4}
	\big| p_{\theta^1}(h_t\given s_0)-p_{\theta^2}(h_t\given s_0)\big|= \|\theta^1-\theta^2\| \cdot (t+1)\cdot B_{\Theta}\cdot p_{\ttheta}(h_t\given s_0).
	\#
	Combining \eqref{equ:rho_Lip_immed_2} and \eqref{equ:rho_Lip_immed_4}, we obtain  
	\#\label{equ:rho_Lip_immed_5}
	&\int \big|\rho_{\theta^1}(s,a)-\rho_{\theta^2}(s,a)\big|dsda\leq (1-\gamma)\cdot\sum_{t=0}^\infty\gamma^t\int \|\theta^1-\theta^2\| \cdot (t+1)\cdot B_{\Theta}\cdot p_{\ttheta}(h_t\given s_0)ds_{1:t}da_{0:t}\notag\\
	&\quad = (1-\gamma)\cdot\sum_{t=0}^\infty\gamma^t\|\theta^1-\theta^2\| \cdot (t+1)\cdot B_{\Theta}=(1-\gamma)\cdot \|\theta^1-\theta^2\| \cdot B_{\Theta}\cdot\frac{1}{(1-\gamma)^2}. 
	\#
	By substituting \eqref{equ:rho_Lip_immed_5} into \eqref{equ:H_lip_immed_2}, we finally arrive at the desired result, i.e.,
	\$
	&\big\|\cH(\theta^1)-\cH(\theta^2)\big\|\leq B_f\cdot\int \big|\rho_{\theta^1}(s,a)-\rho_{\theta^2}(s,a)\big| dads+L_f\cdot\|\theta^1-\theta^2\|\\
	&\quad \leq B_f\cdot\|\theta^1-\theta^2\|  \cdot\frac{B_{\Theta}}{1-\gamma}+L_f\cdot\|\theta^1-\theta^2\|=\bigg(\frac{B_fB_{\Theta}}{1-\gamma}+L_f\bigg)\cdot\|\theta^1-\theta^2\|,
	\$
	where $B_f$ and $L_f$ are as defined in \eqref{equ:f_bnd} and \eqref{equ:Lip_f_total}.  
	In sum, the Lipschitz constant $\rho$ in the lemma has the following form
	\#\label{equ:value_rho}
	\rho:=\frac{U_RB_\Theta L_\Theta}{(1-\gamma)^2}+\frac{  U_R B^3_\Theta{(1+\gamma)}}{(1-\gamma)^3}+\frac{U_RB_\Theta}{1-\gamma}\cdot \max\bigg\{{L_\Theta},\frac{B^2_\Theta \gamma}{1-\gamma},\frac{\rho_\Theta}{B_\Theta},\frac{L_\Theta \gamma}{1-\gamma}, \frac{B_\Theta^2(1+\gamma)+L_\Theta(1-\gamma)\gamma}{(1-\gamma)^2}\bigg\}. 
	\# 
	This completes the proof. 
\end{proof}

%%%%%%%%%%%%%%%%%%%%%%%%%%%%%%%%%%%%%%%%%%%%%%%%%%%%%%%%%%%%%%%%%%%%%%%%%%%%%
%%%%%%%%%%%%%%%%%%%%%%%%%%%%%%%%%%%%%%%%%%%%%%%%%%%%%%%%%%%%%%%%%%%%%%%%%%%%%%%%%%%%%%%%%%%%%%%%  P  R  O  O  F %%%%%%%%%%%%%%%%%%%%%%%%%%%%%%%%%%%%%%%%%%%%%%%%%%%%%%%%%%%%%%%%%%%%%%%%%%%%%%%%%%%%%%%%%%%%%%%%%%%%%%%%%%%%%%%%%%%%%%%%%%%%%%%%%%%%%%%%%%%%%%%%%%%%%%%%%%%%%%%%%%%%%%%%%%%%%%%%%%%%%%%%%%%%%%%%%%%%%%%%%%%
\subsection{Proof of Lemma \ref{lemma:CNC_verify}}\label{apx_cnc}

\begin{proof}
We start with the proof for $\EE\big\{[\vb_\theta^\top \hat{\nabla}J(\theta)]^2\biggiven \theta\big\}$. By definition, we have that for any $\vb\in\RR^d$ and $\|\vb\|=1$,
\#\label{equ:CNC_1_to_verify}
&\EE\big\{[\vb^\top \hat{\nabla}J(\theta)]^2\biggiven \theta\big\}=\EE\big\{\big[\hat{Q}_{\pi_\theta}(s_T,a_T)\cdot\vb^\top \nabla \log\pi_{\theta}(a_T\given s_T)\big]^2\biggiven \theta\big\}\notag\\
&\quad =\EE_{T,(s_T,a_T)}\big\{\EE_{T',(s_{1:T'},a_{1:T'})}\big[\hat{Q}_{\pi_\theta}^2(s_T,a_T)\biggiven \theta,s_T,a_T\big]\cdot\big[\vb^\top \nabla \log\pi_{\theta}(a_T\given s_T)\big]^2\biggiven \theta\big\}. 
\#
For notational simplicity, we write $\EE_{T',(s_{1:T'},a_{1:T'})}[\hat{Q}_{\pi_\theta}^2(s_T,a_T)\biggiven \theta,s_T,a_T]$ as $\EE_{T',(s_{1:T'},a_{1:T'})}[\hat{Q}_{\pi_\theta}^2(s_T,a_T)]$, which is the conditional expectation over the  sequence $(s_{1:T'},a_{1:T'})$ and the random variable $T'$, given $\theta$ and $s_T,a_T$. 
Then note that $\EE_{T',(s_{1:T'},a_{1:T'})}[\hat{Q}_{\pi_\theta}^2(s_T,a_T)]$ is uniformly lower-bounded for any $(s_T,a_T)$ and any $\theta$, since the reward $|R|$ is lower-bounded by $L_R>0$. 
%Recall that the initial state-action pair $(s_0,a_0)$ along the trajectory of estimating $\hat{Q}_{\pi_\theta}(s_T,a_T)$ is $(s_T,a_T)$. 
In particular, we have
\$
&\EE_{T',(s_{1:T'},a_{1:T'})}\big[\hat{Q}_{\pi_\theta}^2(s_T,a_T)\big]=\EE_{T'}\bigg(\EE_{(s_{1:T'},a_{1:T'})}\bigg\{\bigg[\sum_{t=0}^{T'}\gamma^{t/2}\cdot R(s_t,a_t)\bigg]^2\bigggiven T'=\tau\bigg\}\bigg)\\
&\quad\geq  \EE_{T'}\bigg(\frac{1-\gamma^{(T'+1)/2}}{1-\gamma^{1/2}}\cdot L_R\bigg)^2 \geq L_R^2\cdot \sum_{\tau=0}^\infty \gamma^{\tau/2}(1-\gamma^{1/2})=L_R^2>0,
\$
where the first inequality holds because  $R(s,a)$ is either all positive or negative for any $s,a$, and the second  inequality follows from the fact that $[{1-\gamma^{(T'+1)/2}}]\cdot{(1-\gamma^{1/2})^{-1}}\geq 1$ for all $T'\geq 0$. 
 Substituting the preceding expression into the first product term on the right-hand side of \eqref{equ:CNC_1_to_verify} and pulling out the vector $\vb$ yields 
\#
\EE\big\{[\vb^\top \hat{\nabla}J(\theta)]^2\biggiven \theta\big\}&\geq L_R^2\cdot \vb^\top\cdot\EE_{T,(s_T,a_T)}\big\{\nabla \log[\pi_{\theta}(a_T\given s_T)\cdot \nabla \log[\pi_{\theta}(a_T\given s_T)^\top \biggiven \theta\big\}\cdot\vb\notag\\
&\geq L_R^2\cdot L_I=:\hat{\eta}>0,\label{equ:CNC_Verify_res_1}
\#
where the second inequality follows from the fact the Fisher information matrix is assumed to be positive definite (cf.  \eqref{equ:Fisher_lower_bnd}) in Assumption \ref{assum:regularity_for_CNC}.  Note that \eqref{equ:CNC_Verify_res_1} holds for any unit-norm vector $\vb$, and  does also for any eigenvector $\vb_\theta$ (may be more than one) that corresponds to the maximum eigenvalue of $\cH(\theta)$. This verifies that $\EE\big\{[\vb_\theta^\top \hat{\nabla}J(\theta)]^2\biggiven \theta\big\}\geq \hat{\eta}$ for some $\hat{\eta}$ defined in \eqref{equ:CNC_Verify_res_1}.

To establish that the CNC condition holds for $\EE\big\{[\vb^\top \check{\nabla}J(\theta)]^2\biggiven \theta\big\}$, the steps are similar to those previously followed for $\hat{\nabla}J(\theta)$. Specifically, we start with the expected value of the  square of the inner product of $\check{\nabla}J(\theta)$ with a unit vector $\vb$. By definition of $\check{\nabla}J(\theta)$, we have
\#\label{equ:CNC_1_to_verify_2}
&\EE\big\{[\vb^\top \check{\nabla}J(\theta)]^2\biggiven \theta\big\}=\EE\big\{\big[\hat{Q}_{\pi_\theta}(s_T,a_T)-\hat{V}_{\pi_\theta}(s_T)\big]^2\cdot\big[\vb^\top \nabla \log\pi_{\theta}(a_T\given s_T)\big]^2\biggiven \theta\big\}\notag\\
&\quad  =\EE_{T,(s_T,a_T)}\big\{\EE_{T',(s_{1:T'},a_{1:T'})}\big[\hat{Q}_{\pi_\theta}(s_T,a_T)-\hat{V}_{\pi_\theta}(s_T)\big]^2\cdot[\vb^\top \nabla \log\pi_{\theta}(a_T\given s_T)]^2\biggiven \theta\big\},
\#
where for notational simplicity we also write $\EE_{T',(s_{1:T'},a_{1:T'})}\big\{\big[\hat{Q}_{\pi_\theta}(s_T,a_T)-\hat{V}_{\pi_\theta}(s_T)\big]^2\biggiven \theta,s_T,a_T\big\}$ as $\EE_{T',(s_{1:T'},a_{1:T'})}\big[\hat{Q}_{\pi_\theta}(s_T,a_T)-\hat{V}_{\pi_\theta}(s_T)\big]^2$. 
We claim that $\EE_{T',(s_{1:T'},a_{1:T'})}[\hat{Q}_{\pi_\theta}(s_T,a_T)-\hat{V}_{\pi_\theta}(s_T)]^2$ can also be uniformly lower-bounded. Specifically, we have
\#\label{equ:CNC_Verify_2}
&\EE_{T',(s_{1:T'},a_{1:T'})}\big[\hat{Q}_{\pi_\theta}(s_T,a_T)-\hat{V}_{\pi_\theta}(s_T)\big]^2\notag\\
&\quad= \big\{\EE_{T',(s_{1:T'},a_{1:T'})}\big[\hat{Q}_{\pi_\theta}(s_T,a_T)-\hat{V}_{\pi_\theta}(s_T)\big]\big\}^2+\Var\big[\hat{Q}_{\pi_\theta}(s_T,a_T)-\hat{V}_{\pi_\theta}(s_T)\big]\notag\\
&\quad= \big[{Q}_{\pi_\theta}(s_T,a_T)-{V}_{\pi_\theta}(s_T)\big]^2+\Var\big[\hat{Q}_{\pi_\theta}(s_T,a_T)\big]+\Var\big[\hat{V}_{\pi_\theta}(s_T)\big],
\# 
where the first equality is due to $\EE X^2=(\EE X)^2+\Var(X)$, and the second one follows from the fact that $\hat{Q}_{\pi_\theta}(s_T,a_T)$ and $\hat{V}_{\pi_\theta}(s_T)$ are  independent and  unbiased  estimates of ${Q}_{\pi_\theta}(s_T,a_T)$ and ${V}_{\pi_\theta}(s_T)$, respectively. Note that  the first term in \eqref{equ:CNC_Verify_2} may be zero, for example, when $\pi_\theta$ is a degenerated policy such that that  $\pi_\theta(a\given s_T)=\mathbbm{1}_{a=a_T}$. Hence, a uniform  lower-bound on the two variance terms in \eqref{equ:CNC_Verify_2} need to be established. By definition of $\Var[\hat{Q}_{\pi_\theta}(s_T,a_T)]$, we have 
\#\label{equ:CNC_Verify_2_imme_1}
\Var\big[\hat{Q}_{\pi_\theta}(s_T,a_T)\big]&=\EE_{T',(s_{1:T'},a_{1:T'})}\big[\hat{Q}_{\pi_\theta}(s_T,a_T)-{Q}_{\pi_\theta}(s_T,a_T)\big]^2\notag\\
&=\EE_{T'}\Big(\EE_{(s_{1:T'},a_{1:T'})}\big\{\big[\hat{Q}_{\pi_\theta}(s_T,a_T)-{Q}_{\pi_\theta}(s_T,a_T)\big]^2\biggiven T'=\tau\big\}\Big).
\#
Given $(s_T,a_T)$, $\theta$, and $T'=\tau$, the conditional expectation  in \eqref{equ:CNC_Verify_2_imme_1} can be expanded as 
\$
&\EE_{(s_{1:T'},a_{1:T'})}\big\{\big[\hat{Q}_{\pi_\theta}(s_T,a_T)-{Q}_{\pi_\theta}(s_T,a_T)\big]^2\biggiven T'=\tau\big\}\notag\\
&\quad=\EE_{(s_{1:T'},a_{1:T'})}\bigg\{\bigg[\sum_{t=0}^{T'}\gamma^{t/2}\cdot R(s_t,a_t)-{Q}_{\pi_\theta}(s_T,a_T)\bigg]^2\bigggiven T'=\tau\bigg\}.
\$
Now we first focus on the case when $R(s,a)$ are strictly positive, i.e., $R(s,a)\in[L_R,U_R]$.  
In this case,   ${Q}_{\pi_\theta}(s_T,a_T)$ is a scalar that lies in the bounded interval between $[L_R/(1-\gamma),U_R/(1-\gamma)]$.  
Also, notice that $\EE_{(s_{1:T'},a_{1:T'})}[\sum_{t=0}^{T'}\gamma^{t/2}\cdot R(s_t,a_t)]$ is a strictly increasing function of $T'$ since $R(s,a)\geq L_R>0$ for any $(s,a)$.
Moreover, notice that   given  $(s_T,a_T)$, $\EE_{T',(s_{1:T'},a_{1:T'})}[\sum_{t=0}^{T'}\gamma^{t/2}\cdot R(s_t,a_t)]$ is an unbiased estimate of ${Q}_{\pi_\theta}(s_T,a_T)$, and $T'$ follows the  geometric distribution  over non-negative support. Thus,  there must exist a finite $T_*\geq 0$, such that 
\#\label{equ:def_T_star}
\EE_{(s_{1:T_*},a_{1:T_*})}\bigg[\sum_{t=0}^{T_*}\gamma^{t/2}\cdot R(s_t,a_t)\bigg]< {Q}_{\pi_\theta}(s_T,a_T)\leq \EE_{(s_{1:T_*+1},a_{1:T_*+1})}\bigg[\sum_{t=0}^{T_*+1}\gamma^{t/2}\cdot R(s_t,a_t)\bigg]. 
\#
As a result, we can substitute \eqref{equ:def_T_star} into the right-hand side of \eqref{equ:CNC_Verify_2_imme_1},  yielding 
\#
&\Var\big[\hat{Q}_{\pi_\theta}(s_T,a_T)\big]=\sum_{\tau=0}^\infty \gamma^{\tau/2}(1-\gamma^{1/2})\cdot \EE_{(s_{1:\tau},a_{1:\tau})}\bigg[\sum_{t=0}^{\tau}\gamma^{t/2}\cdot R(s_t,a_t)-{Q}_{\pi_\theta}(s_T,a_T)\bigg]^2\notag\\
&\quad \geq \sum_{\tau=0}^\infty \gamma^{\tau/2}(1-\gamma^{1/2})\cdot \bigg\{\EE_{(s_{1:\tau},a_{1:\tau})}\bigg[\sum_{t=0}^{\tau}\gamma^{t/2}\cdot R(s_t,a_t)-{Q}_{\pi_\theta}(s_T,a_T)\bigg]\bigg\}^2\label{equ:CNC_Verify_2_imme_3}\\
&\quad \geq \sum_{\tau=0}^{T_*} \gamma^{\tau/2}(1-\gamma^{1/2})\cdot\bigg[L_R\cdot \sum_{t=\tau+1}^{T_*}\gamma^{t/2}\bigg]^2+\sum_{\tau=T_*+2}^{\infty} \gamma^{\tau/2}(1-\gamma^{1/2})\cdot\bigg[L_R\cdot \sum_{t=T_*+2}^{\tau}\gamma^{t/2}\bigg]^2,\label{equ:CNC_Verify_2_imme_4}
\#
where the first inequality  \eqref{equ:CNC_Verify_2_imme_3}  uses $\EE X^2\geq (\EE X)^2$, and the second inequality \eqref{equ:CNC_Verify_2_imme_4} follows by 
%separating the summation from $0$ to $\infty$ into that from $0$ to $T_*$ and from $T_*+2$ to $\infty$, and 
removing the term with $\tau=T_*$ and $\tau=T_*+1$ in the summation in \eqref{equ:CNC_Verify_2_imme_3} that sandwiched ${Q}_{\pi_\theta}(s_T,a_T)$, and  noticing  the fact that  the term $\EE_{(s_{1:\tau},a_{1:\tau})}[\sum_{t=0}^{\tau}\gamma^{t/2}\cdot R(s_t,a_t)]$ is at least $L_R\cdot \sum_{t=T_*+2}^{\tau}\gamma^{t/2}$ away from ${Q}_{\pi_\theta}(s_T,a_T)$ when $\tau\geq T_*+2$, and at least $L_R\cdot \sum_{t=\tau+1}^{T_*}\gamma^{t/2}$  away from\footnote{Note that we define $\sum_{t=\tau+1}^{T_*}\gamma^{t/2}=0$ if $\tau+1<T_*$.} ${Q}_{\pi_\theta}(s_T,a_T)$ when $\tau\leq T_*$. Furthermore, multiplying  the first term in \eqref{equ:CNC_Verify_2_imme_4}  by $\gamma^{3/2}$ yields   
\#
\Var\big[\hat{Q}_{\pi_\theta}(s_T,a_T)\big]&\geq \gamma^{3/2}\cdot \sum_{\tau=0}^{T_*} \gamma^{\tau/2}(1-\gamma^{1/2})\cdot\bigg[L_R\cdot \frac{\gamma^{(\tau+1)/2}-\gamma^{(T_*+1)/2}}{1-\gamma^{1/2}}\bigg]^2\notag\\
&\quad+\sum_{\tau=T_*+2}^{\infty} \gamma^{\tau/2}(1-\gamma^{1/2})\cdot\bigg[L_R\cdot \frac{\gamma^{(\tau+1)/2}-\gamma^{(T_*+2)/2}}{1-\gamma^{1/2}}\bigg]^2,\notag\\
&= \gamma^{3/2} \cdot \sum_{\tau=0}^{T_*} \gamma^{\tau/2}(1-\gamma^{1/2})\cdot\bigg[L_R\cdot \frac{\gamma^{(\tau+1)/2}-\gamma^{(T_*+1)/2}}{1-\gamma^{1/2}}\bigg]^2,\notag\\
&\quad+\gamma^{3/2} \cdot\sum_{\tau=T_*+1}^{\infty} \gamma^{\tau/2}(1-\gamma^{1/2})\cdot\bigg[L_R\cdot \frac{\gamma^{(\tau+1)/2}-\gamma^{(T_*+1)/2}}{1-\gamma^{1/2}}\bigg]^2,
\label{equ:CNC_Verify_2_imme_5}
\#
where the first inequality follows from the fact that $\gamma^{3/2}<1$, and the equality is obtained by changing the starting point of the summation of the second term to $T_*+1$, and then pulling out $\gamma^{1/2}$ from the 
square bracket. This way,  we can further bound \eqref{equ:CNC_Verify_2_imme_5}  as 
\#
&\Var\big[\hat{Q}_{\pi_\theta}(s_T,a_T)\big]\geq \gamma^{3/2} \cdot L_R^2\cdot \EE_{T'}\bigg[\frac{\gamma^{(T'+1)/2}-\gamma^{(T_*+1)/2}}{1-\gamma^{1/2}}\bigg]^2\notag\\ 
&\quad\geq \gamma^{3/2} \cdot L_R^2\cdot \Var\bigg[\frac{\gamma^{(T'+1)/2}-\gamma^{(T_*+1)/2}}{1-\gamma^{1/2}}\bigg]=\gamma^{3/2} \cdot L_R^2\cdot \Var\bigg[\frac{\gamma^{(T'+1)/2}}{1-\gamma^{1/2}}\bigg],\label{equ:CNC_Verify_2_imme_6}
\#
where the first inequality follows by expressing the right-hand side of  \eqref{equ:CNC_Verify_2_imme_5} as an expectation over $T'$, the second inequality follows from $\EE(X^2)\geq \Var(X)$, and the last equality is due to the fact  that $T_*$ is deterministic and thus does not affect the variance given $(s_T,a_T)$ and $\theta$. 
Note that $\Var[{\gamma^{(T'+1)/2}}]$ can be uniformly bounded as
\#
\Var[{\gamma^{(T'+1)/2}}]&=\EE[{\gamma^{(T'+1)/2}}]^2-\big\{\EE[{\gamma^{(T'+1)/2}}]\big\}^2=\frac{\gamma(1-\gamma^{1/2})}{1-\gamma^{3/2}}-\bigg[\frac{\gamma^{1/2}(1-\gamma^{1/2})}{1-\gamma}\bigg]^2\notag\\
&=\frac{\gamma^{3/2}\cdot (1-\gamma^{1/2})^3}{(1-\gamma^{3/2})\cdot(1-\gamma)^2}>0.\label{equ:CNC_Verify_2_imme_7}
\#
Combining \eqref{equ:CNC_Verify_2_imme_6} and \eqref{equ:CNC_Verify_2_imme_7}, we obtain
\#\label{equ:CNC_Verify_2_imme_8}
\Var\big[\hat{Q}_{\pi_\theta}(s_T,a_T)\big]\geq  \frac{\gamma^{3/2} \cdot L_R^2}{(1-\gamma^{1/2})^2}\cdot  \frac{\gamma^{3/2}\cdot (1-\gamma^{1/2})^3}{(1-\gamma^{3/2})\cdot(1-\gamma)^2}= \frac{L_R^2\cdot\gamma^{3} \cdot (1-\gamma^{1/2})}{(1-\gamma^{3/2})\cdot(1-\gamma)^2}.
\#
By the same arguments as  above, we can also obtain that 
\#\label{equ:CNC_Verify_2_imme_9}
\Var\big[\hat{V}_{\pi_\theta}(s_T)\big]\geq \frac{L_R^2\cdot\gamma^{3} \cdot (1-\gamma^{1/2})}{(1-\gamma^{3/2})\cdot(1-\gamma)^2}.
\#
Substituting \eqref{equ:CNC_Verify_2_imme_8} and \eqref{equ:CNC_Verify_2_imme_9} into \eqref{equ:CNC_Verify_2}, we arrive at 
\#\label{equ:CNC_Verify_2_final}
\EE_{T',(s_{1:T'},a_{1:T'})}\big[\hat{Q}_{\pi_\theta}(s_T,a_T)-\hat{V}_{\pi_\theta}(s_T)\big]^2\geq \frac{2L_R^2\cdot\gamma^{3} \cdot (1-\gamma^{1/2})}{(1-\gamma^{3/2})\cdot(1-\gamma)^2}.
\# 
Finally by combining \eqref{equ:CNC_Verify_2_final} and \eqref{equ:CNC_1_to_verify_2}, we conclude that
\$
\EE\big\{[\vb^\top \check{\nabla}     J(\theta)]^2\biggiven \theta\big\}\geq  \frac{2L_R^2\cdot\gamma^{3} \cdot (1-\gamma^{1/2})}{(1-\gamma^{3/2})\cdot(1-\gamma)^2}\cdot L_I=:\check{\eta}>0.
\$ 
The proof for the case when $R(s,a)\in[-U_R,-L_R]$ is as the one above,   with only some minor modifications due to sign flipping. For example, $\EE_{(s_{1:T'},a_{1:T'})}[\sum_{t=0}^{T'}\gamma^{t/2}\cdot R(s_t,a_t)]$ now becomes a strictly decreasing  function of $T'$ since $R(s,a)\leq -L_R<0$. The remaining arguments are similar, and are omitted  here to avoid repetition.

The proof of $\EE\big\{[\vb^\top \tilde{\nabla}J(\theta)]^2\biggiven \theta\big\}\geq \tilde{\eta}$ for some $\tilde{\eta}>0$ is very similar to the proofs above.  First, we have by definition that 
\#\label{equ:CNC_1_to_verify_3}
\EE\big\{[\vb^\top \tilde{\nabla}J(\theta)]^2\biggiven \theta\big\}&=\EE\big\{\big[R(s_T,a_T)+\gamma\cdot\hat{V}_{\pi_\theta}(s'_T)-\hat{V}_{\pi_\theta}(s_T)\big]^2\cdot\big[\vb^\top \nabla \log\pi_{\theta}(a_T\given s_T)\big]^2\biggiven \theta\big\}\\
&  =\EE_{T,(s_T,a_T)}\big\{\EE_{s_T',T',T'',(s_{1:T'},a_{1:T'}),(s_{1:T''},a_{1:T''})}\big[R(s_T,a_T)+\gamma\cdot\hat{V}_{\pi_\theta}(s'_T)-\hat{V}_{\pi_\theta}(s_T)\big]^2 \nonumber \\
& \qquad \times[\vb^\top \nabla \log[\pi_{\theta}(a_T\given s_T)]^2\biggiven \theta\big\},\notag
\#
where we use $T'$ and $T''$ to represent the random horizon used in calculating $\hat{V}_{\pi_\theta}(s'_T)$ and $\hat{V}_{\pi_\theta}(s_T)$, respectively, and recall that $s_T'$ is sampled from $\PP(\cdot\given s_T,a_T)$. Note that given $(s_T,a_T)$, we have  
\#
&\EE_{s_T',T',T'',(s_{1:T'},a_{1:T'}),(s_{1:T''},a_{1:T''})}\big[R(s_T,a_T)+\gamma\cdot\hat{V}_{\pi_\theta}(s'_T)-\hat{V}_{\pi_\theta}(s_T)\big]^2\notag\\
&\quad =\big\{\EE_{s_T',T',T'',(s_{1:T'},a_{1:T'}),(s_{1:T''},a_{1:T''})}\big[R(s_T,a_T)+\gamma\cdot\hat{V}_{\pi_\theta}(s'_T)-\hat{V}_{\pi_\theta}(s_T)\big]\big\}^2\notag\\
&\qquad\qquad+\gamma^2\cdot \Var[\hat{V}_{\pi_\theta}(s'_T)]+\Var[\hat{V}_{\pi_\theta}(s_T)]\label{equ:CNC_Verify_2_imme_9}\\
&\quad =\big[Q_{\pi_\theta}(s_T,a_T)-V_{\pi_\theta}(s_T)\big]^2+\gamma^2\cdot \Var[\hat{V}_{\pi_\theta}(s'_T)]+\Var[\hat{V}_{\pi_\theta}(s_T)]\label{equ:CNC_Verify_2_imme_10},
\#
where \eqref{equ:CNC_Verify_2_imme_9} and \eqref{equ:CNC_Verify_2_imme_10} are  due to the independence  and  unbiasedness of the estimates $\hat{V}_{\pi_\theta}(s'_T)$ and $\hat{V}_{\pi_\theta}(s_T)$, respectively.  Then, since the variance of  $\hat{V}_{\pi_\theta}(s_T)$ has been lower-bounded by  \eqref{equ:CNC_Verify_2_imme_8}, we can lower-bound   \eqref{equ:CNC_Verify_2_imme_10} and thus further bound \eqref{equ:CNC_1_to_verify_3} by 
\$
\EE\big\{[\vb^\top \tilde{\nabla}J(\theta)]^2\biggiven \theta\big\}\geq  \frac{(1+\gamma^2)\cdot L_R^2\cdot\gamma^{3} \cdot (1-\gamma^{1/2})}{(1-\gamma^{3/2})\cdot(1-\gamma)^2}\cdot L_I=:\tilde{\eta}>0,
\$ 
which completes the proof.
\end{proof}

%%%%%%%%%%%%%%%%%%%%%%%%%%%%%%%%%%%%%%%%%%%%%%%%%%%%%%%%%%%%%%%%%%%%%%%%%%%%%%%%%%%%%%%%%%%%%%%%%%%%%%%%%%%%%%%%%%%%%%%%%%%%%%%%%%%%%%%%%%%%%%%%%%%%%%%%%%%%%%%%%%%%%%%%%%%%%%%%%%%%%%%%%%%%%%%%%%%%%%%%%%%%% P R O O F%%%%%%%%%%%%%%%%%%%%%%%%%%%%%%%%%%%%%%%%%%%%%%%%%%%%%%%%%%%%%%%%%%%%%%%%%%%%%%%%%%%%%%%%%%%%%%%%%%%%%%%%%%%%%%%%%%%%%%%%%%%%%%%%%%%
\subsection{Proof of Theorem \ref{thm:conv_MRPG}}\label{apx_sosp}
\begin{proof}
We first note that we have listed the parameters to be used in our analysis  below in  {Table \ref{table:parameters}} in the main body of the paper, which will be referred to in this section.
 
	Now recall that in Algorithm \ref{alg:EvalPG}, we use $g_\theta$ to unify the notation of the three  stochastic policy gradients $\hat{\nabla}J(\theta)$, $\check{\nabla}J(\theta)$, and $\tilde{\nabla}J(\theta)$ (see the definitions in \eqref{equ:SGD_eva}-\eqref{equ:SGD_eva_3}).   From Theorem \ref{thm:unbiased_and_bnd_grad_est}, we know that all the three stochastic policy gradients are unbiased estimates  of $\nabla J(\theta)$. Moreover, we have shown that all the three stochastic policy gradients have their norms  bounded by some constants $\hat{\ell},\check{\ell}$, and $\tilde{\ell}>0$, respectively, 	which are defined  in Theorem  \ref{thm:unbiased_and_bnd_grad_est}. To unify the notation in the ensuing analysis, we use a common  $\ell$ to denote the bound of $g_\theta$, which takes the value of either $\hat{\ell},\check{\ell}$, or $\tilde{\ell}$, depending on which policy gradient is used. Also, as illustrated   in Lemma \ref{lemma:CNC_verify},  all the  three stochastic policy gradients  satisfy the correlated negative  curvature condition. We thus use a common $\eta$ to represent the value of $\hat{\eta},\check{\eta}$, and $\tilde{\eta}$ correspondingly. Therefore, we have
	\#\label{equ:common_bnd_CNC}
	\|g_\theta\|\leq \ell,\qquad \tx{and}\qquad \EE[(\vb_\theta^\top g_\theta)^2\given \theta]\geq \eta,\text{~~for any~~}\theta,
	\#
	where $\vb_\theta$ is the unit-norm eigenvector corresponding to the maximum eigenvalue of the Hessian at $\theta$. 
	In addition, recall from Lemmas  \ref{lemma:lip_policy_grad} and \ref{lemma:Hessian_Lip} that  $J(\theta)$ is both $L$-gradient Lipschitz and $\rho$-Hessian Lipschitz, i.e., there exist  constants $L$ and $\rho$ (see the  definitions in the corresponding lemmas), such that for any $\theta^1,\theta^2\in\RR^d$, 
	\#\label{equ:restate_Grad_H_Lip}
	&\|\nabla J(\theta^1)-\nabla J(\theta^2)\|\leq L\cdot \|\theta^1-\theta^2\|,\quad \big\|\cH(\theta^1)-\cH(\theta^2)\big\|\leq \rho \cdot \|\theta^1-\theta^2\|. 
	\#

{Our analysis  is separated into three steps that characterize the convergence properties of the iterates in three different regimes,  depending on the magnitude of the gradient and the curvature of the Hessian. This type of analysis for convergence to approximate  second-order stationary points in nonconvex optimization originated from \cite{ge2015escaping}, where isotropic noise is added to the  update to escape the saddle points. Here we do not assume that the stochastic policy gradient has isotropic noise, since : 1) in RL the noise  results from the sampling along the trajectory of the MDP, which do not necessarily satisfy the isotropic property in general; 2) the noise of policy gradients is notoriously known to be large, thus adding artificial  noise may further degrade the performance of the RPG algorithm. An effort to improve the limit points of first-order methods for nonconvex optimization, while avoiding adding artificial noise, has appeared recently in \cite{daneshmand2018escaping}. However, we have identified that the proof in \cite{daneshmand2018escaping} is flawed and cannot be applied directly for the convergence of the RPG algorithms here. Thus, part of our contribution here is to provide a precise fix in its own right, as well as map it to the analysis of policy gradient methods in RL. }

	Note that Algorithm \ref{alg:MRPG} returns the iterates that have indices $k$ such that $k\tx{~mod~}k_{\tx{thre}}=0$, i.e., the iterates belong to the set $\hat{\Theta}^*$. For notational convenience, we index the iterates in $\hat{\Theta}^*$ by $m$, i.e., let $\ttheta_m=\theta_{m\cdot k_{\tx{thre}}}$ for all $m=0,1,\cdots,\lfloor K/k_{\tx{thre}}\rfloor$.  Now we consider the three regimes of the iterates $\{\ttheta_m\}_{m\geq 0}$.

	\vspace{7pt}
	\noindent\textbf{Regime 1: Large gradient}
	\vspace{2pt}
	
	We first introduce the following standard lemma that quantifies the increase of function values, when  stochastic gradient ascent of a smooth function is adopted.
	
	\begin{lemma}\label{lemma:SGA_increase}
		Let $\theta_{k+1}$ be obtained by one stochastic gradient ascent step at $\theta_k$, i.e., $\theta_{k+1}=\theta_{k}+\alpha g_{k}$, where $g_k=g_{\theta_k}$ is an unbiased stochastic gradient at $\theta_k$.
%		 that satisfies $\EE(\|g_{k}\|^2)\leq \ell^2$ for some $\ell$. Also, suppose $J(\theta)$ is $L$-gradient Lipschitz.  
Then, for  any given $\theta_k$, 
the function value $J(\theta_{k+1})$ increases in expectation\footnote{{Note that the expectation here is taken over the randomness of $g_k$.}}  as 
		\$
		\EE[J(\theta_{k+1})]-J(\theta_{k})\geq \alpha \|\nabla J(\theta_{k})\|^2-\frac{L\alpha^2\ell^2}{2}.
		\$
	\end{lemma}
%	\kznote{OLD note: I think this should be correct, since this decrease holds for any given $\theta_k$. It looks too long and not very necessary to have everything conditioned on $\cF_k$. Both sides are just functions  of $\theta_k$. In fact, if we take expectation over $J(\theta_k)$, too, it weakens the argument, since right now what we have is "almost sure", not in expectation. }
%\textcolor{blue}{(new comment) Alec says: I disagree. To have a submartingale relationship, we need to have things conditioned on the filtration. Otherwise it does not fit the hypotheses of the submartingale convergence theorem. Please make sure this is done throughout.}
	\begin{proof}
	By the $L$-smoothness of $J(\theta)$, we have
	\$
	\EE[J(\theta_{k+1})]-J(\theta_{k})\geq \alpha \nabla J(\theta_{k})^\top \EE(g_k\given \theta_k)-\frac{L\alpha^2}{2}\|g_k\|^2= \alpha \|\nabla J(\theta_{k})\|^2-\frac{L\alpha^2}{2}\|g_k\|^2,
	\$
	which completes the proof by using the fact that $\|g_k\|^2\leq \ell^2$ almost surely.
	\end{proof}

%%
% \begin{enumerate}
%%
%\item \textcolor{blue}{Either $J(\theta_{k})$ should be $\mathbb{E}[J(\theta_{k})]$ and $\nabla J(\theta_{k})$ should be $\mathbb{E}[\nabla J(\theta_{k})]$}
%%
% \item  \textcolor{blue}{ or $\EE[J(\theta_{k+1})]$ should be $\EE[J(\theta_{k+1})\given \mathcal{F}_k ]$}
% %
% \end{enumerate}
 
% \textcolor{blue}{This seems to also be a problem in \cite{daneshmand2018escaping} . As is, the expression in Lemma \ref{lemma:SGA_increase} is confusing. This change needs to be propagated throughout all the analysis...}
	
	Therefore, when the  norm of the gradient is large at $\ttheta_m$, a large increase of $J(\ttheta)$ from $\ttheta_m$ to $\ttheta_{m+1}$ is guaranteed, as formally stated in the following lemma.
	
	\begin{lemma}\label{lemma:increase_large_grad}
		Suppose the gradient norm at any given $\ttheta_m$ is large such that $\|\nabla J(\ttheta_m)\|\geq \epsilon$, for some $\epsilon > 0$. Then, the expected value of $J(\ttheta_{m+1})$   increases as 
		\$
		\EE[J(\ttheta_{m+1})]-J(\ttheta_{m})\geq J_{\tx{thre}},
		\$
		where the expectation is taken over the sequence from $\theta_{m\cdot k_{\tx{thre}}+1}$ to $\theta_{(m+1)\cdot k_{\tx{thre}}}$ 
%		\textcolor{blue}{We should define specifically what filtration we are talking about... Also needs to be corrected for Lemma \ref{lemma:saddle_escape}.}.   
	\end{lemma}
	\begin{proof}
	We first decompose the difference between the expected value of $J(\ttheta_{m+1})$ and $J(\ttheta_{m})$ as
	\$
	\EE[J(\ttheta_{m+1})]-J(\ttheta_{m})&=\sum_{p=0}^{k_{\tx{thre}}-1}\EE[J(\theta_{m\cdot k_{\tx{thre}}+p+1})]-\EE[J(\theta_{m\cdot k_{\tx{thre}}+p})]\\
	&=\sum_{p=0}^{k_{\tx{thre}}-1}\EE\big\{\EE[J(\theta_{m\cdot k_{\tx{thre}}+p+1})]-J(\theta_{m\cdot k_{\tx{thre}}+p})\biggiven \theta_{m\cdot k_{\tx{thre}}+p}\big\},
	\$
%\textcolor{blue}{where we use the Law of Total Expectation to rewrite the preceding expression in terms of conditional expectations, using the fact that $J(\theta_{m\cdot k_{\tx{thre}}+p})$ is deterministic given filtration ??????}.	
{where $\EE[J(\theta_{m\cdot k_{\tx{thre}}})\given \theta_{m\cdot k_{\tx{thre}}}]=J(\theta_{m\cdot k_{\tx{thre}}})=J(\ttheta_{m})$ for given $\ttheta_m$}.
By Lemma \ref{lemma:SGA_increase}, we further have
	\#
	&\EE[J(\ttheta_{m+1})]-J(\ttheta_{m})\geq \beta \|\nabla J(\theta_{m\cdot k_{\tx{thre}}})\|^2-\frac{L\beta^2\ell^2}{2}+\sum_{p=1}^{k_{\tx{thre}}-1}\alpha  \EE\|\nabla J(\theta_{m\cdot k_{\tx{thre}}+p})\|^2-\frac{k_{\tx{thre}}L\alpha^2\ell^2}{2}\notag\\
	&\quad\geq \beta \|\nabla J(\theta_{m\cdot k_{\tx{thre}}})\|^2-\frac{L\beta^2\ell^2}{2}-\frac{k_{\tx{thre}}L\alpha^2\ell^2}{2}\geq \beta \|\nabla J(\theta_{m\cdot k_{\tx{thre}}})\|^2-{L\beta^2\ell^2},\label{equ:large_grad_immed_1}
	\#
	where the last inequality follows from {Table \ref{table:parameters}} that 
	\#\label{equ:large_grad_immed_1p25}
	\beta^2\geq k_{\tx{thre}}\cdot \alpha^2.
	\#
	Moreover, by the choice of the large stepsize $\beta$, we have 
	\#\label{equ:large_grad_immed_1p5}
	\|\nabla J(\theta_{m\cdot k_{\tx{thre}}})\|^2=\|\nabla J(\ttheta_m)\|^2\geq \epsilon^2\geq 2\ell^2L\beta,
	\#
	which yields a lower-bound on the right-hand side of \eqref{equ:large_grad_immed_1} as
	\#\label{equ:large_grad_immed_2}
	\EE[J(\ttheta_{m+1})]-J(\ttheta_{m})\geq \beta \|\nabla J(\theta_{m\cdot k_{\tx{thre}}})\|^2-{L\beta^2\ell^2}\geq \beta \|\nabla J(\theta_{m\cdot k_{\tx{thre}}})\|^2/2\geq \beta \epsilon^2/2\geq J_{\tx{thre}}.
	\#
	The choice of $J_{\tx{thre}}\leq \beta \epsilon^2/2$ completes the proof.
	\end{proof}

	\vspace{7pt}
	\noindent\textbf{Regime 2: Near  saddle points}  
	\vspace{2pt}    
	
	When the iterate reaches the neighborhood of saddle points, our modified RPG will use a larger stepsize $\beta$ to find the positive eigenvalue direction, and then uses small stepsize $\alpha$ to follow this positive curvature direction. We establish in the following lemma that  such an updating strategy also leads to a sufficient increase of function value, provided that the maximum eigenvalue of the Hessian $\cH(\ttheta_m)$ is large enough. This enables the iterate to escape the saddle points efficiently. 
	
	\begin{lemma}\label{lemma:saddle_escape}
		Suppose that the Hessian at any given $\ttheta_m$  has a large positive eigenvalue such that $\lambda_{\max}[\cH(\ttheta_m)]\geq \sqrt{\rho\epsilon}$. Then, after $k_{\tx{thre}}$ steps we have 
%		there exists a $\tau\leq k_{\tx{thre}}$ such that the expected function value increases as
		\$
		\EE[J(\ttheta_{m+1})]-J(\ttheta_{m})\geq J_{\tx{thre}},
		\$
		where the expectation is taken over the sequence from $\theta_{m\cdot k_{\tx{thre}}+1}$ to $\theta_{(m+1)\cdot k_{\tx{thre}}}$. 
	\end{lemma}
	
	 Lemma \ref{lemma:saddle_escape} asserts that after  $k_{\tx{thre}}$ steps, the expected function value increases by  at least $J_{\tx{thre}}$. 
	 Together with Lemma \ref{lemma:increase_large_grad}, it can be shown that the expected return $\EE[J(\ttheta_{m+1})]$ is always increasing, as long as the iterate $\ttheta_m$ violates the approximate second-order stationary point condition, i.e., $\|\nabla J(\ttheta_m)\|\geq \epsilon$ or $\lambda_{\max}[\cH(\ttheta_m)]\geq \sqrt{\rho\epsilon}$. 
	 The proof of Lemma \ref{lemma:saddle_escape} is deferred to \S\ref{sec:proof_lemma:saddle_escape} to maintain the flow here.
	 	
	\vspace{7pt}
	\noindent\textbf{Regime 3: Near second-order stationary points}
	\vspace{2pt}
	
	When the iterate converges to the neighborhood of the desired second-order stationary points, both the norm of the gradient and the largest eigenvalue are small. However, due to the variance of the stochastic policy gradient, the function value may still decrease. By Lemma \ref{lemma:SGA_increase} and \eqref{equ:large_grad_immed_2}, we can immediately show that such a decrease is  bounded, i.e.,
	\#\label{equ:dec_bnd_near_SP}
	&\EE[J(\ttheta_{m+1})]-J(\ttheta_{m})\geq -{L\beta^2\ell^2}\geq-\delta J_{\tx{thre}}/2,
	\#
	which is due to the choice of $J_{\tx{thre}}\geq 2L(\ell\beta)^2/\delta$ as in {Table \ref{table:parameters}}.

	Now we combine the arguments above to obtain a probabilistic lower-bound   on the returned approximate  second-order stationary point. Let $\cE_m$ be the event that 
	\$
	\cE_m:=\{\|\nabla J(\ttheta_m)\|\geq \epsilon \tx{~or~} \lambda_{\max}[\cH(\ttheta_m)]\geq \sqrt{\rho\epsilon}\}.
	\$
	By Lemmas \ref{lemma:increase_large_grad} and \ref{lemma:saddle_escape}, we have
	\#\label{equ:comb_inc}
	\EE[J(\ttheta_{m+1})-J(\ttheta_{m})\given \cE_m]\geq J_{\tx{thre}},
	\#
	{where the  expectation is taken over  the randomness of both $\ttheta_{m+1}$ and $\ttheta_{m}$ given the event $\cE_m$.} 
	Namely, after $k_{\tx{thre}}$ steps, as long as $\ttheta_{m}$ is not an $(\epsilon,\sqrt{\rho\epsilon})$-approximate second-order stationary point, a sufficient increase of $\EE[J(\ttheta_{m+1})]$ is guaranteed. Otherwise, we can still control the possible decrease of the return using \eqref{equ:dec_bnd_near_SP}, which yields
	\#\label{equ:comb_dec}
	\EE[J(\ttheta_{m+1})-J(\ttheta_{m})\given \cE_m^c]\geq -\delta J_{\tx{thre}}/2,
	\#
	where $\cE_m^c$ is the complement event of $\cE_m$.

%	Combining \eqref{equ:large_grad} and \eqref{equ:amortized_inc}, we know that the MRPG update guarantees the increase of expected function value, on average, of at least $g_{\tx{thre}}$ per step, i.e., 
%
%	In addition, by Lemma \ref{lemma:SGA_increase}, we obtain that for one update step with large stepsize $\beta$ (which occurs once every $k_{\tx{thre}}$ steps), the decrease of the function value can be bounded, namely, 
%	\$
%	\EE[J(\theta_{k+k_{\tx{thre}}})]-J(\theta_{k})\geq \alpha \|\nabla J(\theta_{k})\|^2-\frac{L\alpha^2\ell^2}{2}\geq -\frac{L\alpha^2\ell^2}{2}\geq -\frac{\delta J_{\tx{thre}}}{4}.
%	\$
%	Therefore, on average, the per step decrease among every $k_{\tx{thre}}$ steps is
%	\#\label{equ:large_step_dec}
%	\frac{\EE[J(\theta_{k+k_{\tx{thre}}})]-J(\theta_{k})}{k_{\tx{thre}}}\geq -\frac{\delta J_{\tx{thre}}}{4t_\tx{thre}}=-\frac{\delta g_{\tx{thre}}}{4}.
%	\#
%	Combining \eqref{equ:dec_bnd_near_SP}  and  \eqref{equ:large_step_dec}, we can control the average  per step decrease of the expected function value caused by the stochastic gradient steps (using both large and small stepsizes) as
%	\#\label{equ:comb_dec}
%	\EE[J(\theta_{k+1})-J(\theta_{k})\given \cE_k^c]\geq -\frac{\delta g_{\tx{thre}}}{2}.
%	\#

Let $\cP_m$ denote the probability of the occurrence of the event $\cE_m$. Thus, the total expectation  $\EE[J(\ttheta_{m+1})-J(\ttheta_{m})]$ can be obtained by combining \eqref{equ:comb_inc} and \eqref{equ:comb_dec} as follows
\#\label{equ:comb_inc_tt_exp}
\EE[J(\ttheta_{m+1})-J(\ttheta_{m})]\geq (1-\cP_m)\cdot \bigg(-\frac{\delta J_{\tx{thre}}}{2}\bigg) +\cP_m\cdot J_{\tx{thre}}.
\#
Suppose the iterate of $\theta_k$ runs for $K$ steps starting from $\theta_0$;  then there are  $M=\lfloor K/k_{\tx{thre}} \rfloor$ of $\ttheta_m$ for  $k=1,\cdots,K$. 
Summing up all the $M$ steps of $\{\ttheta_m\}_{m=1,\cdots,M}$, we obtain from \eqref{equ:comb_inc_tt_exp} that
\$
\frac{1}{M}\sum_{m=1}^M \cP_m\leq \frac{J^*-J(\theta_0)}{M J_{\tx{thre}}}+\frac{\delta}{2}\leq \delta,
\$
where $J^*$ is the global maximum of $J(\theta)$, and the last inequality follows from the choice of $K$ in {Table \ref{table:parameters}} that satisfies 
\#\label{equ:K_lower_bnd}
K\geq 2[{J^*-J(\theta_0)}]k_{\tx{thre}}/({\delta J_{\tx{thre}}}). 
\#  
Therefore, the probability of the event $\cE_m^c$ occurs, i.e., the probability of retrieving an $(\epsilon,\sqrt{\rho\epsilon})$ approximate second-order stationary point uniformly over the iterates in $\hat{\Theta}^*$, can be lower-bounded by
\$
1-\frac{1}{M}\sum_{m=1}^M \cP_m\geq  1-\delta. 
\$
This completes the proof. 
\end{proof}

\subsection{Proof of Lemma \ref{lemma:saddle_escape}}\label{sec:proof_lemma:saddle_escape}
%\textcolor{blue}{This proof is super technical, and hard to follow. Please give a sketch of the main ideas, and explain how we are trying to derive a contradiction... See the text below in blue before \eqref{equ:def_Q_approx}.}
\begin{proof}
The proof is based on the \emph{improve or localize} framework proposed in \cite{jin2017accelerated}.
{  
The basic idea is as follows: starting from some iterate, 
if the following iterates of stochastic gradient update do not improve the objective value to a great degree, then the iterates must not move much from the starting iterate. 
Our goal here is to show that after $k_{\tx{thre}}$ steps, the objective value will increase by at least $J_{\tx{thre}}$. In particular, the proof proceeds by contradiction: suppose the objective value does not increase  by $J_{\tx{thre}}$ from $\ttheta_m$ to $\ttheta_{m+1}$, then the distance between the two iterates can be upper-bounded by a polynomial function of the number of iterates in between, i.e., $k_{\tx{thre}}$. On the other hand, due to the CNC condition (cf. Lemma \ref{lemma:CNC_verify}), the distance between $\ttheta_m$ and $\ttheta_{m+1}$ can be shown to be lower-bounded by an exponential function of $k_{\tx{thre}}$. This way, by choosing large enough $k_{\tx{thre}}$ following Table \ref{table:parameters}, the lower-bound exceeds the upper-bound, which causes a contradiction and justifies our argument. }

First, for notational convenience, we suppose $m=k=0$ without loss of generality, 
%\kznote{I think probably this is the source why you feel confused. In fact, all the ensuing analysis happens "within one period", thus, the starting point does not matter. For notational convenience, we set $m=k=0$. The starting point $\theta_0$ is fixed, and can be any value.}%
and 
denote 
\#\label{equ:obj_grad_hess_shorthand}
J_p = J(\theta_p),\quad \nabla J_p = \nabla J(\theta_p),\quad \cH_p= \nabla^2 J(\theta_p),
\#
for any $p=0,\cdots,k_{\tx{thre}}-1$.  
%The proof is based on contradiction. 
Suppose that starting from $\ttheta_0$, after $1$ step iteration with large stepsize $\beta$ and $k_{\tx{thre}}-1$ steps with small stepsize $\alpha$, the expected return value  does not increase by   more than $J_{\tx{thre}}$, i.e., 
	\#\label{equ:obj_contra}
	\EE(J_{k_{\tx{thre}}})-J_0\leq J_{\tx{thre}}. 
	\#
Then, for any $0\leq p\leq k_{\tx{thre}}$, we can establish that the expectation of the distance from $\theta_{p}$ to $\theta_0$ is upper-bounded, as formally stated in the following lemma.

	\begin{lemma}\label{lemma:distance_bnd}
		Given any $\theta_0$, suppose \eqref{equ:obj_contra} holds, 
		   for any $0\leq p\leq k_{\tx{thre}}$. Then,     the  expected distance between $\theta_{p}$ and $\theta_0$  can be upper-bounded as 
		  \#
		&\EE\|\theta_{p}-\theta_0\|^2\leq\big[4\alpha^2\ell_g^2+4\alpha J_{\tx{thre}}+2{L\alpha(\ell \beta)^2}+2{L\ell^2\alpha^3 k_{\tx{thre}}}\big]\cdot p+2\beta^2\ell^2,\label{equ:uppbnd_res_2}
		\#
		\begin{comment} 
		\$
		\EE(\|\theta_{k_{\tx{thre}}}-\theta_0\|^2)&\leq \big[8\alpha^2\ell^2+4\alpha J_{\tx{thre}}+2{L\alpha(\ell \beta)^2}+4\alpha^3L\ell^2\big]\cdot k_{\tx{thre}}+2{L\ell^2\alpha^3\cdot k_{\tx{thre}}^2}\\
		&\qquad\quad+4\beta^2\ell^2+8\alpha J_{\tx{thre}}+4\alpha L\beta^2\ell^2. 
%		\big[4\alpha^2\ell^2+4\alpha J_{\tx{thre}}+2{L\alpha(\ell \beta)^2}+2{L\ell^2\alpha^3}k_{\tx{thre}}\big]\cdot k_{\tx{thre}}+4\beta^2\ell^2.
		\$
		\end{comment}
		where $\ell_g^2:=2\ell^2+{2B^2_\Theta U^2_R}\cdot{(1-\gamma)^{-4}}$.
	\end{lemma}
	\begin{proof}
		We have obtained from Lemma \ref{lemma:SGA_increase} and \eqref{equ:large_grad_immed_1} (with $m=0$) that 
	\#
	&\EE(J_{k_{\tx{thre}}})-J_0\geq \beta \|\nabla J_0\|^2-\frac{L\beta^2\ell^2}{2}+\sum_{q=1}^{k_{\tx{thre}}-1}\alpha  \EE\|\nabla J_q\|^2-\frac{k_{\tx{thre}}L\alpha^2\ell^2}{2}\notag\\
	&\quad\geq -\frac{L\beta^2\ell^2}{2}+\alpha\sum_{q=0}^{p-1}  \EE\|\nabla J_q\|^2-\frac{k_{\tx{thre}}L\alpha^2\ell^2}{2},\notag
	\#
	since $0\leq \alpha<\beta$ and $0\leq p\leq k_{\tx{thre}}$, where {we note that the total expectation is taken along  the sequence from $\theta_1$ to $\theta_{k_\tx{thre}}$, and we write $\|\nabla J_0\|^2=\EE\|\nabla J_0\|^2$ since $\theta_0$ is given and deterministic}. 
	Combined with  \eqref{equ:obj_contra}, we have 
	\$
	J_{\tx{thre}}\geq \alpha\sum_{q=0}^{p-1}  \EE\|\nabla J_q\|^2-\frac{k_{\tx{thre}}L\alpha^2\ell^2}{2}-\frac{L\beta^2\ell^2}{2},
	\$
	which implies that  
	\#\label{equ:obj_contra_immed_1}
	\sum_{q=0}^{p-1}\EE\|\nabla J_q\|^2\leq  \frac{J_{\tx{thre}}}{\alpha}+\frac{k_{\tx{thre}}L\alpha\ell^2}{2}+\frac{L\beta^2\ell^2}{2\alpha}. 
	\#
	Now, let us consider the distance  between $\theta_{p}$ and $\theta_0$ that can be decomposed as follows: 
	\#\label{equ:bnd_distance_immed_1}
	\EE\|\theta_{p}-\theta_0\|^2=\EE\bigg\|\sum_{q=0}^{p-1}\theta_{q+1}-\theta_{q}\bigg\|^2\leq 2\alpha^2\EE\bigg\|\sum_{q=1}^{p-1}g_q\bigg\|^2+2\beta^2\EE\big\|g_{0}\big\|^2,
\#
where the first equality comes from the telescopic property of the summand and the later inequality comes from $\|a+b\|^2\leq 2\|a\|^2+2\|b\|^2$. 
%\kznote{The argument here is correct, if we agrees that all the full expectations are taken along the sequence starting from $\theta_0$; and $\theta_0$, as a starting point, is given and deterministic.}%
%\textcolor{blue}{I am confused. Before we had an UNCONDITIONAL expectation. Below we have a conditional one. We need to be careful and consistent with how we are computing expected values...Please clarify} 
%On the other hand, if $n\geq 1$, we have
%\#\label{equ:bnd_distance_immed_15}
%	\EE(\|\theta_{p}-\theta_n\|^2\given \cF_n)=\EE\Bigg(\bigg\|\sum_{q=n}^{p-1}\theta_{q+1}-\theta_{q}\bigg\|^2\bigggiven \cF_n\Bigg)= \alpha^2\EE\Bigg(\bigg\|\sum_{q=n}^{p-1}g_q\bigg\|^2\bigggiven\cF_n\Bigg),
%\#
%where recall that $g_q=g_{\theta_q}$ is the stochastic policy gradient at $\theta_q$.

For the first term on the right-hand side of \eqref{equ:bnd_distance_immed_1},  we have  
	\#\label{equ:bnd_distance_immed_2}
	&2\alpha^2\EE \bigg\|\sum_{q=1}^{p-1}g_q\bigg\|^2= 2\alpha^2\EE\bigg\|\sum_{q=1}^{p-1}g_q-\nabla J_q+\nabla J_q\bigg\|^2\notag\\
	&\quad \leq 4\alpha^2\EE\bigg\|\sum_{q=1}^{p-1}g_q-\nabla J_q\bigg\|^2 
%	+2\alpha^2\sum_{p,q=1}^{p-1}\nabla J_p^\top\EE(g_q-\nabla J_q)
+4\alpha^2\EE\bigg\|\sum_{q=1}^{p-1}\nabla J_q\bigg\|^2\notag\\
	&\quad =4\alpha^2\EE\sum_{q=1}^{p-1}\big\|g_q-\nabla J_q\big\|^2+4\alpha^2\EE\bigg\|\sum_{q=1}^{p-1}\nabla J_q\bigg\|^2,
	\#
	where the first inequality follows from $\|a+b\|^2\leq 2\|a\|^2+2\|b\|^2$, and the 
	last equality  uses the fact that  $\EE[(g_p-\nabla J_p)^\top(g_q-\nabla J_q)]=0$  for any  $p\neq q$,  since the stochastic error $g_p-\nabla J_q$ across iterations are independent, and $g_p$ is an  unbiased estimate of $\nabla J_p$.  
	Moreover,  due to the boundedness of $\|\nabla J_q\|$ and $\|g_q\|$ for any value of $\theta_q$ (cf. Theorem  \ref{thm:unbiased_and_bnd_grad_est}),  we have
	\$
	\EE\big\|g_q-\nabla J_q\big\|^2&\leq 2\EE\big\|g_q\big\|^2+2\EE\big\|\nabla J_q\big\|^2\\
	&\leq 2\ell^2+\frac{2B^2_\Theta U^2_R}{(1-\gamma)^4}=:\ell_g^2.
	\$
	Thus, by the Cauchy-Schwarz inequality and \eqref{equ:obj_contra_immed_1}, we can further upper-bound the right-hand side of \eqref{equ:bnd_distance_immed_2} as 
	\#\label{equ:bnd_distance_immed_3}
	&2\alpha^2\EE\bigg\|\sum_{q=1}^{p-1}g_q\bigg\|^2\leq  4\alpha^2\sum_{q=1}^{p-1}\EE\big\|g_q-\nabla J_q\big\|^2+4\alpha^2\EE\bigg\|\sum_{q=1}^{p-1}\nabla J_q\bigg\|^2\notag\\
	&\quad\leq 4\alpha^2\cdot (p-1)\cdot \ell_g^2+4\alpha^2\cdot(p-1)\cdot\sum_{q=1}^{p-1}\EE\big\|\nabla J_q\big\|^2\notag\\
	&\quad \leq 4\alpha^2\cdot (p-1)\cdot \ell_g^2+4\alpha^2\cdot(p-1)\cdot\bigg(\frac{J_{\tx{thre}}}{\alpha}+\frac{k_{\tx{thre}}L\alpha\ell^2}{2}+\frac{L\beta^2\ell^2}{2\alpha}\bigg),
%	\\
%	&\quad\leq 4\alpha^2\cdot (p-1)\cdot \bigg(\ell^2+\frac{J_{\tx{thre}}}{\alpha}+\frac{pL\alpha\ell^2}{2}+\frac{L\beta^2\ell^2}{2\alpha}\bigg).
	\#
	where we  recall that the expectation is taken over the random sequence $\{\theta_1,\cdots,\theta_{p-1}\}$.
%which may be substituted into the right-hand of \eqref{equ:bnd_distance_immed_15} to obtain
%	\$
%	\EE(\|\theta_{p}-\theta_n\|^2\given \theta_n)\leq 2\alpha^2\cdot (p-n)\cdot \ell_g^2+2\alpha^2\cdot(p-n)\cdot\bigg(\frac{J_{\tx{thre}}}{\alpha}+\frac{k_{\tx{thre}}L\alpha\ell^2}{2}+\frac{L\beta^2\ell^2}{2\alpha}\bigg).
%	\$

	\begin{comment}
	\#\label{equ:bnd_distance_immed_3}
	&2\alpha^2\EE\bigg\|\sum_{q=1}^{k_{\tx{thre}}-1}g_q\bigg\|^2\leq  8\alpha^2\EE\bigg(\sum_{q=1}^{k_{\tx{thre}}-1}\big\|g_q\big\|^2\bigg)+8\alpha^2\EE\bigg(\sum_{q=1}^{k_{\tx{thre}}-1}\big\|\nabla J_q\big\|^2\bigg)+4\alpha^2\EE\bigg\|\sum_{q=1}^{k_{\tx{thre}}-1}\nabla J_q\bigg\|^2\\
	&\quad\leq 8\alpha^2\cdot (k_{\tx{thre}}-1)\cdot \ell^2+[8\alpha^2+4\alpha^2\cdot(k_{\tx{thre}}-1)]\cdot\sum_{q=1}^{k_{\tx{thre}}-1}\EE\big\|\nabla J_q\big\|^2\notag\\
	&\quad \leq 8\alpha^2\cdot (k_{\tx{thre}}-1)\cdot \ell^2+[8\alpha^2+4\alpha^2\cdot(k_{\tx{thre}}-1)]\cdot\bigg(\frac{J_{\tx{thre}}}{\alpha}+\frac{k_{\tx{thre}}L\alpha\ell^2}{2}+\frac{L\beta^2\ell^2}{2\alpha}\bigg).\notag
%	\\
%	&\quad\leq 4\alpha^2\cdot (k_{\tx{thre}}-1)\cdot \bigg(\ell^2+\frac{J_{\tx{thre}}}{\alpha}+\frac{k_{\tx{thre}}L\alpha\ell^2}{2}+\frac{L\beta^2\ell^2}{2\alpha}\bigg).
	\#
	\end{comment}
	
	For the second term on the right-hand side of \eqref{equ:bnd_distance_immed_1}, observe that $\EE\|g_{0}\|^2	\leq \ell^2$. Therefore, combined with \eqref{equ:bnd_distance_immed_3}, we may upper estimate \eqref{equ:bnd_distance_immed_1} as 
	\$
	&\EE(\|\theta_{p}-\theta_0\|^2)\leq 4\alpha^2\cdot (p-1)\cdot \ell_g^2+4\alpha^2\cdot(p-1)\cdot\bigg(\frac{J_{\tx{thre}}}{\alpha}+\frac{k_{\tx{thre}}L\alpha\ell^2}{2}+\frac{L\beta^2\ell^2}{2\alpha}\bigg)+2\beta^2\ell^2\\
	&\quad \leq \bigg[4\alpha^2\cdot  \ell_g^2+4\alpha^2\cdot\bigg(\frac{J_{\tx{thre}}}{\alpha}+\frac{k_{\tx{thre}}L\alpha\ell^2}{2}+\frac{L\beta^2\ell^2}{2\alpha}\bigg)\bigg]\cdot p+2\beta^2\ell^2,
	\$
	\begin{comment}
	\$
	&\EE(\|\theta_{k_{\tx{thre}}}-\theta_0\|^2)\leq 8\alpha^2\cdot (k_{\tx{thre}}-1)\cdot \ell^2+[8\alpha^2+4\alpha^2\cdot(k_{\tx{thre}}-1)]\cdot\bigg(\frac{J_{\tx{thre}}}{\alpha}+\frac{k_{\tx{thre}}L\alpha\ell^2}{2}+\frac{L\beta^2\ell^2}{2\alpha}\bigg)\\
	&\leq  \big[8\alpha^2\ell^2+4\alpha J_{\tx{thre}}+2{L\alpha(\ell \beta)^2}+4\alpha^3L\ell^2\big]\cdot k_{\tx{thre}}+2{L\ell^2\alpha^3\cdot k_{\tx{thre}}^2}+4\beta^2\ell^2+8\alpha J_{\tx{thre}}+4\alpha L\beta^2\ell^2,
	\$
	\end{comment}
	which completes the proof. 	
\end{proof}

	By substituting $q=k_{\tx{thre}}$, 	Lemma \ref{lemma:distance_bnd} asserts that the  expected distance from $\theta_{k_{\tx{thre}}}$ to $\theta_0$ is upper-bounded by a quadratic function of $k_{\tx{thre}}$. 
	{As illustrated at the beginning of the proof, we proceed   by providing a lower-bound on this distance, and show that the lower-bound exceeds the upper-bound given in Lemma \ref{lemma:distance_bnd}. As a result, the assumption that  
	\eqref{equ:obj_contra} holds is not true, which implies a sufficient increase of no less than $J_{\tx{thre}}$ from $J_0$ to $\EE(J_{k_{\tx{thre}}})$. 
	}
%	\textcolor{blue}{Here is the first time we have mentioned a contradiction. Can you explain what is the role of this contradiction and why is it necessary in the broader context of the analysis? It seems to be missing from the outline of the reasoning/proof we had before... Therefore, I am not sure what is the point of the subsequent lemmas and am getting lost here. \\ \\ What I specifically want to see is: suppose (something holds). Then we will derive the fact that this contradicts (some condition). Then, as a result, (something doesn't hold), which implies (our result.) Somehow, this is what we are trying to say with \eqref{equ:dec_bnd_near_SP} - \eqref{equ:comb_dec} but I didn't get it.}
	
	 To create such a  lower-bound, we first note that  for any $\theta$ close to $\theta_0$, the function value $J(\theta)$ can be approximated by some quadratic function $\cQ(\theta)$, i.e.,
	\#\label{equ:def_Q_approx}
	\cQ(\theta)=J_0+(\theta-\theta_0)^\top \nabla J_0+\frac{1}{2}(\theta-\theta_0)^\top \cH_0(\theta-\theta_0).
	\#
	This way, one can then bound the difference between the gradients of $J$ and $\cQ$ in the following lemma.
	\begin{lemma}[\cite{nesterov2013introductory}]\label{lemma:taylor_exp_error}
		For any twice-differentiable, $\rho$-Hessian Lipschitz function $J:\RR^d\to \RR$, using the quadratic approximation in \eqref{equ:def_Q_approx}, the following bound holds
		\$
		\|\nabla J(\theta)-\nabla \cQ(\theta)\|\leq \frac{\rho}{2}\cdot\|\theta-\theta_0\|^2.
		\$
	\end{lemma}

	For convenience, we let $\nabla \cQ_p=\nabla \cQ(\theta_p)$ for any $p=0,\cdots,k_{\tx{thre}}-1$. 
%	\textcolor{blue}{Why are we doing this decomposition of $\theta_{p+1}-\theta_0$? More connective tissue needed.} 
{Then, we can express the difference between any $\theta$ and $\theta_0$ in terms of the difference between the gradients $\nabla\cQ_p$ and $\nabla J_p$, and thus relate it back to the difference between $\theta$ and $\theta_0$ from Lemma \ref{lemma:taylor_exp_error}. In particular,} for any $p\geq 0$, we can decompose $\theta_{p+1}-\theta_0$ as follows:
	\#
	&\theta_{p+1}-\theta_0=\theta_{p}-\theta_0+\alpha  g_{p}=\theta_{p}-\theta_0+\alpha \nabla \cQ_p+\alpha  (g_{p}-\nabla\cQ_p+\nabla J_p-\nabla J_p)\notag\\
%	&\quad=\theta_{p}-\theta_0+\alpha \cdot[ \nabla J_0+\cH_0\cdot(\theta-\theta_0)]+\alpha  (\nabla J_p-\nabla\cQ_p+g_{p}-\nabla J_p)\notag\\
	&\quad=(\Ib+\alpha\cH_0)(\theta_{p}-\theta_0)+\alpha  (\nabla J_p-\nabla\cQ_p+g_{p}-\nabla J_p+\nabla J_0)\notag\\
	&\quad=\underbrace{(\Ib+\alpha\cH_0)^p(\theta_{1}-\theta_0)}_{\ub_p}+\alpha \cdot  \bigg[\underbrace{\sum_{q=1}^p(\Ib+\alpha\cH_0)^{p-q}(\nabla J_q-\nabla\cQ_q)}_{\bdelta_p}\notag\\
	&\qquad+\underbrace{\sum_{q=1}^p(\Ib+\alpha\cH_0)^{p-q}\nabla J_0}_{\db_p}+\underbrace{\sum_{q=1}^p(\Ib+\alpha\cH_0)^{p-q}(g_{q}-\nabla J_q)\bigg]}_{\bxi_p},\label{equ:_u_p_delta_p_def}
	\#
	where $\Ib$ is the identity matrix, $\ub_p$, $\bdelta_p$, $\db_p$, and $\bxi_p$ are defined as above, {and recall that $\cH_0=\nabla^2 J(\theta_0)$ denotes the Hessian matrix evaluated at   $\theta_0$ as defined in \eqref{equ:obj_grad_hess_shorthand}.} {The first equality  uses  the update from $\theta_{p}$ to $\theta_{p+1}$, and the second one adds and subtracts $\nabla J_q$ and $\nabla\cQ_q$. The third equality uses the definition of $\nabla \cQ_p$ from \eqref{equ:def_Q_approx}, and the last one follows by iteratively unrolling the third  equation $p$ times.} As a result, we can lower-bound the distance $\EE\|\theta_{p+1}-\theta_0\|^2$ by
	\#\label{equ:RPG_decomp_immed_1}
	&\EE\|\theta_{p+1}-\theta_0\|^2\geq 	\EE\|\ub_p\|^2+2\alpha\EE(\ub_p^\top\bdelta_p)+2\alpha\EE(\ub_p^\top\db_p)+2\alpha\EE(\ub_p^\top\bxi_p),\notag\\
	&\geq \EE\|\ub_p\|^2-2\alpha\EE(\|\ub_p\|\|\bdelta_p\|)+2\alpha\EE(\ub_p^\top)\db_p+2\alpha\EE(\ub_p^\top\bxi_p),
	\#
	where the first inequality uses the fact that $\|a+b\|^2\geq \|a\|^2+2a^\top b$, and the second one is due to the Cauchy-Schwarz inequality and the fact that $\db_p$ is deterministic given $\theta_0$.  Now we bound the terms on the right-hand side of \eqref{equ:RPG_decomp_immed_1} in the following lemmas.  
	
	\begin{lemma}[Lower-Bound on $\EE\|\ub_p\|^2$]\label{lemma:lower_bnd_up_sq}
		Suppose the conditions in Lemma \ref{lemma:saddle_escape} hold.  Then after $p\geq 1$ iterates starting from $\theta_0$, it follows that
		\$
		\EE\|\ub_p\|^2\geq \eta \beta^2\kappa^{2p}, 
		\$ 
		where $\eta$ is the lower-bound of $\EE(\vb_\theta^\top g_\theta)^2$ for any $\theta$ as defined in \eqref{equ:common_bnd_CNC}, and we also define 
		\#\label{equ:def_kappa}
		\kappa:=1+\alpha\cdot \max\{|\lambda_{\max}(\cH_0)|,0\}.
		\#
	\end{lemma}
	  \begin{proof}
	  The proof follows the proof of Lemma $11$ in \cite{daneshmand2018escaping}. 
	  Let $\vb$ denote the unit eigenvector corresponding to $\lambda_{\max}(\cH_0)$ for $\cH_0$; then by the Cauchy-Schwarz inequality, 
	   $\EE\|\ub_p\|^2=\EE(\|\vb^\top\|^2\|\ub_p\|^2)\geq \EE(\vb^\top \ub_p)^2$. By definition of   $\kappa$ in \eqref{equ:def_kappa} and the fact that $\vb$ is one of the eigenvector corresponding to $\lambda_{\max}(\cH_0)$, we have
	   \$
	   \vb^\top (\Ib+\alpha\cH_0)=\vb^\top [\Ib+\alpha\lambda_{\max}(\cH_0)]=\vb^\top \kappa.
	   \$
	   Therefore, we have
	   \$
	   \EE\|\ub_p\|^2\geq \EE[\vb^\top (\theta_{1}-\theta_0)]^2\cdot \kappa^{2p}\geq \eta\beta^2\cdot \kappa^{2p},
	   \$
	  which completes the proof. 
	  \end{proof}
	  
	  {
	  %%%%%%%%%%%%%%%%%%%%%%%%%%%%%%%%%%%%%%%%%%%%%%%%%%%%%%%%%%%%%%%%%%%%%%%%%%%%%%%%%%%%%%%%%%%%%%%%%%%%%%%%%%%%%%%%%%%%%%%%%%%%%%%%%%%%%%%%%%%%%%%%%%%%%%%%%%%%%%%%%%%%%%%%%%%%%%%%%%%%%%%%%%%%%%%%%%%%%%%%%%%%%%%%%%%%%%%%%%%%%%%%%%%%%%%%%%%%%%%%%%%%%%%%%%%%%%%%%%%%%%%%%%%%%%%%%%%%%%%%%%%%%%%%%%%%%%%%%%%%%%%%%%
	  \begin{lemma}[Upper Bound on $\EE(\|\ub_p\|\|\bdelta_p\|)$]\label{lemma:lower_bnd_up_deltap}
	  	Suppose the conditions in Lemma \ref{lemma:saddle_escape} hold, then after $p=1,\cdots, k_{\tx{thre}}-1$ iterates starting from $\theta_0$, it follows that
	  	\$
	  	\EE(\|\ub_p\|\|\bdelta_p\|)
%	  	\leq \frac{\rho\beta^3 k_{\tx{thre}}\ell^3\cdot \kappa^{2p}}{\alpha\lambda}+\frac{\rho\beta k_{\tx{thre}}\alpha^2 \ell^3\cdot\kappa^{2p}}{(\alpha\lambda)^2}. 
		\leq \big(4\ell\alpha^2\cdot \ell_g^2+4\ell\alpha{J_{\tx{thre}}}+{2L\ell^3\alpha^3k_{\tx{thre}}}+{2L\alpha\beta^2\ell^3}\big)\cdot \rho\beta\cdot\frac{\kappa^{2p}}{(\alpha\lambda)^2}+2\rho\beta^3\ell^3\cdot\frac{\kappa^{2p}}{\alpha\lambda}.
	  	\$
	  \end{lemma}
	  \begin{proof}
	  By definition of $\ub_p$ and $\bdelta_p$ in \eqref{equ:_u_p_delta_p_def},  we have
	  \#
	  &\EE(\|\ub_p\|\|\bdelta_p\|)=\EE\bigg[\big\|(\Ib+\alpha\cH_0)^p(\theta_{1}-\theta_0)\big\|\cdot \bigg\|\sum_{q=1}^p(\Ib+\alpha\cH_0)^{p-q}(\nabla J_q-\nabla\cQ_q)\bigg\|\bigg]\notag\\
	  &\quad\leq \kappa^p\beta\cdot \EE  \Big(\|g_0\|\cdot \frac{\rho}{2}\cdot\sum_{q=1}^p\kappa^{p-q}\big\|\theta_q-\theta_0\big\|^2\Big)\leq 	 \frac{\kappa^p\beta \ell\rho}{2}\cdot \sum_{q=1}^p\kappa^{p-q}\cdot\EE  \big\|\theta_q-\theta_0\big\|^2,\label{equ:upp_bnd_u_delta_immed_1}
	  \#
	  where the first inequality follows from the fact that $\|\Ib+\alpha\cH_0\|\leq \kappa$ and Lemma  \ref{lemma:taylor_exp_error} that $\|\nabla J(\theta)-\nabla \cQ(\theta)\|\leq {\rho}/{2}\cdot\|\theta-\theta_0\|^2$, and the second inequality uses the almost sure boundedness that  $\|g_0\|\leq \ell$. 
	  
	 Moreover, by Lemma \ref{lemma:distance_bnd}, we can substitute the upper-bound of $\EE  \big\|\theta_q-\theta_0\big\|^2$ in \eqref{equ:uppbnd_res_2}, and further bound the right-hand side of \eqref{equ:upp_bnd_u_delta_immed_1} as
	 \#
	  &\EE(\|\ub_p\|\|\bdelta_p\|)\leq 	 \frac{\kappa^p\beta \ell\rho}{2}\cdot \sum_{q=1}^p\kappa^{p-q}\cdot\big[4\alpha^2\ell_g^2+4\alpha J_{\tx{thre}}+2{L\alpha(\ell \beta)^2}+2{L\ell^2\alpha^3 k_{\tx{thre}}}\big]\cdot q+2\beta^2\ell^2,\notag\\
	  &\quad\leq \big[4\alpha^2\ell_g^2+4\alpha J_{\tx{thre}}+2{L\alpha(\ell \beta)^2}+2{L\ell^2\alpha^3 k_{\tx{thre}}}\big]\cdot \rho\beta\ell\cdot 
	  \frac{\kappa^{2p}}{(\alpha\lambda)^2}+2\beta^3\ell^3 \rho \cdot 
	  \frac{\kappa^{2p}}{\alpha\lambda},
	  \label{equ:upp_bnd_u_delta_immed_2}
	  \#
	  where the second  inequality uses the fact that
	  \$
	  \sum_{q=1}^p\kappa^{p-q}\leq \frac{2\kappa^p}{\alpha\lambda},\quad \quad
		\sum_{q=1}^p\kappa^{p-q}q\leq \frac{2\kappa^p}{(\alpha\lambda)^2}, 
	  \$ 
	  with   $\lambda:=\max\{|\lambda_{\max}(\cH_0)|,0\}$. 
	  This gives the formula in the lemma and completes the proof. 
  \end{proof} 
  }
  
  \begin{lemma}[Lower-Bound on $\EE(\ub_p^\top)\db_p$]\label{lemma:lower_bnd_up_dp}
  	Suppose the conditions in Lemma \ref{lemma:saddle_escape} hold, then after $p=1,\cdots, k_{\tx{thre}}-1$ iterates starting from $\theta_0$, it follows that
  	\$
  	\EE(\ub_p^\top)\db_p\geq 0.
  	\$
  \end{lemma}
  \begin{proof}
  By definition of $\ub_p$ in \eqref{equ:_u_p_delta_p_def}, it follows that
  \$
  \EE(\ub_p)=(\Ib+\alpha\cH_0)^p\EE(\theta_{1}-\theta_0)=\beta (\Ib+\alpha\cH_0)^p \nabla J_0.
  \$
  By choosing $\alpha\leq 1/L$, we have $\Ib+\alpha\cH_0\succeq 0$, which further yields
  \$
  \EE(\ub_p^\top)\db_p=\beta (\nabla J_0)^\top (\Ib+\alpha\cH_0)^p \sum_{q=1}^p(\Ib+\alpha\cH_0)^{p-q}\nabla J_0=\beta (\nabla J_0)^\top \sum_{q=1}^p(\Ib+\alpha\cH_0)^{2p-q}\nabla J_0\geq 0,
  \$
  which completes the proof.
  \end{proof}
  
  Moreover, due to  unbiasedness of $g_q$, we have 
  \$
  \EE\big(\bxi_p\biggiven \theta_0,\cdots,\theta_p\big)=0.
  \$
  Thus, 
  \#\label{equ:lower_bnd_u_xi_immed_1}
\EE\big(\ub_p^\top\bxi_p\big)=\EE_{\theta_0,\cdots,\theta_p}\big[\EE\big(\ub_p^\top\bxi_p\biggiven \theta_0,\cdots,\theta_p\big)\big]=\EE_{\theta_0,\cdots,\theta_p}\big[\ub_p^\top\EE\big(\bxi_p\biggiven \theta_0,\cdots,\theta_p\big)\big]=0,
  \#
  where the last equation is due to the fact that $\ub_p$ is $\sigma(\theta_0,\cdots,\theta_p)$-measurable.

  Now we are ready to present the lower-bound on the distance $\EE\|\theta_{p+1}-\theta_0\|^2$ using \eqref{equ:RPG_decomp_immed_1}. In particular, we combine the results of  Lemma \ref{lemma:lower_bnd_up_sq}, Lemma \ref{lemma:lower_bnd_up_deltap}, Lemma \ref{lemma:lower_bnd_up_dp}, and \eqref{equ:lower_bnd_u_xi_immed_1}, and arrive at the following lower-bound 
  
  \#\label{equ:RPG_decomp_immed_2}
  &\EE\|\theta_{p+1}-\theta_0\|^2
	\geq\eta \beta^2\kappa^{2p}-2\alpha\cdot\Big[\big(4\ell\alpha^2\cdot \ell_g^2+4\ell\alpha{J_{\tx{thre}}}+{2L\ell^3\alpha^3k_{\tx{thre}}}+{2L\alpha\beta^2\ell^3}\big)\notag\\
	&\qquad\qquad\qquad\qquad\cdot \rho\beta\cdot\frac{\kappa^{2p}}{(\alpha\lambda)^2}+2\rho\beta^3\ell^3\cdot\frac{\kappa^{2p}}{\alpha\lambda}\Big]\notag\\
	&\quad=\Big(\eta\beta-\frac{8\ell\alpha\ell_g^2\rho}{\lambda^2}-\frac{8\ell J_{\tx{thre}}\rho}{\lambda^2}-\frac{4L\ell^3\alpha^2k_{\tx{thre}}\rho}{\lambda^2}-\frac{4L\beta^2\ell^3\rho}{\lambda^2}-\frac{4\beta^2\ell^3\rho}{\lambda}\Big)\cdot \beta \kappa^{2p}.
%	\cdot\Big[\frac{\rho\beta^3 k_{\tx{thre}}\ell^3\cdot \kappa^{2p}}{\alpha\lambda}+\frac{\rho\beta k_{\tx{thre}}\alpha^2 \ell^3\cdot\kappa^{2p}}{(\alpha\lambda)^2}\Big]\notag\\
%	& =\Big(\eta\beta-\frac{2\rho\beta^2 k_{\tx{thre}}\ell^3}{\lambda}-\frac{2\alpha\rho k_{\tx{thre}}  \ell^3}{\lambda^2}\Big)\cdot \beta \kappa^{2p}.
  \#
  
  To establish contradiction, we need to show that the lower-bound on the distance $\EE(\|\theta_{p+1}-\theta_0\|^2)$ in  \eqref{equ:RPG_decomp_immed_2} is greater than the upper bound in Lemma \ref{lemma:distance_bnd}. In particular, we may choose parameters as in {Table \ref{table:parameters}} such that the terms in the bracket on the right-hand side of \eqref{equ:RPG_decomp_immed_2} are  greater than $\eta\beta/6$. To this end, we let 
  \$
&\frac{8\ell\alpha\ell_g^2\rho}{\lambda^2}\leq \eta\beta/6,\quad 
  \frac{8\ell J_{\tx{thre}}\rho}{\lambda^2}\leq \eta\beta/6,\quad \frac{4L\ell^3\alpha^2k_{\tx{thre}}\rho}{\lambda^2}\leq \eta\beta/6,\\ &\qquad\qquad\quad\frac{4L\beta^2\ell^3\rho}{\lambda^2}\leq \eta\beta/6,\quad \frac{4\beta^2\ell^3\rho}{\lambda}\leq \eta\beta/6,
  \$ 
  which require  
  \#
  &\qquad\qquad\qquad\quad\quad\beta\leq \eta\lambda/(24\ell^3\rho),\quad \beta\leq \eta\lambda^2/(24L\ell^3\rho),\label{equ:contra_immed_1}\\	
	&J_{\tx{thre}}\leq {\eta\beta\lambda^2}/{(48\ell\rho)},\quad \alpha\leq \eta\beta\lambda^2/(48\ell\ell_g^2\rho),\quad \alpha\leq [\eta\beta\lambda^2/(24L\ell^3k_{\tx{thre}}\rho)]^{1/2}.\label{equ:contra_immed_2}
  \#
  Note that the choice of $\alpha$ depends on $k_{\tx{thre}}$, which is determined as follows. Specifically,  
  we need to choose a large enough $k_{\tx{thre}}$, such that   the following contradiction holds
  \$
  \frac{\eta\beta^2}{6}\cdot \kappa^{2k_{\tx{thre}}}\geq \big[4\alpha^2\ell_g^2+4\alpha J_{\tx{thre}}+2{L\alpha(\ell \beta)^2}+2{L\ell^2\alpha^3 k_{\tx{thre}}}\big]\cdot p+2\beta^2\ell^2,
  \$
  where the right-hand side follows from \eqref{equ:uppbnd_res_2} by setting $p=k_{\tx{thre}}$.  
  To this end, we need $k_{\tx{thre}}$ to satisfy
  \#\label{equ:contra_immed_3}
  k_{\tx{thre}}\geq \frac{c}{\alpha\lambda}\cdot\log\bigg(\frac{L\ell_g}{\eta\beta\alpha\lambda}\bigg)
  \#
  where $c$ is a constant independent  of parameters $L$, $\lambda$, $\eta$, and $\rho$.  
  By substituting the lower-bound of  \eqref{equ:contra_immed_3} into  \eqref{equ:contra_immed_2}, we arrive at  
  \#\label{equ:contra_immed_4}
  \alpha\leq c'\eta\beta\lambda^3/(24L\ell^3\rho),
  \# 
  where $c'>\max\{[c\log({L\ell}/{\eta\beta\alpha\lambda})]^{-1},1\}$ is some large constant. 
   This is satisfied by the choice of stepsizes  in {Table \ref{table:parameters}}, and thus completes the proof of the lemma. 
\end{proof}

\bibliographystyle{unsrt}
%%% !!!!! comment out the following .bbl command, only for ArXiV upload!!!!!
%\input{Policy_Grad.bbl}
%%%% !!!! uncomment out the following bibliography files when compiling with Bibtex!!!!
\bibliography{RL_1,RL_2}

\begin{thebibliography}{10}

\bibitem{sutton2017reinforcement}
Richard~S Sutton, Andrew~G Barto, et~al.
\newblock {\em Reinforcement Learning: {A}n Introduction}.
\newblock 2 edition, 2017.

\bibitem{bertsekas2005dynamic}
Dimitri~P Bertsekas.
\newblock {\em Dynamic Programming and Optimal Control}, volume~1.
\newblock 2005.

\bibitem{bellman57a}
Richard~Ernest Bellman.
\newblock {\em Dynamic Programming}.
\newblock Courier Dover Publications, 1957.

\bibitem{sutton2000policy}
Richard~S Sutton, David~A McAllester, Satinder~P Singh, and Yishay Mansour.
\newblock Policy gradient methods for reinforcement learning with function
  approximation.
\newblock In {\em Advances in Neural Information Processing Systems}, pages
  1057--1063, 2000.

\bibitem{silver2014deterministic}
David Silver, Guy Lever, Nicolas Heess, Thomas Degris, Daan Wierstra, and
  Martin Riedmiller.
\newblock Deterministic policy gradient algorithms.
\newblock In {\em International Conference on Machine Learning}, pages
  379--387, 2014.

\bibitem{schulman2015trust}
John Schulman, Sergey Levine, Pieter Abbeel, Michael Jordan, and Philipp
  Moritz.
\newblock Trust region policy optimization.
\newblock In {\em International Conference on Machine Learning}, pages
  1889--1897, 2015.

\bibitem{lillicrap2015continuous}
Timothy~P Lillicrap, Jonathan~J Hunt, Alexander Pritzel, Nicolas Heess, Tom
  Erez, Yuval Tassa, David Silver, and Daan Wierstra.
\newblock Continuous control with deep reinforcement learning.
\newblock {\em arXiv preprint arXiv:1509.02971}, 2015.

\bibitem{mnih2016asynchronous}
Volodymyr Mnih, Adria~Puigdomenech Badia, Mehdi Mirza, Alex Graves, Timothy
  Lillicrap, Tim Harley, David Silver, and Koray Kavukcuoglu.
\newblock Asynchronous methods for deep reinforcement learning.
\newblock In {\em International Conference on Machine Learning}, pages
  1928--1937, 2016.

\bibitem{fazel2018global}
Maryam Fazel, Rong Ge, Sham~M Kakade, and Mehran Mesbahi.
\newblock Global convergence of policy gradient methods for linearized control
  problems.
\newblock {\em arXiv preprint arXiv:1801.05039}, 2018.

\bibitem{zhang2019policy}
Kaiqing Zhang, Zhuoran Yang, and Tamer Ba{\c{s}}ar.
\newblock Policy optimization provably converges to {N}ash equilibria in
  zero-sum linear quadratic games.
\newblock In {\em Advances in Neural Information Processing Systems}, 2019.

\bibitem{bhandari2019global}
Jalaj Bhandari and Daniel Russo.
\newblock Global optimality guarantees for policy gradient methods.
\newblock {\em arXiv preprint arXiv:1906.01786}, 2019.

\bibitem{liu2019neural}
Boyi Liu, Qi~Cai, Zhuoran Yang, and Zhaoran Wang.
\newblock Neural proximal/trust region policy optimization attains globally
  optimal policy.
\newblock {\em arXiv preprint arXiv:1906.10306}, 2019.

\bibitem{agarwal2019optimality}
Alekh Agarwal, Sham~M Kakade, Jason~D Lee, and Gaurav Mahajan.
\newblock Optimality and approximation with policy gradient methods in markov
  decision processes.
\newblock {\em arXiv preprint arXiv:1908.00261}, 2019.

\bibitem{bu2019lqr}
Jingjing Bu, Afshin Mesbahi, Maryam Fazel, and Mehran Mesbahi.
\newblock {LQR} through the lens of first order methods: Discrete-time case.
\newblock {\em arXiv preprint arXiv:1907.08921}, 2019.

\bibitem{wang2019neural}
Lingxiao Wang, Qi~Cai, Zhuoran Yang, and Zhaoran Wang.
\newblock Neural policy gradient methods: {G}lobal optimality and rates of
  convergence.
\newblock {\em arXiv preprint arXiv:1909.01150}, 2019.

\bibitem{baxter2001infinite}
Jonathan Baxter and Peter~L Bartlett.
\newblock Infinite-horizon policy-gradient estimation.
\newblock {\em Journal of Artificial Intelligence Research}, 15:319--350, 2001.

\bibitem{bartlett2011experiments}
Peter~L Bartlett, Jonathan Baxter, and Lex Weaver.
\newblock Experiments with infinite-horizon, policy-gradient estimation.
\newblock {\em arXiv preprint arXiv:1106.0666}, 2011.

\bibitem{santi2018stochastic}
Santiago Paternain.
\newblock {\em Stochastic Control Foundations of Autonomous Behavior}.
\newblock PhD thesis, University of Pennsylvania, 2018.

\bibitem{shapiro2009lectures}
Alexander Shapiro, Darinka Dentcheva, and Andrzej Ruszczy{\'n}ski.
\newblock {\em Lectures on Stochastic Programming: Modeling and Theory}.
\newblock SIAM, 2009.

\bibitem{wright1999numerical}
Stephen Wright and Jorge Nocedal.
\newblock Numerical {O}ptimization.
\newblock {\em Springer Science}, 35(67-68):7, 1999.

\bibitem{pirotta2015policy}
Matteo Pirotta, Marcello Restelli, and Luca Bascetta.
\newblock Policy gradient in {Lipschitz Markov Decision Processes}.
\newblock {\em Machine Learning}, 100(2-3):255--283, 2015.

\bibitem{bhatnagar2008incremental}
Shalabh Bhatnagar, Mohammad Ghavamzadeh, Mark Lee, and Richard~S Sutton.
\newblock Incremental natural actor-critic algorithms.
\newblock In {\em Advances in Neural Information Processing Systems}, pages
  105--112, 2008.

\bibitem{bhatnagar2009natural}
Shalabh Bhatnagar, Richard Sutton, Mohammad Ghavamzadeh, and Mark Lee.
\newblock Natural actor-critic algorithms.
\newblock {\em Automatica}, 45(11):2471--2482, 2009.

\bibitem{bhatnagar2010actor}
Shalabh Bhatnagar.
\newblock An actor--critic algorithm with function approximation for discounted
  cost constrained \relax{M}arkov \relax{D}ecision \relax{P}rocesses.
\newblock {\em Systems \& Control Letters}, 59(12):760--766, 2010.

\bibitem{chow2017risk}
Yinlam Chow, Mohammad Ghavamzadeh, Lucas Janson, and Marco Pavone.
\newblock Risk-constrained reinforcement learning with percentile risk
  criteria.
\newblock {\em Journal of Machine Learning Research}, 18:167--1, 2017.

\bibitem{borkar2008stochastic}
Vivek~S Borkar.
\newblock {\em Stochastic Approximation: \relax{A} Dynamical Systems
  Viewpoint}.
\newblock Cambridge University Press, 2008.

\bibitem{pirotta2013adaptive}
Matteo Pirotta, Marcello Restelli, and Luca Bascetta.
\newblock Adaptive step-size for policy gradient methods.
\newblock In {\em Advances in Neural Information Processing Systems}, pages
  1394--1402, 2013.

\bibitem{papini2017adaptive}
Matteo Papini, Matteo Pirotta, and Marcello Restelli.
\newblock Adaptive batch size for safe policy gradients.
\newblock In {\em Advances in Neural Information Processing Systems}, pages
  3591--3600, 2017.

\bibitem{jin2017escape}
Chi Jin, Rong Ge, Praneeth Netrapalli, Sham~M Kakade, and Michael~I Jordan.
\newblock How to escape saddle points efficiently.
\newblock In {\em International Conference on Machine Learning}, pages
  1724--1732, 2017.

\bibitem{daneshmand2018escaping}
Hadi Daneshmand, Jonas Kohler, Aurelien Lucchi, and Thomas Hofmann.
\newblock Escaping saddles with stochastic gradients.
\newblock In {\em International Conference on Machine Learning}, pages
  1155--1164, 2018.

\bibitem{kakade2002natural}
Sham~M Kakade.
\newblock A natural policy gradient.
\newblock In {\em Advances in Neural Information Processing Systems}, pages
  1531--1538, 2002.

\bibitem{greensmith2004variance}
Evan Greensmith, Peter~L Bartlett, and Jonathan Baxter.
\newblock Variance reduction techniques for gradient estimates in reinforcement
  learning.
\newblock {\em Journal of Machine Learning Research}, 5(Nov):1471--1530, 2004.

\bibitem{peters2006policy}
Jan Peters and Stefan Schaal.
\newblock Policy gradient methods for robotics.
\newblock In {\em IEEE/RSJ International Conference on Intelligent Robots and
  Systems}, pages 2219--2225, 2006.

\bibitem{konda2000actor}
Vijay~R Konda and John~N Tsitsiklis.
\newblock Actor-critic algorithms.
\newblock In {\em Advances in Neural Information Processing Systems}, pages
  1008--1014, 2000.

\bibitem{castro2010convergent}
Dotan~Di Castro and Ron Meir.
\newblock A convergent online single-time-scale actor-critic algorithm.
\newblock {\em Journal of Machine Learning Research}, 11(Jan):367--410, 2010.

\bibitem{ge2015escaping}
Rong Ge, Furong Huang, Chi Jin, and Yang Yuan.
\newblock Escaping from saddle points--online stochastic gradient for tensor
  decomposition.
\newblock In {\em Conference on Learning Theory}, pages 797--842, 2015.

\bibitem{dauphin2014identifying}
Yann~N Dauphin, Razvan Pascanu, Caglar Gulcehre, Kyunghyun Cho, Surya Ganguli,
  and Yoshua Bengio.
\newblock Identifying and attacking the saddle point problem in
  high-dimensional non-convex optimization.
\newblock In {\em Advances in Neural Information Processing Systems}, pages
  2933--2941, 2014.

\bibitem{xu2017newton}
Peng Xu, Farbod Roosta-Khorasani, and Michael~W Mahoney.
\newblock Newton-type methods for non-convex optimization under inexact
  {H}essian information.
\newblock {\em arXiv preprint arXiv:1708.07164}, 2017.

\bibitem{konda1999actor}
Vijaymohan~R Konda and Vivek~S Borkar.
\newblock Actor-critic--type learning algorithms for {Markov Decision
  Processes}.
\newblock {\em SIAM Journal on Control and Optimization}, 38(1):94--123, 1999.

\bibitem{pemantle1990nonconvergence}
Robin Pemantle.
\newblock Nonconvergence to unstable points in urn models and stochastic
  approximations.
\newblock {\em The Annals of Probability}, 18(2):698--712, 1990.

\bibitem{zhang2018fully}
Kaiqing Zhang, Zhuoran Yang, Han Liu, Tong Zhang, and Tamer Ba{\c{s}}ar.
\newblock Fully decentralized multi-agent reinforcement learning with networked
  agents.
\newblock In {\em International Conference on Machine Learning}, pages
  5872--5881, 2018.

\bibitem{zhang18cdc}
Kaiqing Zhang, Zhuoran Yang, and Tamer Ba{\c{s}}ar.
\newblock Networked multi-agent reinforcement learning in continuous spaces.
\newblock In {\em Proceedings of IEEE Conference on Decision and Control},
  pages 5872--5881, 2018.

\bibitem{papini2018stochastic}
Matteo Papini, Damiano Binaghi, Giuseppe Canonaco, Matteo Pirotta, and Marcello
  Restelli.
\newblock Stochastic variance-reduced policy gradient.
\newblock In {\em International Conference on Machine Learning}, pages
  4026--4035, 2018.

\bibitem{papini2019smoothing}
Matteo Papini, Matteo Pirotta, and Marcello Restelli.
\newblock Smoothing policies and safe policy gradients.
\newblock {\em arXiv preprint arXiv:1905.03231}, 2019.

\bibitem{doya2000reinforcement}
Kenji Doya.
\newblock Reinforcement learning in continuous time and space.
\newblock {\em Neural Computation}, 12(1):219--245, 2000.

\bibitem{dann2014policy}
Christoph Dann, Gerhard Neumann, Jan Peters, et~al.
\newblock Policy evaluation with temporal differences: A survey and comparison.
\newblock {\em Journal of Machine Learning Research}, 15:809--883, 2014.

\bibitem{robbins1985convergence}
Herbert Robbins and David Siegmund.
\newblock A convergence theorem for non-negative almost supermartingales and
  some applications.
\newblock In {\em Herbert Robbins Selected Papers}, pages 111--135. Springer,
  1985.

\bibitem{dalal2017finite}
Gal Dalal, Balazs Szorenyi, Gugan Thoppe, and Shie Mannor.
\newblock Finite sample analysis of two-timescale stochastic approximation with
  applications to reinforcement learning.
\newblock {\em arXiv preprint arXiv:1703.05376}, 2017.

\bibitem{yang18cdc}
Zhuoran Yang, Kaiqing Zhang, Mingyi Hong, and Tamer Ba\c{s}ar.
\newblock A finite sample analysis of the actor-critic algorithm.
\newblock In {\em Proceedings of IEEE Conference on Decision and Control},
  pages 5872--5881, 2018.

\bibitem{ghadimi2013stochastic}
Saeed Ghadimi and Guanghui Lan.
\newblock Stochastic first-and zeroth-order methods for nonconvex stochastic
  programming.
\newblock {\em SIAM Journal on Optimization}, 23(4):2341--2368, 2013.

\bibitem{nesterov2006cubic}
Yurii Nesterov and Boris~T Polyak.
\newblock Cubic regularization of newton method and its global performance.
\newblock {\em Mathematical Programming}, 108(1):177--205, 2006.

\bibitem{zhang2017hitting}
Yuchen Zhang, Percy Liang, and Moses Charikar.
\newblock A hitting time analysis of stochastic gradient {Langevin} dynamics.
\newblock In {\em Conference on Learning Theory}, pages 765--775, 2017.

\bibitem{Brockman2016open}
Greg Brockman, Vicki Cheung, Ludwig Pettersson, Jonas Schneider, Schulman John,
  Tang Jie, and Zaremba Wojciech.
\newblock Openai gym.
\newblock {\em arXiv preprint arXiv:1606.01540}, 2016.

\bibitem{williams1992simple}
Ronald~J Williams.
\newblock Simple statistical gradient-following algorithms for connectionist
  reinforcement learning.
\newblock {\em Machine Learning}, 8(3-4):229--256, 1992.

\bibitem{yeh2006real}
James Yeh.
\newblock {\em Real Analysis: Theory of Measure and Integration Second
  Edition}.
\newblock World Scientific Publishing Company, 2006.

\bibitem{bartle2014elements}
Robert~G Bartle.
\newblock {\em The Elements of Integration and Lebesgue Measure}.
\newblock John Wiley \& Sons, 2014.

\bibitem{bellman1954theory}
Richard Bellman.
\newblock The theory of dynamic programming.
\newblock Technical report, RAND Corp Santa Monica CA, 1954.

\bibitem{furmston2016approximate}
Thomas Furmston, Guy Lever, and David Barber.
\newblock Approximate {Newton} methods for policy search in {M}arkov {D}ecision
  {P}rocesses.
\newblock {\em The Journal of Machine Learning Research}, 17(1):8055--8105,
  2016.

\bibitem{jin2017accelerated}
Chi Jin, Praneeth Netrapalli, and Michael~I Jordan.
\newblock Accelerated gradient descent escapes saddle points faster than
  gradient descent.
\newblock {\em arXiv preprint arXiv:1711.10456}, 2017.

\bibitem{nesterov2013introductory}
Yurii Nesterov.
\newblock {\em Introductory Lectures on Convex Optimization: A Basic Course},
  volume~87.
\newblock Springer Science \& Business Media, 2013.

\end{thebibliography}

\end{document}